\documentclass[10pt,reqno]{article}
\usepackage[a4paper]{anysize}\marginsize{3cm}{3cm}{2cm}{2cm}
\pdfpagewidth=\paperwidth \pdfpageheight=\paperheight
\usepackage[T1]{fontenc}
\usepackage[english]{babel}
\usepackage{amssymb,amsmath,amsthm}
\usepackage{mathtools}
\usepackage{verbatim}
\usepackage{relsize}
\usepackage[colorlinks=true, urlcolor=blue, linkcolor=blue, citecolor=blue]{hyperref}
\usepackage[percent]{overpic}
\usepackage{comment}
\usepackage{algorithm}
\usepackage{algpseudocode}
\usepackage{xcolor}
\definecolor{OliveGreen}{rgb}{0,0.6,0}
\usepackage{subcaption}
\usepackage{enumitem}
\newlist{CI}{enumerate}{1}
\setlist[CI]{label=(C\arabic*)}
\usepackage{pgfplots}
\pgfplotsset{compat=1.15}
\usetikzlibrary{arrows}
\usepackage{authblk}

\renewcommand{\arraystretch}{1.3}

\numberwithin{equation}{section}
\theoremstyle{plain}
\newtheorem*{theorem*}{Theorem}
\newtheorem{theorem}{Theorem}
\numberwithin{theorem}{section}
\newtheorem{proposition}[theorem]{Proposition}
\newtheorem{lemma}[theorem]{Lemma}
\newtheorem{corollary}[theorem]{Corollary}
\newtheorem{conjecture}[theorem]{Conjecture}

\theoremstyle{definition}
\newtheorem{definition}[theorem]{Definition}
\newtheorem{remark}[theorem]{Remark}
\newtheorem{example}[theorem]{Example}

\newcommand{\R}{\mathbb{R}}

\newcommand{\sC}{\mathcal{C}}
\newcommand{\sD}{\mathcal{D}}
\newcommand{\sS}{\mathcal{S}}
\newcommand{\sQ}{\mathcal{Q}}
\newcommand{\sX}{\mathcal{X}}
\newcommand{\sF}{\mathcal{F}}
\newcommand{\sP}{\mathcal{P}}
\newcommand{\sT}{\mathcal{T}}
\newcommand{\sI}{\mathcal{I}}
\newcommand{\indep}{\rotatebox[origin=c]{90}{$\models$}}
\newcommand{\theproperty}{running intersection property}

\DeclareMathOperator{\dist}{dist}
\DeclareMathOperator{\conv}{conv}
\DeclareMathOperator{\codim}{codim}
\DeclareMathOperator{\rank}{rank}
\DeclareMathOperator{\supp}{supp}

\newcommand{\red}[1]{{\color{red}#1}}
\newcommand{\blue}[1]{{\color{blue}#1}}

\newcommand\blfootnote[1]{%
  \begingroup
  \renewcommand\thefootnote{}\footnote{#1}%
  \addtocounter{footnote}{-1}%
  \endgroup
}

\date{}

\title{Log-concave density estimation\\in undirected graphical models}

\author[1]{Kaie Kubjas}
\author[1]{Olga Kuznetsova}
\author[2]{Elina Robeva}
\author[2]{Pardis Semnani}
\author[1]{Luca Sodomaco}
\affil[1]{Department of Mathematics and Systems Analysis, Aalto University, Finland}
\affil[2]{Department of Mathematics, University of British Columbia, Canada}

\begin{document}

\maketitle

\begin{abstract}
We study the problem of maximum likelihood estimation of densities that are log-concave and lie in the graphical model corresponding to a given undirected graph $G$. We show that the maximum likelihood estimate (MLE) is the product of the exponentials of several tent functions, one for each maximal clique of $G$. While the set of log-concave densities in a graphical model is infinite-dimensional, our results imply that the MLE can be found by solving a finite-dimensional convex optimization problem. We provide an implementation and a few examples.
Furthermore, we show that the MLE exists and is unique with probability 1 as long as the number of sample points is larger than the size of the largest clique of $G$ when $G$ is chordal.
We show that the MLE is consistent when the graph $G$ is a disjoint union of cliques.
Finally, we discuss the conditions under which a log-concave density in the graphical model of $G$ has a log-concave factorization according to $G$.
\end{abstract}

\blfootnote{{\em E-mail addresses.} {\tt kaie.kubjas@aalto.fi}, {\tt olga.kuznetsova@aalto.fi}, {\tt erobeva@math.ubc.ca},\\{\tt psemnani@math.ubc.ca}, {\tt luca.sodomaco@aalto.fi}}

\section{Introduction}

In this paper we study the problem of maximum likelihood estimation of a log-concave density that lies in a given undirected graphical model. Models of such densities give a non-parametric generalization of the well-studied Gaussian graphical models~\cite{lauritzen1996graphical}. We show that the maximum likelihood estimate admits a nice form given by tent functions, and we compute it via a finite-dimensional convex program. Furthermore, we prove existence and uniqueness of the maximum likelihood estimate for chordal graphs.

Shape-constrained density estimation is a type of non-parametric density estimation.
Classes of shape-constrained densities that have previously been studied include  non-increasing~\cite{grenander_theory_1956}, convex~\cite{groeneboom2001estimation}, $k$-monotone~\cite{balabdaoui2007estimation}, $s$-concave~\cite{doss2016global} and quasi-concave densities~\cite{koenker2010quasi}. Log-concave density estimation was first studied in~\cite{walther2002detecting} in dimension 1. The class of log-concave distributions includes many important parametric distributions, has attractive statistical properties and their MLE is known to exist and can be computed with established optimization algorithms without the need to choose a smoothing parameter. As a result, log-concavity has since gained popularity in applications~\cite{bagnoli2005log}. A concise description of the log-concave maximum likelihood estimator in higher dimensions was given in~\cite{cule2010maximum}.
  
In this paper, we consider log-concave density estimation in undirected graphical models.
A graphical model is a statistical model where a graph is used to encode conditional independence relations between a few random variables. Graphical models have origins in various areas of science: statistical mechanics, genetics, graph theory, economics and social sciences. Their attractiveness is due to a number of factors including ease of visualisation, modularity, and compatibility with data structures of modern computers~\cite[Chapter 1.1]{lauritzen1996graphical}.
Gaussian graphical models are a prominent class of models that has received attention both from a theoretical and applications point of view. We here provide a non-parametric generalization of this class.
We refer the reader to~\cite{lauritzen1996graphical} for the theory of graphical models and~\cite{maathuis2018handbook} for the latest developments in the field.

Let $G=(V,E)$ be an undirected graph with vertex set $V=[d]\coloneqq\{1,\ldots,d\}$ and edge set $E$ which contains 2-element subsets of $V$. A subset $C\subseteq V$ such that $\{i,j\}\in E$ for all distinct $i,j\in C$ is called a {\em clique} of $G$.
We denote the set of maximal cliques of $G$ by $\sC(G)$.

To each vertex $i\in V$ we associate a random variable $X_i$. Let the random vector $X = (X_1,\ldots, X_d)\in\R^d$ have a distribution function that admits a probability density $f_0$ on $\R^d$ with respect to the Lebesgue measure $\mu$.
The probability density $f_0$ of $X$ is said to \emph{factorize according to $G$} if a.e. 
\begin{equation}\label{eq:factorization}
    f_0(x)=\frac{1}{Z}\prod_{C \in \sC(G)}\psi_C(x_C)\,,\quad x_C=(x_i\colon i\in C),
\end{equation}
where $\psi_C\ge 0$. 
The factors $\psi_C$ are called the \emph{clique potentials} and $Z$ is a normalizing constant. The clique potentials are not unique and are not required to integrate to 1.
Given an i.i.d. sample $\sX=\{X^{(1)},\ldots,X^{(n)}\}\subseteq\R^d$ from a distribution with unknown density $f_0$, our task is to find the best estimate $f_0$ in the class of densities that factorize according to the graph $G$. To do this we maximize the {\em log-likelihood function}
\begin{equation}\label{eq: log-likelihood}
\ell(f,\sX) \coloneqq \sum_{i=1}^n w_i \log f(X^{(i)})\,.
\end{equation}
Here the weights $w_1,\ldots,w_n$ are positive and add up to 1. Often they are taken to equal $1/n$, but having general weights allows us to give different levels of importance to the different samples. If no assumptions other than the factorization~\eqref{eq:factorization} are imposed on the unknown density, then the likelihood function is unbounded, and a maximum likelihood estimate does not exist~\cite{maathuis2018handbook}.

Rather than imposing a parametric assumption on $\sF_G$, such as Gaussianity, here we impose a less restrictive non-parametric assumption. 

\begin{definition}\label{def: log-concave factorization}
A probability density $f$ on $\R^d$ has a {\em log-concave factorization according to a graph $G = (V, E)$} with $|V| = d$ if we can write
\begin{equation*}\label{eq: factorization}
    f(x) = \frac{1}{Z}\prod_{C\in\sC(G)}\psi_C(x_C)\,,  
\end{equation*}
where each of the clique potentials $\psi_C$ is nonnegative and log-concave, i.e., its logarithm $\log\psi_C$ is concave. We denote by $\sF_G$ the set of upper semi-continuous densities on $\R^d$ that admit a log-concave factorization according to the graph $G$. 
\end{definition}

The upper semi-continuity condition helps us avoid considering densities that differ on a set of Lebesgue measure zero.
While every function that has a log-concave factorization according to $G$ is log-concave and factorizes according to $G$, the converse does not always hold. We devote Section~\ref{sec:decomposition-of-convex-functions} to this question.

Our aim now is to maximize the log-likelihood of the observed sample $\sX$ over all densities $f$ in the family $\sF_G$:
\begin{equation}\label{eq: opt_problem}
\text{maximize}_{f\in\sF_G}\,\ell(f,\sX)\,.
\end{equation}

\textbf{Results.} The {\em maximum likelihood estimate (MLE)} is defined as a maximizer $\hat{f}$ of the optimization problem~\eqref{eq: opt_problem}. The support of the MLE $\hat{f}$ always contains the convex hull of the sample $\sX$, but, depending on the graph, it can be larger. We give a complete characterization of it in Proposition~\ref{prop: support}. While the optimization problem~\eqref{eq: opt_problem} is over the family $\sF_G$ which is infinite-dimensional, it turns out that the MLE $\hat{f}$ has a special form in terms of tent functions (piecewise linear concave functions), so we do not need to optimize over all of $\sF_G$. 
We show in Theorem~\ref{thm: shape of MLE tent} that the MLE is the exponential of the sum of several tent functions, one for each maximal clique $C\in\sC(G)$, and can be computed via the finite-dimensional problem~\eqref{eq:new_opt_problem}. Each of these tent functions only uses variables from one of the maximal cliques $C\in\sC(G)$, and has tent poles located at the projection of the sample $\sX$ onto the coordinates from $C$. The sum of these tent functions is also a tent function, but its tent pole locations are not a priori clear from the data $\sX$.
As a consequence of Proposition~\ref{prop: subdivision_intersection}, we show how to compute these locations, and we use this result in the implementation of the solution to the optimization problem~\eqref{eq:new_opt_problem}, which we describe in detail in Section~\ref{sec:optimization}.

We show in Theorem~\ref{thm: existence and uniqueness MLE} that, given an undirected chordal graph $G$, the MLE $\hat{f}$ exists and is unique with probability 1 as long as the number of sample points $n$ is greater than the size of the largest clique of $G$. We conjecture that this result can be extended to any undirected graph.   
When the graph $G$ is a disjoint union of cliques, we show in Theorem~\ref{thm:consistency-for-connected-components} that the MLE is also consistent.

\smallskip
\textbf{Related work.} Maximum likelihood estimation has been studied for other families of graphical models. For example, one could impose that the interaction terms $\psi_{C}(x_C)$ belong to tensor product spaces with additional smoothness properties~\cite{lin2000tensor}. Another approach is to restrict the feasible set to the set of densities of the form 
\begin{equation}\label{eq: factorization simplified}
    f(x) = \frac{1}{Z}\prod_{\{i,j\}\in E}\psi_{i,j}(x_i,x_j)\prod_{i\in V}\psi_{i}(x_i)\,,  
\end{equation}
that is, omit higher-order interaction terms. When the underlying graph is acyclic, equation~\eqref{eq: factorization simplified} simplifies to the estimation of bivariate and univariate marginals, which is done using kernel density estimation~\cite{bach2003beyond,  liu2011forest}. In another direction of work, parametric assumptions on conditional densities~\cite{chen2015selection} or joint densities~\cite{liu2012high} are added to move the problem into a semi-parametric setting.
Furthermore, the estimation of the normalizing constant $Z$ can be avoided by using the Fisher-Hyv\"arinen score instead of Kullback–Leibler divergence~\cite{hyvarinen2005estimation}.

A prominent example of parametric log-concave graphical models is the case of Gaussian graphical model estimation~\cite[Chapter 9]{maathuis2018handbook}. In particular, the sample size required for the existence of the MLE has been studied in~\cite{buhl1993existence, MR3014306}. Here, chordality plays an important role, a theme that will be echoed in Section~\ref{sec:existence}.

When the underlying graph $G$ is complete, the problem~\eqref{eq: opt_problem} is equivalent to classical log-concave maximum likelihood estimation. The existence and uniqueness of this estimator in higher dimensions was shown in~\cite{cule2010maximum} and studied in~\cite{robeva2019geometry} from a geometric point of view. We refer the reader to~\cite{samworth2018recent} for a review of the theoretical properties of log-concave densities and their relevance in statistical problems,~\cite{saumard2014log} for analytical properties, and~\cite{walther2009inference} for applications in modeling and inference. While the class of log-concave densities is too large from a statistical point of view \cite{kim2016global}, subclasses of log-concave densities have been considered in the past. For example, independent component analysis for log-concave densities was studied in~\cite{samworth2012independent}, log-concave densities arising from a piecewise linear function with a fixed number of linear pieces were studied in~\cite{kim2018adaptation}, and log-concave densities that are also totally positive were studied in~\cite{robeva2021maximum}.

\smallskip
\textbf{Organization.} 
The rest of the paper is organized as follows. We begin with Section~\ref{sec:support}, in which we describe the support of the MLE. In Section~\ref{sec:form} we show that the MLE is a product of exponentials of tent functions, and can therefore be computed via a finite-dimensional optimization problem. In Section~\ref{sec:existence} we show existence and uniqueness of the MLE for chordal graphs. In Section~\ref{sec:location-of-tentpoles} we show how to characterize the sum of the tent functions that comprise the MLE. In Section~\ref{sec:optimization} we describe an algorithm that computes the MLE, and provide examples. In Section~\ref{sec: disjoint union} we discuss the particular case when the graph is given by a set of disjoint cliques, which includes the setting when the set of edges is empty and the variables $X_1,\ldots, X_d$ are independent. In Section~\ref{sec:decomposition-of-convex-functions} we provide some conditions on the graph $G$ and on the density under which being log-concave and factorizing according to $G$ results in 
having a log-concave factorization according to $G$.

\smallskip
\textbf{Notation.} Given an integer $k\ge 1$, we denote by $[k]$ the set $\{1,\ldots,k\}$.
If not specified, we always consider graphs $G=(V,E)$ where $V=[d]$ for some integer $d\ge 1$. We call $\sC(G)$ the set of maximal cliques of $G$. Given a nonempty subset $S\subseteq[d]$, we denote by $\pi_S$ the projection of $\R^d$ onto the coordinates in $S$, and by $\R^S$ the ambient space of such projection. When $S=V$ we simply write $\R^d$ instead of $\R^V$. Furthermore, for any vector $x\in\R^d$, we use the shorthand $x_S$ to indicate the projection $\pi_S(x)$. More explicitly, if $S=\{s_1,\ldots,s_m\}\subseteq[d]$, we use the shorthand $\pi_{s_1\cdots s_m}$ instead of $\pi_{\{s_1,\ldots,s_m\}}$.

We usually write $f$ for a density on $\R^V=\R^d$. Given subsets $A,B\subseteq V$, we call $f_A$ and $f_{A|B}$ the corresponding marginal and conditional densities on $\R^A$. We always denote by $\sX$ an i.i.d. sample of $n$ points $\{X^{(1)},\ldots, X^{(n)}\}$ obtained from a probability density $f$ on $\R^d$.  Given a subset $S\subseteq[d]$, we call $\sX_S$ its projection $\pi_S(\sX)=\{X_S^{(1)},\ldots, X_S^{(n)}\}$.

\section{Support of the MLE}\label{sec:support}

In this section we assume the existence of the MLE $\hat{f}$ that solves the problem \eqref{eq: opt_problem}.
We show that the support of $\hat{f}$, i.e. the set where $ \hat{f} > 0$, is a convex polytope. Note that the support of $\hat{f}$ has to be convex since $\hat{f}$ is log-concave.
If we drop the factorization assumption of Definition \ref{def: log-concave factorization} and require only log-concavity of the MLE as in~\cite{cule2010maximum}, then the MLE is supported on the convex hull $\conv(\sX)$ of the sample $\sX$. In our setting, we show that the support of $\hat{f}$ over the family $\sF_G$ is a convex polytope that contains $\conv(\sX)$ and depends on the sample $\sX$ and the graph $G$.

\begin{proposition}\label{prop: support}
Let $G = (V, E)$ be an undirected graph, and let $\sC(G)$ be the set of maximal cliques of $G$. Let $\sX=\{X^{(1)},\ldots, X^{(n)}\}$ be an i.i.d. sample from a probability density $f_0$ on $\R^d$. 
Then, if the MLE $\hat{f}$ over $\sF_G$ exists, it is supported on
\begin{equation}\label{eq: support of MLE for undirected}
\sS_{G,\sX} \coloneqq \bigcap_{C\in\sC(G)} \pi_C^{-1}\left(\pi_C(\conv(\sX))\right)\,.
\end{equation}
\end{proposition}

\begin{proof}
Let $\supp(\hat{f})$ be the support of $\hat{f}$. First, we show that $\sS_{G,\sX}\subseteq\supp(\hat{f})$.
Since $\supp(\hat{f})$ is convex and $\hat{f}(X^{(i)})\neq 0$ for all $i\in[n]$ (or else the log-likelihood $\ell(\hat{f},\sX)$ is $-\infty$), then $\conv(\sX)\subseteq\supp(\hat{f})$.
Recall that $\hat{f}$ lies in the graphical model corresponding to $G$, hence it factors as
\begin{equation}\label{eq: factorization estimator}
\hat{f}(x) = \prod_{C\in\sC(G)} \psi_C(x_C)\,.
\end{equation}
Consider a point $x\in \conv(\sX)$. Then $\hat{f}(x)$ is positive, and by \eqref{eq: factorization estimator} every factor $\psi_C(x_C)$ for all $C\in\sC(G)$ is positive as well.
Therefore, the function $\psi_C$ is positive on $\pi_C(\conv(\sX))$, for $C\in\sC(G)$.
Thus, $\hat{f}$ is positive on the intersection $\sS_{G,\sX}$, in particular $\sS_{G,\sX}\subseteq\supp(\hat{f})$.

Now suppose that $\sS_{G,\sX}\subsetneq\supp(\hat{f})$.
Consider the restriction ${\hat{f}}|_{\sS_{G,\sX}}$. It attains the same log-likelihood value $\ell(\hat{f},\sX)$ as $\hat{f}$, it is log-concave, and it has a log-concave factorization according to the graph $G$ (where the factor corresponding to $C\in\sC(G)$ is restricted to $\pi_C(\conv(\sX))$). Furthermore, ${\hat{f}}|_{\sS_{G,\sX}}$ might have integral strictly less than 1, in which case, we have to rescale it by a constant larger than 1.
This yields an even higher likelihood value, contradicting the fact that $\hat{f}$ is the MLE.
Therefore $\sS_{G,\sX}=\supp(\hat{f})$.
\end{proof}

\begin{remark}
Proposition~\ref{prop: support} holds more generally if we replace $\sF_G$ by the set of upper semi-continuous log-concave probability densities that factor according to $G$ without the requirement that the clique potentials $\phi_C$ are log-concave.
\end{remark}

\begin{example}\label{ex: shape of S_n}
Consider the graph $G$ in Figure \ref{fig:simple chordal graph}. The maximal cliques of $G$ are the two edges in $G$. Suppose we have the sample $\sX=\{(0,1,2),(1,5,4),(2,3,5)\}$. The region $\sS_{G,\sX}$ and its projections $\pi_{12}(\sS_{G,\sX})$ and $\pi_{23}(\sS_{G,\sX})$ are represented in Figure \ref{fig: Sn-xyz}. The set $\pi_{12}^{-1}(\pi_{12}(\conv(\sX)))$ is the prism based on the light blue triangle on the $(x_1,x_2)$-plane with vertices $\{(0,1),(1,5),(2,3)\}$. Analogously $\pi_{23}^{-1}(\pi_{23}(\conv(\sX)))$ is the prism based on the yellow  triangle on the $(x_2,x_3)$-plane with vertices $\{(1,2),(5,4),(3,5)\}$. The polytope $\sS_{G,\sX}$ is the region of intersection between the two prisms. It has five vertices: the three points of $\sX$ and the points $P=(2,3,3)$ and $Q=(1/2,3,5)$. 
\begin{figure}[ht]
    \begin{subfigure}[t]{.5\textwidth}
    \centering
    \vspace{-100pt}
    \includegraphics[scale=0.5]{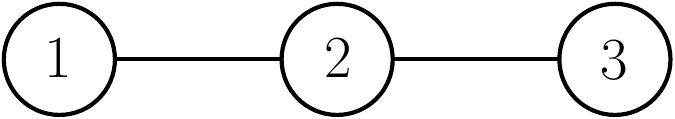}
    \caption{}
    \label{fig:simple chordal graph}
    \end{subfigure}
    \hspace{-50pt}
    \begin{subfigure}[t]{.5\textwidth}
    \centering
    \begin{overpic}[width=0.7\textwidth, tics=10]{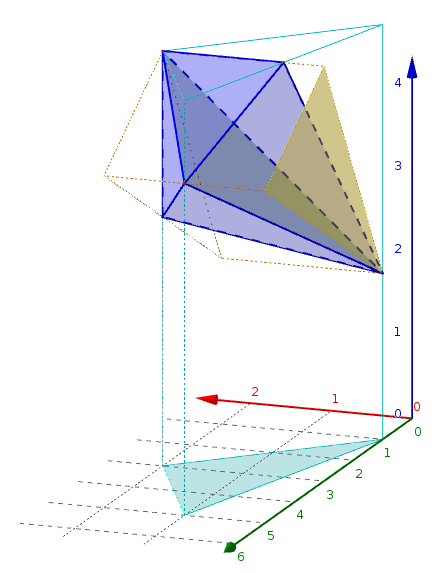}
    \put (21,92) {\red{\small{$X^{(3)}$}}}
    \put (31,63) {\red{\small{$X^{(2)}$}}}
    \put (64,48) {\red{\small{$X^{(1)}$}}}
    \put (22,60) {\small{$P$}}
    \put (49,91) {\small{$Q$}}
    \put (30,30) {\scriptsize{\red{$x_1$}}}
    \put (36,2) {\scriptsize{\textcolor{OliveGreen}{$x_2$}}}
    \put (70,92) {\scriptsize{\blue{$x_3$}}}
    \end{overpic}
    \caption{}
    \label{fig: Sn-xyz}
    \end{subfigure}%
    \caption{A simple chordal graph $G$ and the polytope $\sS_{G,\sX}$ of Example \ref{ex: shape of S_n}.}
    \label{fig: S_n}
\end{figure}
\end{example}

The next result computes the minimum number of sample points needed for the support $\sS_{G,\sX}$ of $\hat{f}$ to be full-dimensional. Its proof is located in Appendix \ref{appendix: proofs section 2,4,5}.

\begin{lemma}\label{lemma: support of undirected is full dimensional}
Let $G$ be an undirected graph. Let $\sX$ be an i.i.d. sample from a probability density $f_0$ on $\R^d$ that has full-dimensional support. 
If $n=|\sX| \ge \max_{C\in \sC(G)}|C|+1$, then the support $\sS_{G,\sX}$ defined in \eqref{eq: support of MLE for undirected} of the MLE $\hat{f}$ is a.s. full-dimensional. 
\end{lemma}

Using the factorization~\eqref{eq:factorization} of densities in the model $\sF_G$ allowed us to determine precisely what the support $\sS_{G, \sX}$ of the MLE $\hat{f}$ is. In Section~\ref{sec:existence}, it will also be useful to consider the conditional independence statements satisfied by the random vectors whose densities lie in $\sF_G$.

Recall that, by the Hammersley-Clifford Theorem~\cite[Theorem 3.9]{lauritzen1996graphical}, any density $f\in\sF_G$ satisfies the {\em global Markov property} for the graph $G$. In other words, for any three disjoint subsets of vertices $A, B, S\subseteq V$ such that $S$ separates $A$ and $B$ in the graph $G$ (i.e., every path between a vertex in $A$ and a vertex in $B$ has to go through a vertex in $S$), we must have that $X_A\indep X_B|X_S$ (in short, $A\indep B|S$) for the random vector $X$ whose density is $f$.

Now suppose that given a sample $\sX$, the density $\hat{f}$ is the MLE over $\sF_G$. Thus the convex hull $\conv(\sX)$ is contained in the support of $\hat{f}$. But what other points are in the support of $\hat{f}$? Suppose that $X_A\indep X_B|X_C$ is a global Markov statement for $G$ such that $\{A,B,S\}$ is a partition of $V$. Then, for two points $(x_A, x_B, x_S), (y_A,y_B,x_S)\in\conv(\sX)$, we have that
\begin{align*}
0 &< \hat{f}(x_A,x_B,x_S) \hat{f}(y_A,y_B,x_S) \\
&= \hat{f}(x_A|x_S) \hat{f}(x_B|x_S) \hat{f}(y_A|x_S)\hat{f}(y_B|x_S) \hat{f}(x_S)^2 \\
&= \hat{f}(x_A,y_B,x_S) \hat{f}(x_A,y_B,x_S)\,.
\end{align*}
Therefore, the points $(x_A,y_B,x_S)$ and $(y_A,x_B,x_S)$ lie in the support of $\hat{f}$ as well. Starting from $\conv(\sX)$ (which we know lies in the support of $\hat{f}$), we can try to add more points until we obtain the whole support of $\hat{f}$. This construction is useful for bounding the values of $\hat{f}$ on its support if we know its minimum and maximum over $\conv(\sX)$. We use such bounds in the proof of our existence and uniqueness Theorem~\ref{thm: existence and uniqueness MLE}.

\begin{definition}\label{def: sets Di}
Let $G$ be an undirected graph with $d$ vertices. Let $\sI_G$ be the set of global Markov statements $A\indep B|S$ for $G$ such that the collection $\{A,B,S\}$ is a partition of $V$.
Define $\sD^{(0)}_{G}\coloneqq\conv(\sX)\subseteq\R^d$. For $i\ge 1$, we iteratively define
\begin{equation}\label{eq: def sets D^i}
\sD^{(i)}_{G} \coloneqq \conv\bigcup_{A \indep B|S\in\sI_G}\{(x_A,y_B,x_S),(y_A,x_B,x_S)\mid(x_A,x_B,x_S),(y_A,y_B,x_S)\in\sD^{(i-1)}_{G}\}\,.
\end{equation}
\end{definition}

The following proposition shows that the sequence $(\sD_G^{(i)})_i$ always has a set-theoretic limit, and the limit is contained in the support $\sS_{G,\sX}$. The proof is located in Appendix \ref{appendix: proofs section 2,4,5}.

\begin{proposition}\label{prop: limit D_G i exists}
Let $G$ be an undirected graph with $d$ vertices. The set-theoretic limit $\lim_{i\to\infty}\sD_G^{(i)}$ exists and is a bounded convex set of $\R^d$. In particular $\lim_{i\to\infty}\sD_G^{(i)}=\bigcup_{i\ge 0}\sD_G^{(i)}\subseteq\sS_{G, \sX}$.
\end{proposition}

It is natural to ask whether the sequence $(\sD_G^{(i)})_i$ stabilizes after finitely many steps, and if the limit coincides with the support $\sS_{G, \sX}$. This is true for {\em chordal graphs}, which are introduced in Definition~\ref{def: chordal graph}.

\begin{proposition}\label{prop: achieve support}
Let $G$ be a chordal graph. Consider the polytope $\sS_{G,\sX}$ introduced in \eqref{eq: support of MLE for undirected} and the sets $\sD^{(i)}_{G}$ of Definition \ref{def: sets Di}. 
Then $\lim_{i\to\infty}\sD_G^{(i)}=\sD^{(t-1)}_{G}=\sS_{G,\sX}$, where $t=|\sC(G)|$.
\end{proposition}

The proof of this Proposition uses junction trees and is located in Appendix~\ref{appendix: proofs section 2,4,5}. The rough idea behind it is as follows. First, we know that $\sD_{G}^{(i)}\subseteq\sS_{G, \sX}$ for any $i$. Now, for any point $a\in\sS_{G, \sX}$, we first note that for any clique $C\in\sC(G)$ there exists a point $a^{(C)}\in \conv(\sX) = \sD^{(0)}_{G}$ such that the projections of $a$ and $a^{(C)}$ on the coordinates from $C$ agree, i.e., $a_C = a^{(C)}_C$. Using a junction tree $\sT$ for $G$, we ``peel off'' leaves from $\sT$ to show that we can find points in $\sD_{G}^{(i)}$ that agree with $a$ along more and more coordinates as we increase $i$. Finally, we see that the point $a$ itself lies in $\sD_{G}^{(t-1)}$, where $t=|\sC(G)|$. Thus, $\sD_{G}^{(t-1)} = \sS_{G, \sX}$.

If the graph $G$ is not chordal, we can consider a {\em chordal cover} $\tilde{G}$ of it (see Appendix~\ref{app:background}), and show that $\sD_G^{(i)}\subseteq \sS_{G, \sX}$ is a full-dimensional subset when $i$ is large enough.

\begin{corollary}\label{corol: enough_samples}
Let $G$ be an undirected graph and let $\tilde{G}$ be a chordal cover of $G$.
If $|\sX|=n\ge\max_{C\in \sC(\tilde{G})}|C|+1$, then $\sD^{(t-1)}_{\tilde{G}}$ (which is contained in $\sS_{G, \sX}$) is a full dimensional subset in $\R^d$ almost surely, where $t=|\sC(\tilde{G})|$.
\end{corollary}

The next proposition, whose proof is located in Appendix~\ref{appendix: proofs section 2,4,5}, provides an alternative characterization of the sets $\sD_G^{(i)}$. This allowed us to write an algorithm that computes the sets $\sD_G^{(i)}$ given a graph $G$ and a sample $\sX$. The algorithm has been implemented in \texttt{Macaulay2} \cite{GS} and its code is available at \cite{githubrepo}. 

\begin{proposition}\label{prop: equivalent def of X A ind B given S}
Let $D$ be a subset of $\R^d$ and $\{A,B,S\}$ be a partition of $[d]$. Consider the subset
\[
\sP_{A,B,S}(D)\coloneqq\{(x_A,y_B,x_S),(y_A,x_B,x_S)\mid(x_A,x_B,x_S),(y_A,y_B,x_S)\in D\}\,.
\]
Then
\[
\sP_{A,B,S}(D) = \pi_{A\cup S}^{-1}(\pi_{A\cup S}(D))\cap \pi_{B\cup S}^{-1}(\pi_{B\cup S}(D))\,.
\]
\end{proposition}

\begin{example}\label{ex: convergent sequence polytopes 4-cycle}
Computing the limit $\lim_{i\to\infty}\sD_G^{(i)}$ for non-chordal graphs $G$ is still an open problem. Consider the $4$-cycle $G=([4],\{\{1,2\},\{2,3\},\{3,4\},\{4,1\}\})$. In particular $G$ is not chordal. Let $\sX=\{(7, 2, 8, 0), (3, 7, 9, 3), (7, 9, 8, 4), (8, 0, 1, 8)\}$. Using our algorithm to compute the sets $\sD_G^{(i)}$, we obtain that $\lim_{i\to\infty}\sD_G^{(i)}=\sD_G^{(7)}=\sS_{G,\sX}$. This polytope is the convex hull of the $18$ points whose coordinates are displayed in Table \ref{table: vertices support MLE 4-cycle} in Appendix \ref{appendix: proofs section 2,4,5}.
\end{example}

Following up on Example \ref{ex: convergent sequence polytopes 4-cycle}, we observed that when considering the same graph $G$ but different samples $\sX$, the minimum index $s$ such that $\lim_{i\to\infty}\sD_G^{(i)}=\sD_G^{(s)}$ may change, and may also be arbitrarily large. In some cases we had to interrupt the computation of the sets $\sD_G^{(i)}$, nevertheless we have experimental evidence that the sequence of volumes $(\mu(\sD_G^{(i)}))_i$ converges to $\mu(\sS_{G,\sX})$. This motivates the following conjecture.

\begin{conjecture}\label{conj: limit equal to support}
Let $G$ be an undirected graph with $d$ vertices. Let $\sX\subseteq\R^d$ be an i.i.d. sample from a non-degenerate probability density $f_0$ on $\R^d$. Consider the support $\sS_{G,\sX}$ of the MLE $\hat{f}$ over $\sF_G$ and the sets $\sD^{(i)}_{G}$ of Definition \ref{def: sets Di}. Then $\lim_{i\to\infty}\sD_G^{(i)}=\sS_{G,\sX}$.
\end{conjecture}

\section{Form and computation of the MLE}\label{sec:form}

Assuming the MLE exists, in this section we show that the MLE has a specific form in terms of {\em tent functions} (see Definition~\ref{def: tent function} and Theorem~\ref{thm: shape of MLE tent}$(i)$). While the optimization problem~\eqref{eq: opt_problem} is a priori infinite-dimensional, the tent function formulation implies that we can instead use a finite-dimensional problem for finding the MLE (Theorem~\ref{thm: shape of MLE tent}$(ii)$).

\begin{definition}\label{def: tent function}
Consider a subset $X=\{x^{(1)},\ldots,x^{(n)}\}$ in $\R^d$ and a vector $y=(y_1,\ldots,y_n)$ in $\R^n$.
A function $h_{X,y}\colon\R^d\to\R\cup\{\infty\}$ is called a {\em tent function} with tent poles located at $x^{(i)}$ and heights $y_i$ if $h_{X,y}$ is the least concave function satisfying $h_{X,y}(x^{(i)})\ge y_i$ for all $i\in[n]$.
In particular $h_{X,y}$ is piecewise linear, it is greater than $-\infty$ on $\conv(X)$ and equals $-\infty$ otherwise.
\end{definition}

We now state the main result of this section.

\begin{theorem}\label{thm: shape of MLE tent}
Let $G$ be an undirected graph, and let $\sC(G)$ be its set of maximal cliques. Let $\sX=\{X^{(1)},\ldots, X^{(n)}\}$ be an i.i.d. sample from a probability density  $f_0$  that has a log-concave factorization according to $G$.
Assume that the MLE $\hat f$ over $\sF_G$ exists. Then
\begin{itemize}
\item[$(i)$] the MLE $\hat{f}$ has the form
\begin{equation}\label{eq: factorization f hat}
\hat{f}(x) = \frac{1}{Z}\prod_{C\in\sC(G)}
\exp(h_{\mathcal X_C, y_C}(x_C))\,,
\end{equation}
where the tent function $h_{\sX_C, y_C}$ is supported on the projection $\sX_C$ of $\sX$ with heights $y_C = (y^{(1)}_C,\ldots, y^{(n)}_C)$.
\item[$(ii)$] The set of unknown heights $(y_C)_{C\in\sC(G)}$ is the solution of the optimization problem
\begin{equation}\label{eq:new_opt_problem}
\max_{(y_C)_{C\in\sC(G)}}\tau((y_C)_{C\in\sC(G)})\,,
\end{equation}
where the objective function $\tau\colon\R^{n|\sC(G)|}\to\R$ is defined by
\[
\tau((y_C)_{C\in\sC(G)})\coloneqq\sum_{i=1}^n\sum_{C\in\sC(G)}w_i y_C^{(i)} - \int_{\R^d}\prod_{C\in\sC(G)}\exp(h_{\sX_C, y_C})\,d\mu\,.
\]
Furthermore, if $(\hat{y}_C)_{C\in\sC(G)}$ is the output of \eqref{eq:new_opt_problem} and $\hat{g}(x)\coloneqq\prod_{C\in\sC(G)}\exp(h_{\sX_C,\hat{y}_C}(x))$, then $\hat{g}$ is a density.
\end{itemize}
\end{theorem}

\begin{proof}
$(i)$ Consider the factorization \eqref{eq: factorization f hat}
into log-concave potentials $\widehat{\psi}_C$ for all $C\in\sC(G)$. Suppose that, for some clique $C$, there exists a full-dimensional set of points of $\R^{C}$ where the potential $\widehat{\psi}_C$ does not agree with the exponential of a tent function supported on $\sX_C$.
Then, consider the heights $y_C = (y^{(1)}_C,\ldots, y^{(n)}_C)$, where $y^{(i)}_C = \widehat{\psi}_C(X^{(i)}_C)$, and the tent function $h_{\sX_C,y_C}\colon \R^{C}\to\R$ supported on $\sX_C$ with heights $y_C$. Consider the new function
$$
\hat{f}_1(x) \coloneqq \frac{1}{Z}\exp(h_{\sX_C,y_C}(x_C))\cdot\prod_{C'\in\sC(G), C'\neq C} \widehat{\psi}_{C'}(x_{C'})\,,
$$
where we have replaced $\hat\psi_C(x_C)$ by $\exp(h_{\sX_C,y_C}(x_C))$.

Since by assumption $\exp(h_{\sX_C,y_C})$ differs from $\widehat{\psi}_C$ on a set of positive measure, then $\exp(h_{\sX_C,y_C}(x_C))$ is strictly smaller than $\widehat{\psi}_C(x_C)$ on a set of positive measure.
In turn $\hat{f}_1(x)$ is strictly smaller than $\hat{f}(x)$ on a set of positive measure, hence $Z_1\coloneqq\int_{\R^d}\hat{f}_1\,d\mu<\int_{\R^d}\hat{f}\,d\mu = 1$.
Now consider the function
$$
\frac{1}{Z_1}\hat{f}_1(x) = \frac{1}{Z_1\,Z}\exp(h_{\sX_C,y_C}(x_C))\prod_{C'\in\sC(G), C'\neq C} \widehat{\psi}_{C'}(x_{C'})\,,
$$
which is a density on $\R^d$ and factorizes according to the graph $G$.
Its log-likelihood equals
\begin{align*}
\ell\left(\frac{1}{Z_1}\hat{f}_1,\sX\right) &= \sum_{i=1}^n w_i\left[-\log(Z_1)-\log(Z) + h_{\sX_C, y_C}(X^{(i)}_C) + \sum_{C'\neq C} \hat \psi_{C'}(X^{(i)}_{C'})\right]\\
&= \sum_{i=1}^n w_i\left[-\log(Z_1)-\log(Z) + \widehat{\psi}_C(X^{(i)}_C) + \sum_{C'\neq C} \hat \psi_{C'}(X^{(i)}_{C'})\right]\\
&= -\log(Z_1)\sum_{i=1}^n w_i + \ell(\hat{f},\sX) = -\log(Z_1) + \ell(\hat{f},\sX) > \ell(\hat{f},\sX)\,,
\end{align*}
where the first equality follows by construction of $h_{\sX_C,y_C}$ and the last inequality follows since $Z_1 < 1$, and, therefore, $-\log(Z_1)>0$. 
Thus the density $\hat{f}_1/Z_1$ lies in our model, but $\ell(\hat{f}_1/Z_1,\sX)>\ell(\hat{f},\sX)$, contradicting the hypothesis on $\hat{f}$.
Therefore, each function $\widehat{\psi}_C$ equals the exponential of a tent function $\exp(h_{\sX_C, y_C})$ supported on $\sX_C$.

\medskip
$(ii)$ Let $(\hat{y}_C)_{C\in\sC(G)}$ be the output of the problem~\eqref{eq:new_opt_problem}. We show that
\[
\int_{\R^d}\prod_{C\in\sC(G)}\exp(h_{\sX_C,\hat{y}_C}(x_C))d\mu(x)\eqqcolon T = 1\,.
\]
Suppose $T\neq 1$. Pick a clique $C_0\in\sC(G)$, and consider the new heights $\{z_C\mid C\in\sC(G)\}$, where $z_C = \hat{y}_C$ if $C\neq C_0$, and $z_{C_0} = \hat{y}_{C_0} - \log(T)\mathbf 1$.
Then
\[
\int_{\R^d}\prod_{C\in\sC(G)}\exp(h_{\sX_C,z_C}(x_C))d\mu(x) = \frac{1}{T}\int_{\R^d}\prod_{C\in\sC(G)}\exp(h_{\sX_C,\hat{y}_C}(x_C))d\mu(x)= 1\,,
\]
and, in turn, 
\[
\sum_{C\in\sC(G)}\sum_{i=1}^n w_i z_C^{(i)} = \sum_{C\neq C_0}\sum_{i=1}^n w_i \hat{y}_C^{(i)}+\sum_{i=1}^n w_i(\hat{y}_{C_0}^{(i)}-\log(T)) = \sum_{C\in\sC(G)}\sum_{i=1}^n w_i \hat{y}_C^{(i)} - \log(T)\,.
\]
Therefore, the new objective function in~\eqref{eq:new_opt_problem} with heights $\{z_C\}_{C\in\sC(G)}$ is
\begin{equation}\label{eq: new obj function}
\sum_{C\in\sC(G)}\sum_{i=1}^n w_i z_C^{(i)} - \int_{\R^d}\prod_{C\in\sC(G)}\exp(h_{\sX_C,z_C}(x_C))d\mu(x) = \sum_{C\in\sC(G)}\sum_{i=1}^n w_i \hat{y}_C^{(i)} - \log(T) - 1\,.
\end{equation}
The difference between the objective function value in \eqref{eq: new obj function} and the original one in \eqref{eq:new_opt_problem} is $\Delta(T)=T-\log(T)-1$. This difference has to be non-positive since the heights $(\hat y_C)_{C\in\sC(G}$ maximize the objective function.
But note that the function $\Delta(T)$ is always nonnegative on its domain $(0,+\infty)$ and vanishes precisely at $T=1$. Hence $T=1$, since otherwise the new heights $(z_C)_{C\in\sC(G)}$ would give a higher objective value.

Therefore, the maximizers $(\hat{y}_C)_{C\in\sC(G)}$ of~\eqref{eq:new_opt_problem} always yield the integral to equal 1.
Therefore, they also maximize the log likelihood function subject to the integral of the density being equal to 1, and, thus, the heights $(\hat{y}_C)_{C\in\sC(G)}$ are optimal, and give us the MLE $\hat f = \exp(\sum_{C\in\sC(G)} h_{\sX_C, y_C}(x_C))$.
\end{proof}

\begin{remark}
Theorem \ref{thm: shape of MLE tent}$(i)$ states that the MLE $\hat f$ is the exponential of a sum of tent functions (which is a tent function itself) so that each of these tent functions only uses variables from a given clique $C\in\sC(G)$. 
In Section~\ref{sec:location-of-tentpoles} we will describe the tent poles of $\log\hat{f}$.
\end{remark}

We end this section by noting that if the graph has more than one connected component, then, the MLE is the product of the MLEs for the different components.

\begin{proposition} \label{prop:mle-for-connected-components}
Let $G=(V,E)$ be an undirected graph. Suppose that $G=\bigsqcup_{i=1}^k G_i$, where each $G_j=(V_j, E_j)$ is a connected component of $G$, and let $\sX$ be an i.i.d. sample from a density in the family $\sF_G$. Then, the MLE $\hat{f}$ equals $\prod_{j=1}^k\hat{f}_{V_j}$, where each $\hat{f}_{V_j}$ is the MLE of a density in $\sF_{G_j}$ with respect to the projection $\sX_{V_j}$ of $\sX$.
\end{proposition}

\begin{proof}
The connected components of $G$ yield a partition $\{V_1,\dots,V_k\}$ of the vertex set $V=[d]$, where each $V_i$ is the set of vertices of $G_i$. Since the connected components of $G$ are pairwise disconnected, every $f\in\sF_G$ satisfies the identity $f(x)=\prod_{j=1}^k f_{V_j}(x_{V_j})$, where $f_{V_j}$ is the marginalization of $f$ over $V_j$.
Therefore, the log-likelihood $\ell(f,\sX)$ simplifies as
\[
\ell(f,\sX) = \sum_{i=1}^n w_i\log\left(\prod_{j=1}^k f_{V_j}(X_{V_j}^{(i)})\right) = \sum_{j=1}^k\sum_{i=1}^n w_i \log f_{V_j}(X_{V_j}^{(i)}) = \sum_{j=1}^k\ell(f_{V_j},\sX_{V_j})\,.
\]
Furthermore, whenever $f_{V_j}\in \sF_{G_j}$ for all $j$, the product $f=\prod_{j=1}^k f_{V_j}$ is a density, and, thus, lies in $\sF_G$, because
\[
\int_{\R^{d}}\prod_{j=1}^k f_{V_j}(x_{V_j})d\mu(x)=\prod_{j=1}^k\int_{\R^{V_j}}f_{V_j}(x_{V_j})d\mu(x_{V_j})=1\,.
\]
Furthermore, if $f\in\sF_G$, then $f = \prod_{j=1}^k f_{V_j}(x_{V_j})$, and we can rescale $f_{V_1}, \ldots, f_{V_k}$ to make sure that each of them is a density so that $f_{V_j}\in\sF_{G_j}$ for all $j$.
In addition, if $\log f=\log \prod_{j=1}^k f_{V_j} = \sum_{j=1}^k \log f_{V_j}$ is concave, then each of the $\log f_{V_j}$ has to be concave because $\{V_1,\dots,V_k\}$ is a partition of $V$. 
Summing up, the original optimization problem~\eqref{eq: opt_problem} in this case is equivalent to
\begin{align*}
    \mathrm{maximize}_{f_{V_j}\in\sF_{G_j},\,j\in[k]}\,\sum_{j=1}^k\ell(f_{V_j},\sX_{V_j})\,,
\end{align*}
which is equivalent to separately solving the optimization problems
\begin{align}
    \mathrm{maximize}_{f_{V_j}\in\sF_{G_j}}\,\ell(f_{V_j},\sX_{V_j})\quad\forall\,j\in[k]\,.
\end{align}
Therefore, the MLE $\hat{f}$ equals to the product of the MLEs $\hat{f}_{V_j}$ for the different connected components $G_j$, for $j\in[k]$.
\end{proof}

We now consider the case when the graph $G$ is a disjoint union $\bigsqcup_{i=1}^k G_i$ of complete graphs $G_i$, that correspond to the connected components of $G$.
Denote the vertex set of each $G_i$ by $V_i$. Since the collections of variables $\{V_1,\dots,V_k\}$ are mutually independent in the model, every density $f\in\sF_G$ can be written as $f=\prod_{i=1}^k f_i$ where $f_i\in\sF_{G_i}$ for all $i\in[k]$.

Corollary~\ref{cor:independence-model} follows from Proposition~\ref{prop:mle-for-connected-components} and shows that, in this case, the optimisation problem~\eqref{eq: opt_problem} over $G$ can be viewed as a sum of optimisation problems on each clique $G_i$ independently.

\begin{corollary} \label{cor:independence-model}
The optimal log-concave density for the disjoint union of cliques is the product of optimal log-concave densities for each of the cliques separately. In particular, the optimal log-concave density in the independence model is the product of optimal log-concave densities for each of the variables separately. 
\end{corollary}

\section{Existence and uniqueness of the MLE}\label{sec:existence}

The main result of this section is Theorem~\ref{thm: existence and uniqueness MLE}, which states the existence and uniqueness of the MLE over the family $\sF_G$ for an undirected chordal graph $G$. The notion of chordal graph is recalled in Definition~\ref{def: chordal graph}.
The minimum number of sample points required for existence and uniqueness of the MLE depends on the size of the largest clique of $G$.

\begin{theorem}\label{thm: existence and uniqueness MLE}
Let $G$ be an undirected chordal graph. Let $\sC(G)$ be the set of maximal cliques of $G$.
Let $\sX=\{X^{(1)},\ldots, X^{(n)}\}$ be an i.i.d. sample from a probability density $f_0$ on $\R^d$, and assume $|\sX|=n\ge\max_{C\in\sC(G)}|C|+1$.
Then, the MLE over $\sF_{G}$ exists and is unique with probability 1.
\end{theorem}

\begin{remark}
This result is reminiscent of the equivalent result obtained in~\cite{uhler2011geometry} regarding the minimum number of sample points in a Gaussian graphical model.
\end{remark}

The following three lemmas are key in the proof of Theorem~\ref{thm: existence and uniqueness MLE}. The first one of them shows that given a number $R\in\R$, for a density $f\in\sF_G$ to have log-likelihood at least $R$, it has to be bounded above.

\begin{lemma}\label{lem: bound log-likelihood}
Let $\tilde{G}$ be a chordal cover of the  undirected graph $G$.
Consider a sample $\sX$ with $|\sX|=n\ge\max_{C\in \sC(\tilde{G})}\{|C|\}+1$.
For every $R\in\R$ there exists a constant $M_R$ such that, for every log-concave density $f \in \sF_G$, if there is $x \in \sS_{G,\sX}$ such that $\log f(x)>M_R$, then $\ell(f,\sX)\le R$.
\end{lemma}

The next lemma shows that if a function $f\in\sF_G$ has likelihood at least $L\in\R$, where $L$ is the log-likelihood of the uniform density, then, $\log f$ has to be bounded below on $\sS_{\tilde G, \sX}$.

\begin{lemma}\label{lem: lower_bound_on_function}
Consider the assumptions of Lemma \ref{lem: bound log-likelihood}.
Suppose that $|\sC(\tilde G)|=t$.
Let $L$ be the log-likelihood of the uniform density on $\sS_{G,\sX}$ and let $M_L$ be the corresponding constant in Lemma~\ref{lem: bound log-likelihood}.
Consider a density $f \in \sF_G$ and let $M = \max_{x\in\R^d}\log f(x)$. If $\ell(f,\sX)>L$ and $M\le M_L$, then 
\begin{equation}\label{eq: minimum log f S_n}
\min_{\sD^{(t-1)}_{G}}\log f\ge m_L \coloneqq \frac{2^{t-1}}{\max_i w_i}(L-M_L)+M_L\,.
\end{equation}
\end{lemma}

The last lemma we need shows that if a concave function supported on $\sS_{G, \sX}$ has a concave decomposition $\sum_{C\in\sC(G)} \phi_C(x_C)$ according to a graph $G$ and is bounded, then each of the summands $\phi_C$ can also be chosen to be bounded.
\begin{lemma}\label{lem:bounded_components}
Let $G$ be a graph and $\sX$ a  sample of points in $\R^d$ with
$|\sX|=n\ge\max_{C\in \sC(G)}\{|C|\}+1$.
Let $\phi\colon\R^d\to\R$ be a concave function with support $\sS_{G,\sX}$.
Suppose that $\phi(x)=\sum_{C\in\sC(G)}\phi_C(x_C)$ for all $x\in\R^d$, where each function $\phi_C\colon\pi_C(\sS_{G,\sX})\to\R$ is concave. If $\mathrm{Im}(\phi)\subseteq[-K,K]$, then there exists a decomposition $\phi(x)=\sum_{C\in\sC(G)}\tilde{\phi}_C(x_C)$ such that each function $\tilde{\phi}_C$ is concave and $\mathrm{Im}(\tilde{\phi}_C)\subseteq[-\rho K, \rho K]$ for some $\rho\ge 0$, which depends only on the graph $G$ and the set  of sample points $\sX$.
\end{lemma}

The proofs of Lemmas \ref{lem: bound log-likelihood}, \ref{lem: lower_bound_on_function} and \ref{lem:bounded_components} can be found in Appendix \ref{appendix: proofs section 2,4,5}. In particular, we apply Lemma \ref{lem:bounded_components} to $\phi=\log f$ where $f\in\sF_G$.
This means that, if we can bound $\log f(x)$ on $\sS_{G,\sX}$, then we can also bound the logarithm of the clique potential $\log\psi_C$ for all $C \in \sC(G)$.

\begin{proof}[Proof of Theorem~\ref{thm: existence and uniqueness MLE}]
Let $L$ be the log-likelihood of the uniform density on $\sS_{G,\sX}$, let $M_L$ and $m_L$ be the corresponding constants in Lemmas~\ref{lem: bound log-likelihood} and~\ref{lem: lower_bound_on_function}.
Define $K\coloneqq\max \left\{|M_L|,|m_L|\right\}$. It follows from Lemmas~\ref{lem: bound log-likelihood} and~\ref{lem: lower_bound_on_function} that for the maximum likelihood estimation problem we can restrict to densities $f\in \sF_G$ such that $\log f(x) \in [-K,K]$ for all $x\in\sS_{G,\sX}$.
By Lemma~\ref{lem:bounded_components}, if $\log f(x) \in [-K,K]$ on $\sS_{G,\sX}$, then each potential $\psi_C$ can be chosen such that $\log\psi_C(x) \in [-\rho K, \rho K]$ for a fixed $\rho \in\R_+$. Moreover, we can assume by Theorem \ref{thm: shape of MLE tent}$(i)$ that
\begin{equation}\label{eq: assumption tent function}
\log\psi_C = h_{\sX_C,y_C}\quad\forall\,C\in\sC(G)\,,
\end{equation}
where $h_{\sX_C,y_C}$ is a tent function supported on $\sX_C$ with heights $y_C=(y_C^{(1)},\ldots,y_C^{(n)})$, in particular $y_C^{(i)}=\log\psi_C(X_C^{(i)})$ for all $i\in[n]$ and all $C\in\sC(G)$. 

Thus, by Theorem~\ref{thm: shape of MLE tent}$(ii)$, our goal is to maximize the objective function $\tau\colon\R^{n|\sC(G)|}\to\R$  introduced in \eqref{eq: new obj function}: 
\begin{equation}\label{eq: def composition map}
\tau((y_C)_{C \in \sC(G)}) \coloneqq \sum_{i=1}^n\sum_{C\in\sC(G)}w_i y_C^{(i)} - \int_{\R^d}\prod_{C\in\sC(G)}\exp(h_{\sX_C, y_C})d\mu\,.
\end{equation}
The function $\tau$ has a finite-dimensional domain and is continuous on the compact set $[-\rho K, \rho K]^{n|\sC(G)|}$. In particular, it admits a global maximizer over $[-\rho K, \rho K]^{n|\sC(G)|}$. This proves the existence of the MLE.

\noindent In order to show uniqueness of the MLE, consider two optimal densities on $\sF_G$
\[
f_1(x)=\prod_{C\in\sC(G)}\psi_{1,C}(x_C)\,,\quad f_2(x)=\prod_{C\in\sC(G)}\psi_{2,C}(x_C)\,,
\]
which means that they both maximize the log-likelihood $\ell(\cdot,\sX)$.
Consider the function
\[
g(x) \coloneqq \frac{\sqrt{f_1(x)f_2(x)}}{\int_{\sS_{G,\sX}}\sqrt{f_1f_2}\,d\mu} = \frac{1}{\int_{\sS_{G,\sX}}\sqrt{f_1f_2}\,d\mu}\prod_{C\in\sC(G)}\sqrt{\psi_{1,C}(x_C)\psi_{2,C}(x_C)}\,.
\]
It is a density with support $\sS_{G,\sX}$ and admits a log-concave decomposition, hence $g\in\sF_G$ and 
\begin{align*}
\ell(g,\sX) &= \sum_{i=1}^n w_i\log g(X^{(i)})\\
&= \sum_{i=1}^n w_i\left[\frac{1}{2}\left(\log f_1(X^{(i)})+\log f_2(X^{(i)})\right)-\log\left(\int_{\sS_{G,\sX}}\sqrt{f_1f_2}\,d\mu\right)\right]\\
&= \frac{1}{2}(\ell(f_1,\sX)+\ell(f_2,\sX))-\log\left(\int_{\sS_{G,\sX}}\sqrt{f_1f_2}\,d\mu\right)\\
&= \ell(f_1,\sX)-\log\left(\int_{\sS_{G,\sX}}\sqrt{f_1f_2}\,d\mu\right)\,.
\end{align*}
By the Cauchy-Schwarz inequality $\int_{\sS_{G,\sX}}\sqrt{f_1f_2}\,d\mu\le 1$ with equality if and only if $f_1=f_2$ almost everywhere.
Thus $\ell(g,\sX)\ge\ell(f_1,\sX)$ and the equality holds if and only if $f_1=f_2$ almost everywhere.
Since they are continuous on $\sS_{G,\sX}$, by~\cite[Theorem 10.2]{rockafellar1970convex} $f_1=f_2$. 
\end{proof}

\section{Location of tent poles} \label{sec:location-of-tentpoles}

In the case of log-concave density estimation (not in a graphical model), the MLE equals the exponential of a tent function with tent poles at the data sample~\cite{cule2010maximum}. In our setting, even if the MLE $\hat{f}$ is a tent function, then it can have tent pole locations that are neither the vertices of $\sS_{G,\sX}$ nor the data sample $\sX$.

\begin{example} \label{example:tent_poles}
Consider the graph $G=(V,E)$ with $V=[2]$ and $E=\emptyset$. The MLE $\hat{f}$ for $G$ is the product of the MLEs for each of the two vertices by Proposition~\ref{prop:mle-for-connected-components} and Corollary~\ref{cor:independence-model}. Suppose we have the sample $\sX=\{(0,0),(1,1),(2,2)\}$ with weights $w=(1/10,8/10,1/10)$.  For both vertices, the MLE is the exponential of a tent function with tent pole locations at  $\{0,1,2\}$. Hence $\hat{f}$ is the exponential of a tent function with tent pole locations at $\{0,1,2\} \times \{0,1,2\}$.
\end{example}

We refer to Definition \ref{def: subdivision} for the notion of (regular) subdivision of a polytope. The following result allows us to compute the subdivision generated by a tent function $h$, and, therefore, know the regions on which $h$ is linear. That is how we evaluate $\int\exp h\,d\mu$ and its derivatives (in terms of the heights $(y_C)_{C\in\sC(G)}$). The proofs of Proposition \ref{prop: subdivision_intersection} and Corollary \ref{cor:faces_of_the_subdivision} are located in Appendix \ref{appendix: proofs section 2,4,5}.

\begin{proposition}\label{prop: subdivision_intersection}
Consider $G$ and $\sX$ as before.
Let $h=\sum_{C \in \sC(G)} h_{\sX_C,y_C}$ be a tent function supported on $\sS_{G,\sX}$. Let $\sigma$ denote the regular subdivision of $\sS_{G,\sX}$ induced by $h$, and for every $C \in \sC(G)$, let $\sigma_C$ denote the subdivision of $\pi_C(\sS_{G,\sX})$ induced by $h_{\sX_C,y_C}$. Then
\begin{equation} \label{equation:intersection-of-tent-functions-on-cliques}
\sigma=\left\{\bigcap_{C \in \sC(G)}\pi_C^{-1}(Q_C) \mid \text{$Q_C\in\sigma_C$ $\forall C \in \sC(G)$}\right\}\,,
\end{equation}
i.e., each cell of $\sigma$ is an intersection over $C\in\mathcal C(G)$ of the inverse projections of a cell of $\sigma_C$.
\end{proposition}

\begin{corollary}\label{cor:faces_of_the_subdivision}
Assume that $\dim(\sS_{G,\sX})=d$. Let $h=\sum_{C\in\sC(G)} h_{\sX_C,y_C}$.
Given $z \in \sS_{G,\sX}$, let $k_z$ be the codimension of the minimal face of $\sigma$  containing $z$, and, for every $C\in\sC(G)$, let $k_{z,C}$ be the codimension of the minimal faces of $\sigma_C$  containing $z_C$.
Then $k_z\le\sum_{C\in\sC(G)} k_{z,C}$.
\end{corollary}

\begin{conjecture}
Assume that $\dim(\sS_{G,\sX})=d$. Let $h=\sum_{C\in\sC(G)} h_{\sX_C,y_C}$.
Given $z \in \sS_{G,\sX}$, let $k_z$ and $l_z$ be the codimensions of the minimal faces of $\sigma$ and $\sS_{G,\sX}$ containing $z$, respectively.
Similarly, for every $C\in\sC(G)$ let $k_{z,C}$ and $l_{z,C}$ be the codimensions of the minimal faces of $\sigma_C$ and $\pi_C(\sS_{G,\sX})$ containing $z_C$, respectively. If $\sS_{G,\sX}$ and all tent functions $h_{\sX_C,y_C}$ are sufficiently generic, then
\begin{equation}\label{eq: difference k_z l_z}
k_z - l_z = \sum_{C\in\sC(G)} k_{z,C} - \sum_{C\in\sC(G)} l_{z,C}\,.
\end{equation}
\end{conjecture}

\begin{example} \label{example:illustration-of-the-conjecture-about-faces}
Consider again the graph $G$ of Example \ref{ex: shape of S_n}.
Suppose now that $\sS_{G,\sX}=[0,1]^3$, therefore $\pi_{12}(\sS_{G,\sX})= [0,1]^2$ and $\pi_{23}(\sS_{G,\sX})=[0,1]^2$. We consider subdivisions of $\pi_{12}(\sS_{G,\sX})$ and $\pi_{23}(\sS_{G,\sX})$ having two maximal cells as shown in the middle and on the right of Figure~\ref{fig:subdivision}. The subdivision $\sigma=\pi_{12}^{-1}(\sigma_{12}) \cap \pi_{23}^{-1}(\sigma_{23})$ has four maximal cells as shown on the left of Figure~\ref{fig:subdivision}. 
Consider the point $z=(1/2,0,1/2)$.
Then $k_z=3$ as $z$ is a vertex of $\sigma$, $k_{z,12}=k_{z,23}=2$ because $(z_1,z_2)$ and $(z_2,z_3)$ are vertices of $\sigma_{12}$ and $\sigma_{23}$, respectively.
Hence the inequality $k_z<k_{z,12}+k_{z,23}$ is strict. However, if we take into account boundaries, i.e. $l_z=1$ and $l_{12,z}=l_{23,z}=1$, then we get the equality $k_{12,z}+k_{23,z}-l_{12,z}+l_{23,z}=2+2-1-1=3-1=k_z-l_z$.
\begin{figure}[ht]
    \centering
    \includegraphics[width=0.6\textwidth]{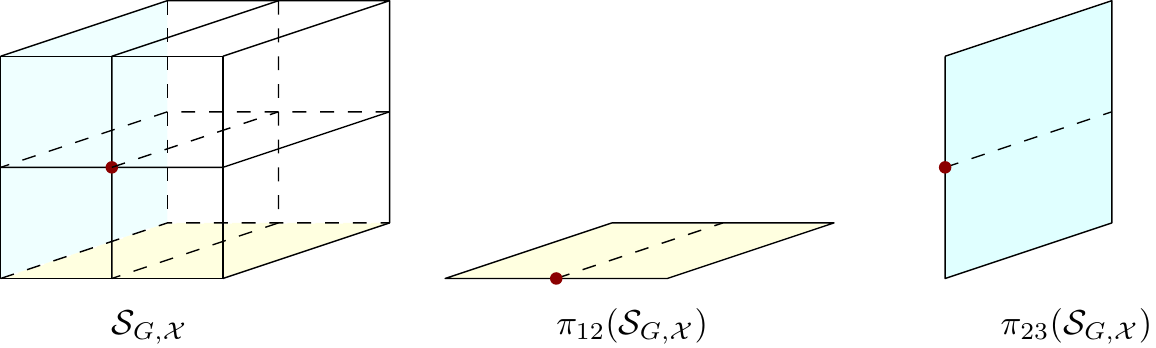}
    \caption{Illustration of Example~\ref{example:illustration-of-the-conjecture-about-faces}. We consider the graph with $V=\{1,2,3\}$ and $E=\{12,23\}$. The left figure shows $\sS_{G,\sX}$ and the two right figures show the projections $\pi_{12}(\sS_{G,\sX})$ and $\pi_{23}(\sS_{G,\sX})$.}
    \label{fig:subdivision}
\end{figure}
\end{example}

\begin{corollary} \label{corollary:tent_poles}
Under the same assumptions as in Corollary~\ref{cor:faces_of_the_subdivision}, if $z \in \sS_{G,\sX}$ is a tent pole location for $h$, then $\sum_C k_{z,C} \ge d$.
\end{corollary}

\begin{example}
In Example~\ref{example:tent_poles},  $\dim(\pi_1(\sS_{G,\sX}))=\dim(\pi_2(\sS_{G,\sX}))=1$. Hence the codimension of a face in $\sigma_1$ or $\sigma_2$ can be at most 1. For $\sum_C k_C \ge d=2$ to hold, we must have $k_1=k_2=1$. This implies that if $z$ is a vertex in the subdivision $\sigma$, then $z_1$ and $z_2$ are vertices in the subdivisions $\sigma_1$ and $\sigma_2$. In this example, $z$ is a vertex of $\sigma$ if and only if $z_1$ and $z_2$ are vertices of $\sigma_1$ and $\sigma_2$ respectively. An analogous statement is true for any independence model: a point $z \in \sS_{G,\sX}$ is a vertex of $\sigma$ if and only if each $z_C$ is a vertex of $\sigma_C$. In particular, the set of vertices for $\sigma$ is the Cartesian product over $C \in \sC(G)$ of vertices for $\sigma_C$.
\end{example}

\begin{remark}
For a point $z_C$ to belong to a face of codimension $k_C$, the minimum in~\eqref{eqn:tent-funtion-as-max} has to be achieved by at least $k_C+1$ arguments. Hence if $l_C$ denotes the number of arguments in~\eqref{eqn:tent-funtion-as-max} that achieve the minimum at $z_C$, then if $z$ is a tent pole, $\sum_C (l_C-1) \ge d$.
\end{remark}

\section{Optimization algorithm}\label{sec:optimization}

We implement an algorithm for computing the optimal log-concave density in an undirected graphical model based on the optimization formulation~\eqref{eq:new_opt_problem}. Our implementation in \texttt{R} is available at \cite{githubrepo}. The objective function $\tau\colon\R^{n|\sC(G)|}\to\R$ from~\eqref{eq:new_opt_problem}  is not differentiable by a similar argument as in~\cite[Section 3.2]{cule2010maximum}. Let us consider the tent function $h_{\sX_C,y_C}$ for the clique $C$. If a tent pole is touching but not supporting a tent, then decreasing the height of this tent pole does not change the tent function, while increasing the height of this tent pole changes the tent function. This implies that the objective function $\tau$ is not differentiable at this point. 

The set of tuples $(y_C)_{C \in \sC(G)}$ for which the objective function $\tau$ is not differentiable is a set of Lebesgue measure zero in $\R^{n|\sC(G)|}$, however, these points cannot be neglected in our optimization algorithm. To apply an optimization method using gradients, at the points where $\tau$ is not differentiable, we compute a subgradient of $\tau$.
To find the optimal solution to~\eqref{eq:new_opt_problem}, we use the \texttt{BFGS} method in the \texttt{optim} package in \texttt{R}, which is a quasi-Newton method that uses function and (sub)gradient values. 

When computing the function and (sub)gradient values, we need to keep track of the heights $(y_C)_{C \in \sC(G)}$ on cliques and the tent function $h$ on $\sS_{G,\sX}$ together with its tent poles and induced triangulation, so that we can apply the formula in~\cite[Lemma 2.4]{robeva2019geometry} for computing the integral in~\eqref{eq:new_opt_problem}. To do so, we apply Proposition~\ref{prop: subdivision_intersection}, which allows us to compute the tent poles of $h$ from the tent poles of $h_{\sX_C,y_C}$. When considering a point $(y_C)_{C \in \sC(G)}$ where $\tau$ is not differentiable, then these heights induce subdivisions of point configurations $\sX_C$, for $C \in \sC(G)$, that can be further triangulated. Different triangulations result in different subgradients. In practice, the choice of triangulation is made by the \texttt{geometry} package in \texttt{R} that we use for computing triangulations induced by a tent function. A future research direction is to explore where the implementation can be improved using ideas from~\cite{chen2021new}.

We illustrate our optimization algorithm on two examples.

\begin{example}\label{ex: MLE indModel}
Let $G=(V,E)$ with $V=[2]$ and $E=\emptyset$, which corresponds to the independence model on two random variables. As discussed in Corollary~\ref{cor:independence-model}, we expect that our optimization method should give the same tent function as when we apply \texttt{LogConcDEAD} on each vertex separately and then add the resulting tent functions.

We consider the dataset $\sX$ of Table \ref{table: sample 50 points indModel} in Appendix \ref{app: sec optimization}. It consists of $50$ points in $\R^2$ sampled from a standard normal distribution.
The tent functions found by our method and \texttt{LogConcDEAD} are depicted in Figure~\ref{fig:optimization-1-2}. With respect to the uniform weight vector $w$, the optimal value found by our method is $3.755142$ and by \texttt{LogConcDEAD} is $3.754209$.
On the one hand, our method produces a solution with $728$ tent poles. 
On the other hand \texttt{LogConcDEAD} produces a solution with $408$ tent poles. Some of these tent poles are not effective, meaning that one can select a subset of tent poles that give the same tent function. 
Figure~\ref{fig:optimization-1-2} suggests that the corresponding tent functions have few effective tent poles and linear regions. Determining only the effective tent poles is in general challenging and is studied in~\cite{grosdos2020exact}. Figure \ref{fig: densities} compares the true standard Gaussian density with the optimal densities obtained from our method with $E=\emptyset$ and \texttt{LogConcDEAD} with $E=\{\{1,2\}\}$.

\begin{figure}[ht]
\centering
\begin{subfigure}{.31\textwidth}
  \centering
  \includegraphics[width=\linewidth]{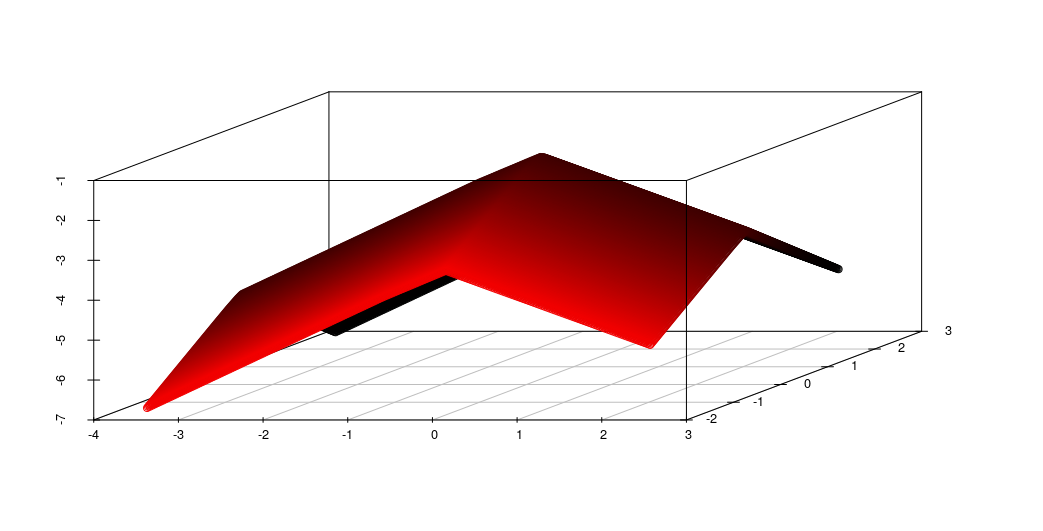}
  \caption{Tent function constructed by our method for $E=\emptyset$.}
  \label{fig:optimization-1-2-our}
\end{subfigure}
\hspace{6pt}
\begin{subfigure}{.31\textwidth}
  \centering
  \includegraphics[width=\linewidth]{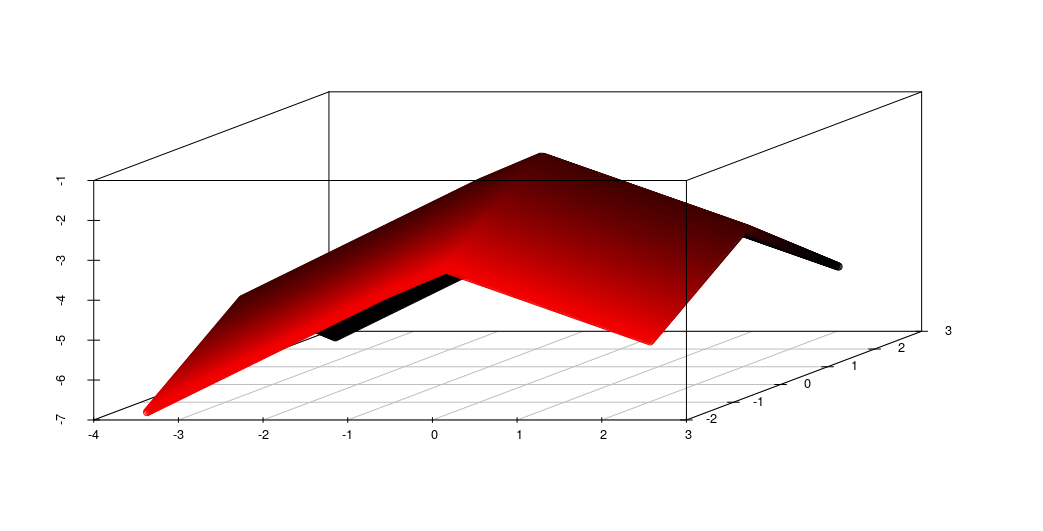}
  \caption{Tent function constructed by \texttt{LogConcDEAD} for $E=\emptyset$.}
  \label{fig:optimization-1-2-LogConcDEAD}
\end{subfigure}
\hspace{6pt}
\begin{subfigure}{.31\textwidth}
  \centering
  \includegraphics[width=\linewidth]{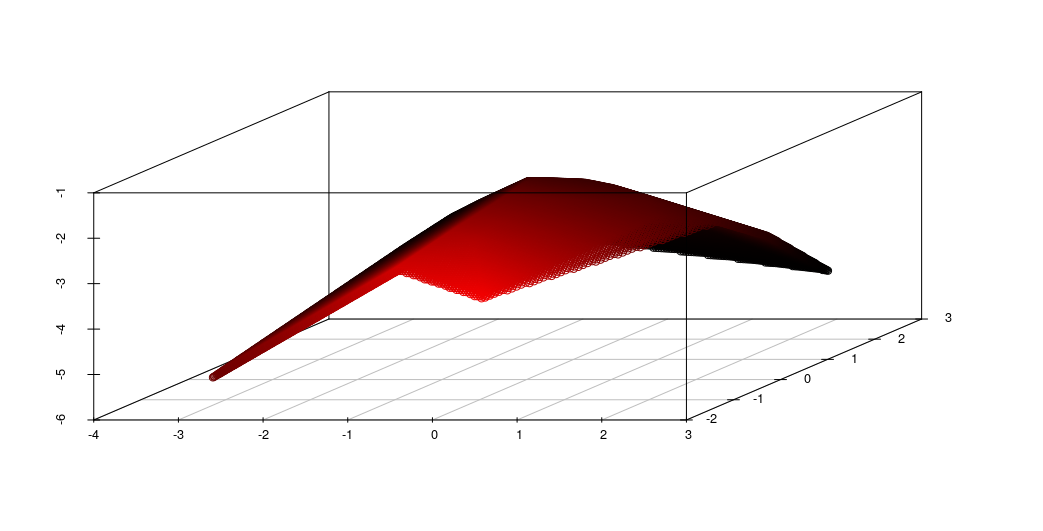}
  \caption{Tent function constructed by \texttt{LogConcDEAD} for $E=\{\{1,2\}\}$.}
  \label{fig:optimization-12-LogConcDEAD}
\end{subfigure}
\caption{Comparison of our method and \texttt{LogConcDEAD} for finding the optimal tent function for the independence model on two random variables.
Although the tent functions in~\ref{fig:optimization-1-2-our} and~\ref{fig:optimization-1-2-LogConcDEAD} look the same, they are obtained using different methods.}
\label{fig:optimization-1-2}
\end{figure}

\begin{figure}[ht]
\centering
\begin{subfigure}{.31\textwidth}
  \centering
  \includegraphics[width=\linewidth]{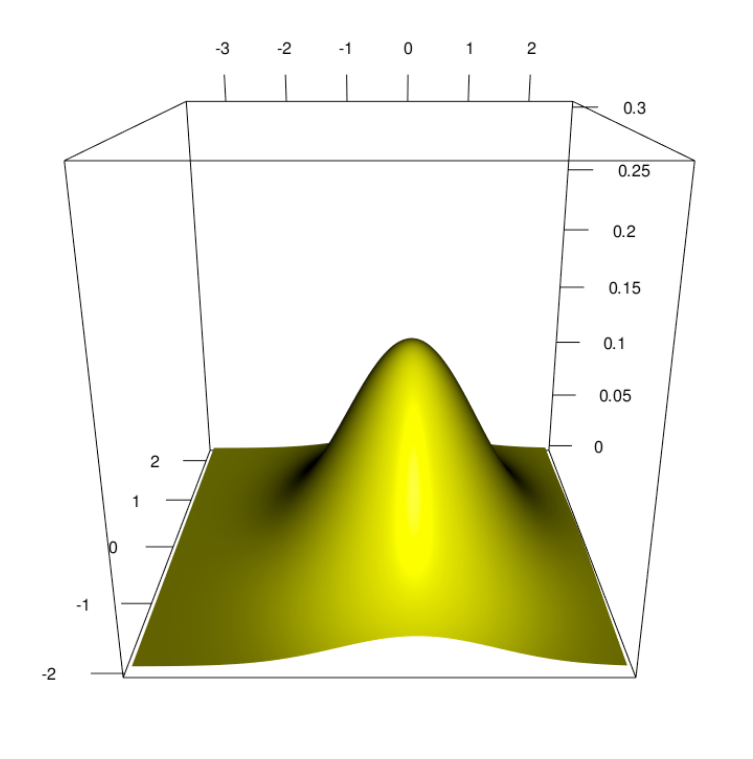}
  \caption{2-dimensional standard normal distribution.}
  \label{fig: true density}
\end{subfigure}
\hspace{6pt}
\begin{subfigure}{.31\textwidth}
  \centering
  \includegraphics[width=\linewidth]{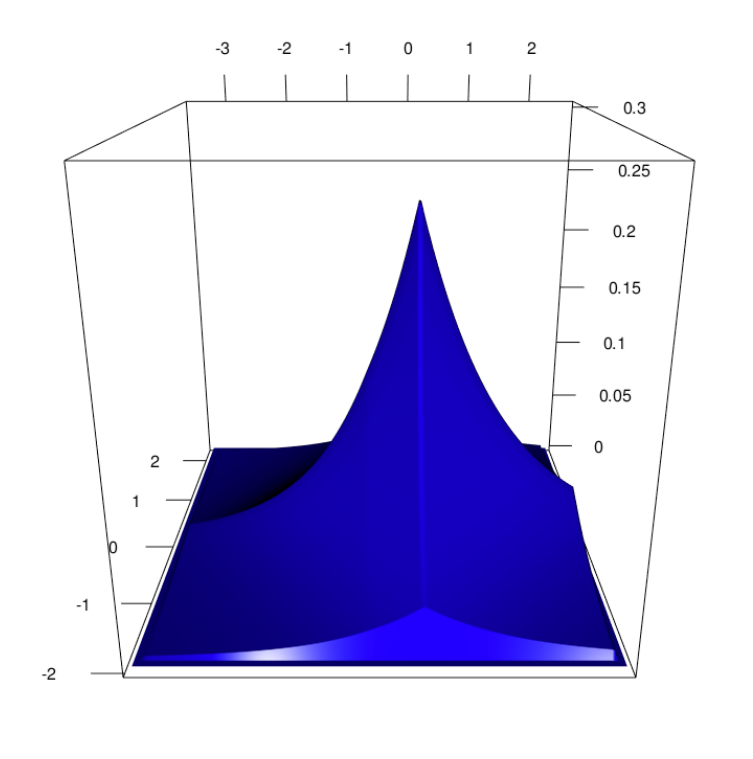}
  \caption{Density constructed by our method with $E=\emptyset$.}
  \label{fig: density our}
\end{subfigure}
\hspace{6pt}
\begin{subfigure}{.31\textwidth}
  \centering
  \includegraphics[width=\linewidth]{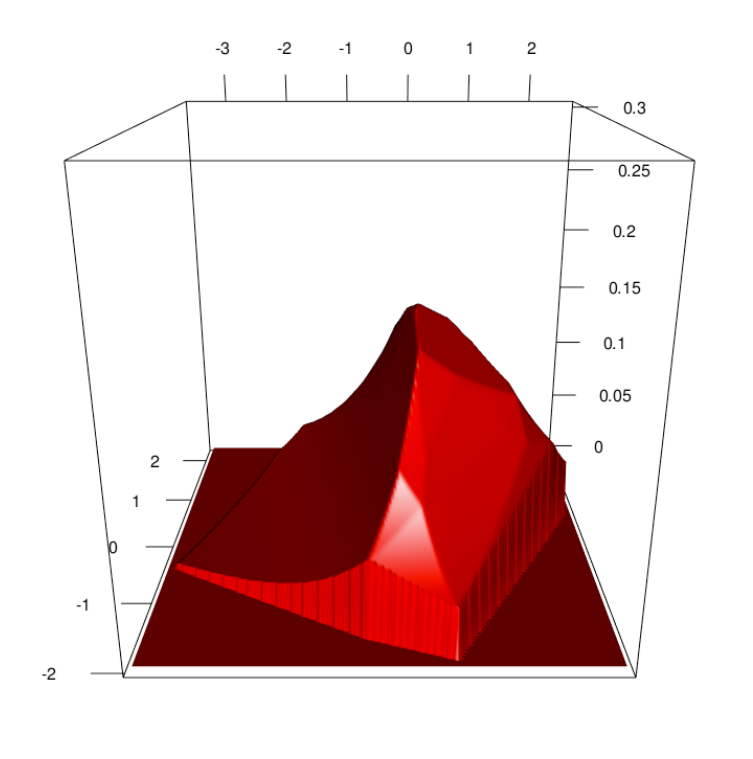}
  \caption{Density constructed by \texttt{LogConcDEAD} with $E=\{\{1,2\}\}$.}
  \label{fig: density LogConcDEAD}
\end{subfigure}
\caption{Comparison of the optimal densities constructed by our method for the independence model on two random variables and \texttt{LogConcDEAD} for classical log-concave estimation. The dataset consists of 50 points sampled from the 2-dimensional standard normal distribution as in Example~\ref{ex: MLE indModel}.}
\label{fig: densities}
\end{figure}

Finally, for sample sizes $n\in\{10,20,\ldots,100, 200, \ldots, 500\}$, we compute  the $l_2$-distance between the log-concave graphical MLE on two variables with $E=\emptyset$ and the true 2-dimensional Gaussian distribution. In the first two rows of Table \ref{table: mean variance distances} in Appendix \ref{app: sec optimization}, we present the mean and standard deviation of the $l_2$-distance for each $n$ over 10 trials.
A similar study is carried out for $E=\{\{1,2\}\}$ using \texttt{LogConcDEAD}. The values are collected in the third and fourth row of Table \ref{table: mean variance distances}. In Figure \ref{fig: mean variance distances}, we provide also an error bar plot of the $\ell_2$-distances for $n\in\{10,20,\ldots,100\}$. The red (respectively, blue) points and bars represent the mean values and the standard deviations of the distances between each log-concave graphical (respectively, log-concave without constraints) MLE and the true Gaussian distribution.

\begin{figure}[ht]
\centering
\includegraphics[width =0.5\linewidth]{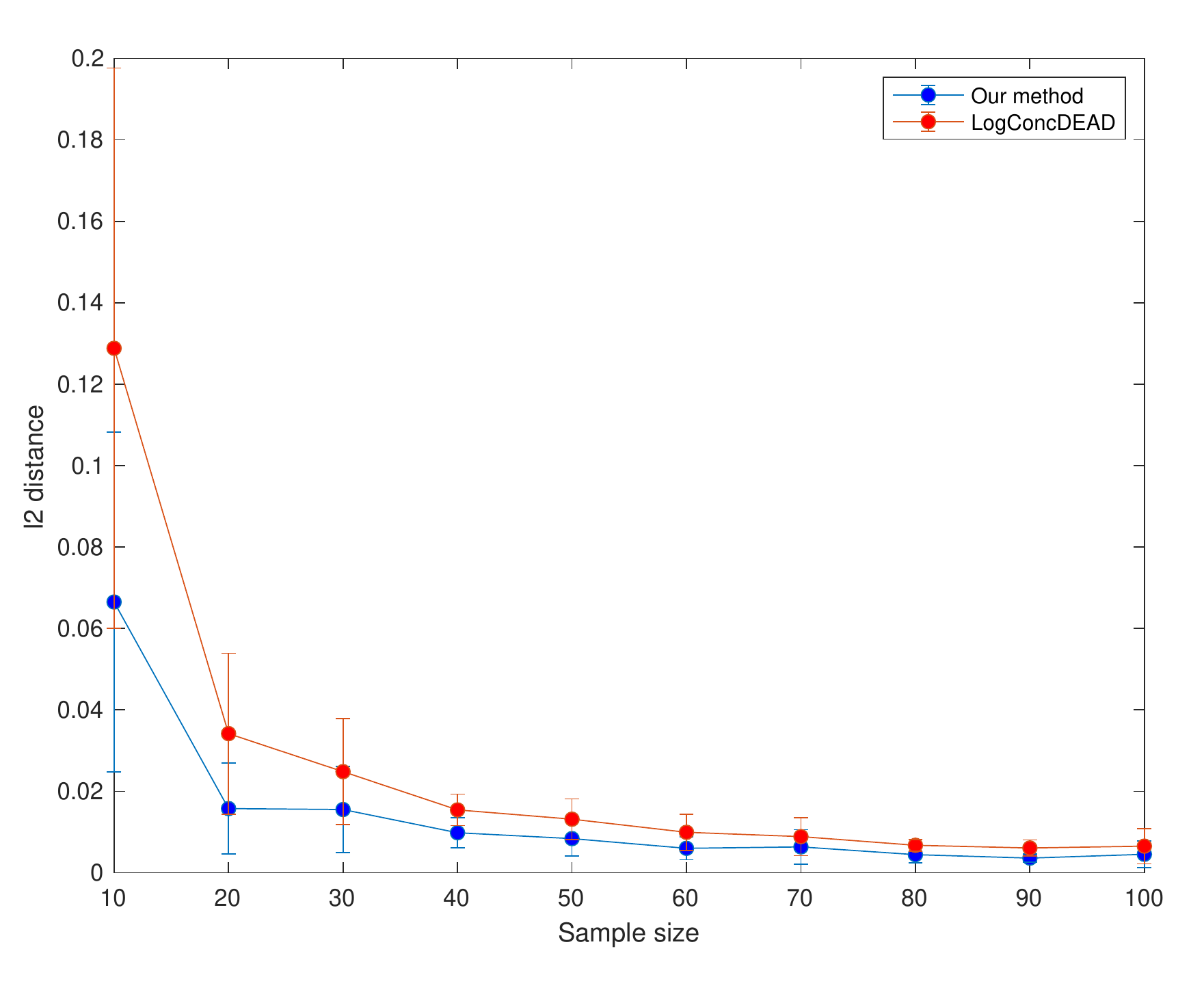}
\caption{The error bar plot compares the performances between our method and \texttt{LogConcDEAD} for each $n \in\{10,20,\ldots,100\}$ over 10 trials.}
\label{fig: mean variance distances}
\end{figure}

\end{example}

\begin{example}\label{ex: optimal-tent-function2}
We consider the graph $G=([3],\{\{1,2\},\{2,3\}\})$ and the dataset $\sX$ of Table \ref{table: sample 60 points graph 12-23} in Appendix \ref{app: sec optimization}, which consists of $60$ points in $\R^3$ sampled from a normal distribution with mean $0$ and standard deviation $5$.
The corresponding polytope $\sS_{G,\sX}$ has 24 vertices, listed in Table \ref{table: vertices support MLE graph 12-23} in Appendix \ref{app: sec optimization}.
Only three sample points in $\sX$ are also vertices of $\sS_{G,\sX}$.
Let $w$ be the uniform weight vector. Our optimization algorithm outputs that the optimal tent function has $2361$ tent poles.
The optimal value found by our method is $9.60738$.
The optimal tent function $h$ is the sum of tent functions on the cliques, i.e., $h=h_{\sX_{12},y_{12}}+h_{\sX_{23},y_{23}}$. The sets $\sX_{12}$ and $\sX_{23}$ and tent functions $h_{\sX_{12},y_{12}}$ and $h_{\sX_{23},y_{23}}$ are depicted in Figure~\ref{fig:optimal-tent-function2}.

\begin{figure}[ht]
\centering
\begin{subfigure}{.5\textwidth}
  \centering
  \includegraphics[width=\linewidth]{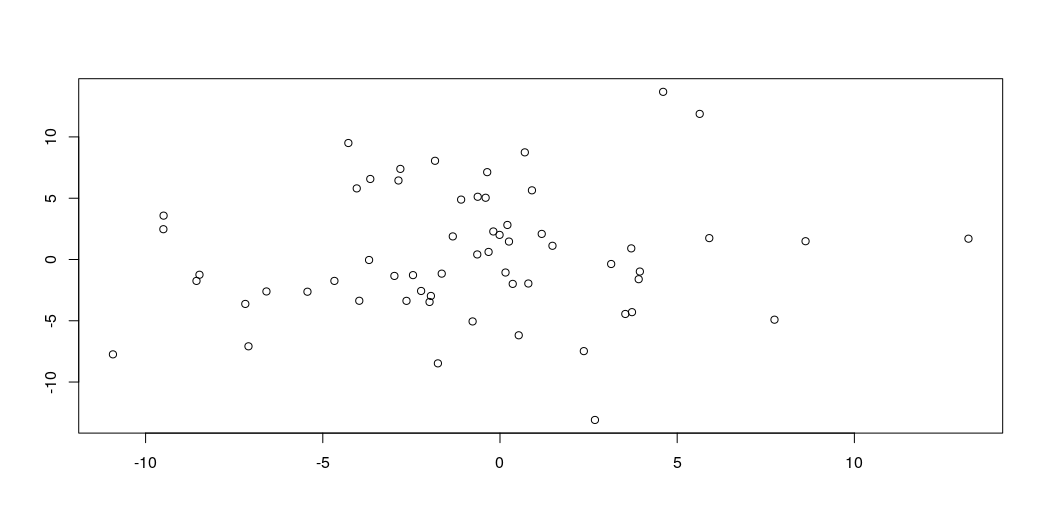}
  \caption{$\sX_{12}$}
  \label{fig:proj12-points}
\end{subfigure}%
\hfill
\begin{subfigure}{.5\textwidth}
  \centering
  \includegraphics[width=\linewidth]{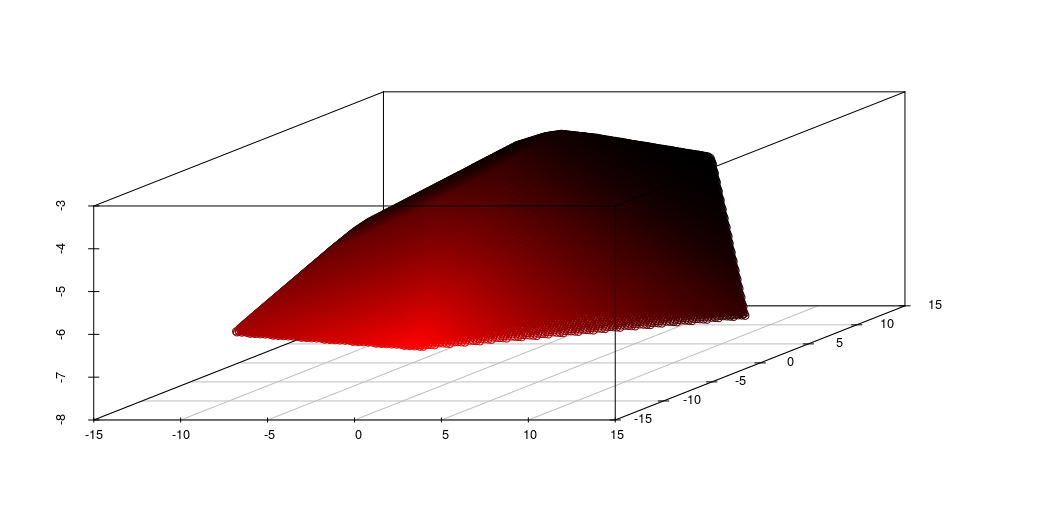}
  \caption{$h_{\sX_{12},y_{12}}$}
  \label{fig:proj12-tent-function}
\end{subfigure}
\begin{subfigure}{.5\textwidth}
  \centering
  \includegraphics[width=\linewidth]{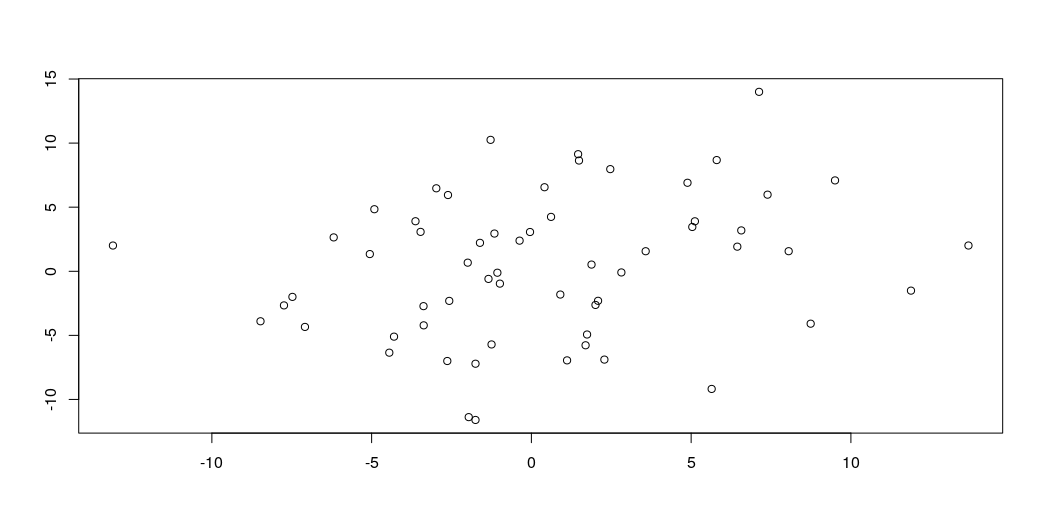}
  \caption{$\sX_{23}$}
  \label{fig:proj23-points}
\end{subfigure}%
\hfill
\begin{subfigure}{.5\textwidth}
  \centering
  \includegraphics[width=\linewidth]{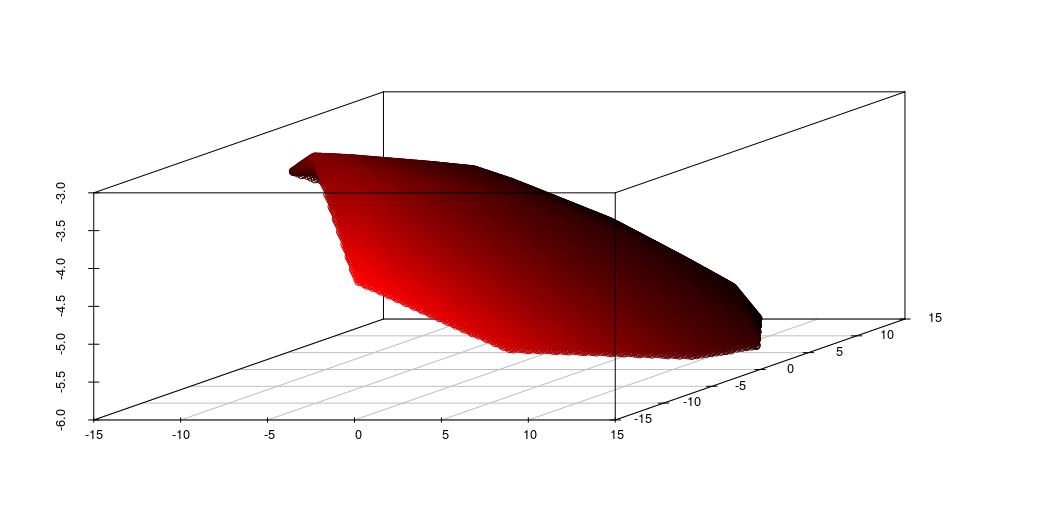}
  \caption{$h_{\sX_{23},y_{23}}$}
  \label{fig:proj23-tent-function}
\end{subfigure}
\caption{Decomposition of the optimal tent function over the cliques in Example~\ref{ex: optimal-tent-function2}.}
\label{fig:optimal-tent-function2}
\end{figure}
\end{example}

\section{Consistency for a disjoint union of cliques}\label{sec: disjoint union}

In this section we show that, when the graph is a disjoint union of cliques, then the maximum likelihood estimator is consistent. 
We denote every estimated density with $\hat{f}_n$ to stress its dependence on the cardinality $n$ of the given sample $\sX$.

\begin{theorem} \label{thm:consistency-for-connected-components}
Let $G=(V,E)$ be an undirected graph with $V=[d]$. Suppose that $G=\bigsqcup_{i=1}^k G_i$, where each $G_i$ is a complete graph that corresponds to a connected component of $G$. Denote the vertex set of each $G_i$ by $V_i$.
Let $f_0$ be any density on $\R^d$ such that
\begin{equation}\label{eq: integrability conds}
\begin{gathered}
\int_{\R^{V_i}}\|x_{V_i}\| (f_0)_{V_i} d\mu_{V_i}<\infty\,,\quad\int_{\R^{V_i}}(f_0)_{V_i}\log_{+}(f_0)_{V_i}d\mu_{V_i}<\infty\,,\quad\forall\,i\in[k]\,,\\
\int_{\R^d}f_0\log_{+}f_0\,d\mu<\infty\,,\quad\text{and}\ \mathrm{int}(\supp(f_0)) \neq \emptyset\,.
\end{gathered}
\end{equation}
\begin{enumerate}
    \item\label{item: KL} There exists a log-concave density $f^* \in \sF_G$, unique almost everywhere, such that
    \begin{equation}\label{eq: infimum product}
    d_\mathrm{KL}(f_0,f^*)=\inf_{f\in\sF_G}d_\mathrm{KL}(f_0,f)\,.
    \end{equation}
    \item\label{item: strong convergence} Suppose that $(\hat{f}_n)$ is a sequence of estimated densities in $\sF_G$. Let $a_0\in\R_{>0}$ and $b_0\in\R$ satisfy $f^*(x)\le e^{-a_0\|x\|+b_0}$ for all $x\in\R^d$. Then for all $a<a_0$
    \begin{equation}\label{eq: strong convergence tv}
    \int_{\R^d} e^{a\|x\|}|\hat{f}_n(x)-f^*(x)|d\mu(x)\stackrel{a.s.}{\to} 0
    \end{equation}
    and, whenever $f^*$ is continuous,
    \begin{equation}\label{eq: strong convergence sup}
    \sup_{x \in \R^d} e^{a\|x\|}|\hat{f}_n(x)-f^*(x)| \stackrel{a.s.}{\to} 0\,.
    \end{equation}
\end{enumerate}
\end{theorem}
\noindent The assumptions~\eqref{eq: integrability conds} are satisfied by all $f_0 \in \sF_G$ for the same reasons as in~\cite[Theorem 4]{cule2010theoretical}.
\begin{proof}
\ref{item: KL}. Consider a density $f\in\sF_G$, in particular $f=\prod_{i=1}^k f_i$ for some densities $f_i\in\sF_{G_i}$. We show that
\begin{equation}\label{eq: identity KL divergence product}
d_\mathrm{KL}(f_0,f) = c(f_0)+\sum_{i=1}^k d_\mathrm{KL}((f_0)_{V_i},f_i)
\end{equation}
for some constant $c(f_0)$ depending only on the fixed density $f_0$. Indeed
\begin{align}
\begin{split}
d_\mathrm{KL}(f_0,f) &= \int_{\R^d}f_0(\log f_0-\log f)d\mu = \int_{\R^d}f_0\log f_0\,d\mu-\sum_{i=1}^k\int_{\R^d}f_0\log f_i\,d\mu\\
&= \int_{\R^d}f_0\log f_0\,d\mu-\sum_{i=1}^k\int_{\R^{V_i}}(f_0)_{V_i}\log f_i\,d\mu_{V_i}\\
&=\int_{\R^d}f_0\log f_0\,d\mu-\sum_{i=1}^k\int_{\R^{V_i}}(f_0)_{V_i}\log (f_0)_{V_i}d\mu_{V_i}\\
&\quad +\sum_{i=1}^k\int_{\R^{V_i}}(f_0)_{V_i}(\log(f_0)_{V_i}-\log f_i)d\mu_{V_i}\\
&= c(f_0)+\sum_{i=1}^k d_\mathrm{KL}((f_0)_{V_i},f_i)\,.
\end{split}
\end{align}
Applying~\cite[Theorem 4]{cule2010theoretical}  to every marginal density $(f_0)_{V_i}$, we know that for each $i\in[k]$ there exists a density $f_i^*\in\sF_{G_i}$, unique almost everywhere, such that
\[
d_\mathrm{KL}((f_0)_{V_i},f_i^*)=\inf_{f_i\in\sF_{G_i}}d_\mathrm{KL}((f_0)_{V_i},f_i)\,.
\]
Using \eqref{eq: identity KL divergence product}, we have that
\begin{equation}\label{eq: lower bound infimum KL}
d_\mathrm{KL}(f_0,f) \ge c(f_0)+\sum_{i=1}^k d_\mathrm{KL}((f_0)_{V_i},f_i^*) = d_\mathrm{KL}\left(f_0,\prod_{i=1}^k f_i^*\right)\,,
\end{equation}
where in the last equality we applied \eqref{eq: identity KL divergence product} with $f=f^*\coloneqq\prod_{i=1}^k f_i^*$. Note that $f^*$ is a density on $\R^d$ because 
\begin{equation}\label{eq: density f^*}
    \int_{\R^d}f^*d\mu=\int_{\R^d}\prod_{i=1}^k f_i^*\,d\mu=\prod_{i=1}^k\int_{\R^{V_i}}f_i^*\,d\mu_{V_i}=1\,.
\end{equation}
Summing up, the inequality \eqref{eq: lower bound infimum KL} implies that
\[
\inf_{f\in\sF_G}d_\mathrm{KL}(f_0,f) \ge d_\mathrm{KL}(f_0,f^*)\,.
\]
Furthermore by definition $d_\mathrm{KL}(f_0,f^*) \ge \inf_{f\in\sF_G}d_\mathrm{KL}(f_0,f)$, hence the equality \eqref{eq: infimum product} is satisfied. The uniqueness of $f^*$ a.e. is immediate.

\ref{item: strong convergence}. The log-concave densities $f_i^*$ on $\R^{V_i}$ are such that the sequence of estimated densities $ (\hat{f}_{n,i})$ of $\sX_{V_i}=\{X^{(1)}_{V_i}, \ldots, X^{(n)}_{V_i}\}$ satisfy $\int_{\R^{V_i}}|\hat{f}_{n,i}-f_i^*|d\mu_{V_i} \stackrel{a.s.}{\to} 0$ by the base case.
We show by induction on the number of components $k$ that the estimates $\hat{f}_n=\prod_{i=1}^{k}\hat{f}_{n,i}$ form a minimizing sequence. Let $\hat{g}_n\coloneqq\prod_{i=1}^{k-1}\hat{f}_{n,i}$ and $g^*\coloneqq\prod_{i=1}^{k-1}f_i^*$. Then
\begin{align}
\begin{split}
\int_{\R^d}|\hat{f}_n-f^*|d\mu &= \int_{\R^d}|\hat{g}_n\hat{f}_{n,k} - g^*f_k^*|d\mu\\
&\le \int_{\R^d} \hat{f}_{n,k}|\hat{g}_n-g^*|+g^*|\hat{f}_{n,k}-f_k^*|d\mu\\
&= \int_{\R^{V\setminus V_k}}|\hat{g}_n-g^*|d\mu_{V\setminus V_k} + \int_{\R^{V_k}}|\hat{f}_{n,k}-f_k^*|d\mu_{V_k} \stackrel{a.s.}{\to} 0\,,
\end{split}
\end{align}
which implies that the sequence of measures given by $(\hat{f}_n)$ converges in distribution to the measure given by $f^*$.
From~\cite[Lemma 1]{cule2010theoretical}
we know that $a_0$ and $b_0$ exist, so we can apply~\cite[Proposition 2]{cule2010theoretical} to complete the proof.   
\end{proof}

\section{Convex decompositions of convex functions} \label{sec:decomposition-of-convex-functions}

Recall the model we have been studying is the family $\sF_G$ of densities $f$ that admit a log-concave factorization $f(x)=\prod_{C\in\sC(G)}\psi_C(x_C)$.
The initial idea of this project was to instead study the family $\tilde{\sF}_G$ of log-concave densities $f$ that factorize as $f(x) = \prod_{C\in\sC(G)}\psi_C(x_C)$.
The difference between $\sF_G$ and $\tilde{\sF}_G$ is that in the former, the clique potentials $\psi_C$ themselves are required to be log-concave, while in the latter only their product $f$ is required to be log-concave.
While it is clear that $\sF_G \subseteq\tilde{\sF}_G$, the reverse inclusion is much harder to tackle. In fact, we know for certain that if $G$ is {\em not chordal}, then $\sF_G$ is strictly contained in $\tilde{\sF}_G$. For every non-chordal graph $G=(V,E)$ with $V=[d]$ and set of maximal cliques $\sC(G)$, there exists a log-concave density $f$ in $\R^d$ that does not admit a factorization $f(x)=\frac{1}{Z}\prod_{C\in\sC(G)}\psi_C(x_C)$ with nonnegative log-concave potentials $\psi_C$.

\begin{example}\label{ex: no log-concave decomposition}
A canonical example is given in~\cite[Section 12.1, 382--383]{vandenberghe2015chordal} where $G$ is a $4$-cycle and
\[
f(x_1,x_2,x_3,x_4)=\frac{1}{Z}\exp[-\lambda(x_1^2+\dots+x_4^2)-2(x_1x_2+x_2x_3+x_3x_4-x_1x_4)]\,,
\]
where $Z=Z(\lambda)$ is a normalizing constant.
It turns out that, if $\sqrt{2}\le \lambda <2$, then $f$ is log-concave but does not admit a log-concave decomposition.
Furthermore, one may use the argument in ~\cite[Section 9.2, 342--344]{vandenberghe2015chordal} to extend this example to any $n$-cycle for $n\ge 4$.
\end{example}

\begin{remark}
The fact that $\sF_G\subsetneq\tilde{\sF}_G$ does not exclude the possibility that the MLE over $\tilde{\sF}_G$ belongs to $\sF_G$, a question that could be addressed in future work.
\end{remark}

\textbf{Related work.} There are several results regarding the cases where a function can be written as the sum of convex functions in some sense, each of which studies the problem from a different perspective.
In~\cite{pales2003approximately}, P{\'a}les considers  $(\varepsilon,\delta)$-convex functions, which are real valued functions $f$ defined on a real interval $I$ such that for all $x,y\in I$ and $t\in [0,1]$,
\[
f(tx+(1-t)y) \le tf(x) + (1-t) f(y) + \varepsilon t(1-t) |x-y| + \delta\,.
\]
\cite[Corollary 4]{pales2003approximately} proves that $f$ is an $(\varepsilon,\delta)$-convex function for some positive $\varepsilon$ and $\delta$ if and only if it can be written as a sum of a convex function, a function with bounded supremum norm, and a function with bounded Lipschitz modulus.

In \cite{rockafellar1970convex}, various properties have been proved for the sum of convex functions. For instance, \cite[Theorem 9.3]{rockafellar1970convex} states that the sum of closed proper convex functions is a closed proper convex function, \cite[Theorem 16.4]{rockafellar1970convex} shows that the operations of addition and infimal convolution of convex functions are dual to each other, \cite[Theorem 19.4]{rockafellar1970convex} proves that the sum of two polyhedral convex functions is polyhedral, and \cite[Theorem 19.4]{rockafellar1970convex} refines the formula in \cite[Theorem 16.4]{rockafellar1970convex} in the case where the convexity of some of the summands is polyhedral. Furthermore, \cite[Theorem 23.8]{rockafellar1970convex} concerns the subdifferential of the sum of some proper convex functions.

In \cite[Chapter 12]{vandenberghe2015chordal}, Vandenberghe and Andersen define a function $f\colon\R^d\to \R$ as {\em partially separable} if it can be expressed as 
\[
f(x)= \sum_{k=1}^\ell f_k(x_{\beta_k(1)},\ldots,x_{\beta_k(r_k)})\,,
\]
where the sets $\{\beta_k(1),\ldots,\beta_k(r_k)\}$ are (small) index sets in $[d]$. Partial separability, and an extension called {\em group partial separability}, are applied to different fields such as nonlinear optimization modeling software, quasi-Newton and trust-region algorithms, and sparse semidefinite relaxations of polynomial optimization problems.

\smallskip
\textbf{Results.} While we have not been able to show that the families $\tilde{\sF}_G$ and $\sF_G$ are equal when $G$ is a chordal graph, in the rest of this section we give several sufficient conditions under which a function $f\in\tilde{\sF}_G$ also lies in $\sF_G$.

We begin our discussion with Proposition \ref{prop: convex decomposition iff theta_i}, which gives an equivalent condition for an element of $\tilde{\sF}_G$ lying in $\sF_G$ when $G$ is chordal. For such a graph $G$, Corollary \ref{corollary: Gaussian log-concave factorization} then proves that any Gaussian density function in the graphical model of $G$ also lies in $\sF_G$.

The next result in this section is Theorem \ref{thm: sufficient condition for convex decomposition on chordal graphs}, which proposes a different sufficient condition for a function $f\in\tilde{\sF}_G$ also lying in $\sF_G$ when $G$ is a chordal graph.

Finally, Theorem \ref{thm: convex decomposition} concerns the case where $G$ is chosen from a proper subset of chordal graphs. This result considers the log-concave densities $f$ that factorize according to the graph $G$ with twice continuously differentiable clique potentials, and gives a sufficient condition for $f$ having a log-concave factorization according to $G$ with twice continuously differentiable clique potentials.
Figure \ref{fig: section 8 map} in Appendix \ref{app: section decomposition convex} illustrates a road map of Section \ref{sec:decomposition-of-convex-functions}.

\smallskip
\textbf{Notation.}
Consider a twice differentiable function $f\colon\R^d\to\R$. Let $A$ and $B$ be two subsets of $[d]$.
We denote the Hessian matrix of $f$ with respect to $x_A$ and $x_B$ by ${H_f}|_{A,B}$, i.e.
\[
{H_f}|_{A,B} \coloneqq \left(\frac{\partial^2 f}{\partial x_i \partial x_j} \right)_{i\in A,\,j\in B}\,.
\]
In particular, when $A=B$ we use the shorthand ${H_f}|_{A}\coloneqq {H_f}|_{A,A}$\,.

Let $G$ be a graph with set of maximal cliques $\sC(G)=\{C_1,\ldots,C_t\}$. Given indices $p\le q$ in $[t]$, we use the shorthand ${C}_{[p,q]}$ to denote $\bigcup_{j=p}^q C_j$.

\smallskip
Proposition \ref{prop: chordal if and only if the running intersection property} gives an equivalent definition of chordal graphs, which concerns an ordering of the maximal cliques in the graphs with a specific property.

\begin{definition}\label{running intersection property}
Let $G=(V,E)$ be a graph with $t$ maximal cliques. We say that an ordering $C_1,\ldots,C_t$ of the maximal cliques of $G$ satisfies the {\em \theproperty} if for all $i\in[t-1]$, there exists $j\in\{i+1,\ldots,t\}$ such that $C_i \cap C_{[i+1,t]}\subseteq C_j$.
\end{definition}

The running intersection property introduced in Definition \ref{running intersection property} is slightly different from \cite{blair1993clique}. Proposition \ref{prop: chordal if and only if the running intersection property} is immediate by applying \cite[Theorem 3.4]{blair1993clique} to all the connected components of a given graph $G$.

\begin{proposition}\label{prop: chordal if and only if the running intersection property}
Let $G$ be a graph with $|\sC(G)|=t$. Then $G$ is chordal if and only if there exists an ordering $C_1,\ldots,C_t$ of $\sC(G)$ that satisfies the \theproperty.
\end{proposition}

\begin{example}
Figure \ref{fig:chordal graph 1} shows a chordal graph with six maximal cliques. The following ordering of these maximal cliques satisfies the \theproperty:
\[
\begin{gathered}
C_1 = \{1,2,4\}\,,\ C_2 = \{2,3,5\}\,,\ C_3 = \{4,6,7\}\,,\\
C_4 = \{5,7,8\}\,,\ C_5 = \{2,4,5\}\,,\ C_6 = \{4,5,7\}\,.
\end{gathered}
\]
\begin{figure}[ht]
\centering
    \begin{subfigure}[t]{.5\textwidth}
    \centering
    \includegraphics[scale=0.35]{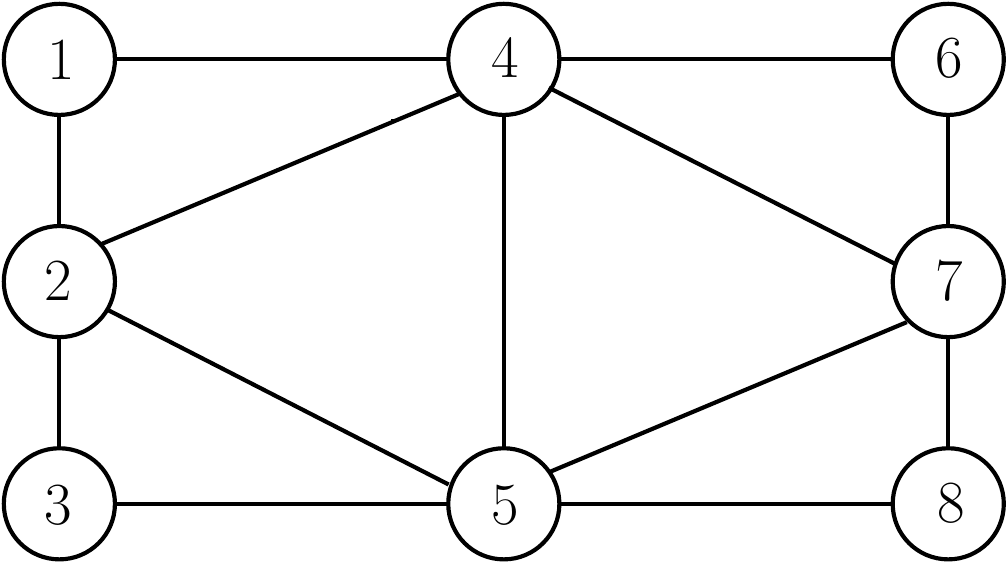}
    \caption{}
    \label{fig:chordal graph 1}
    \end{subfigure}%
    \begin{subfigure}[t]{.5\textwidth}
    \centering
    \includegraphics[scale=0.35]{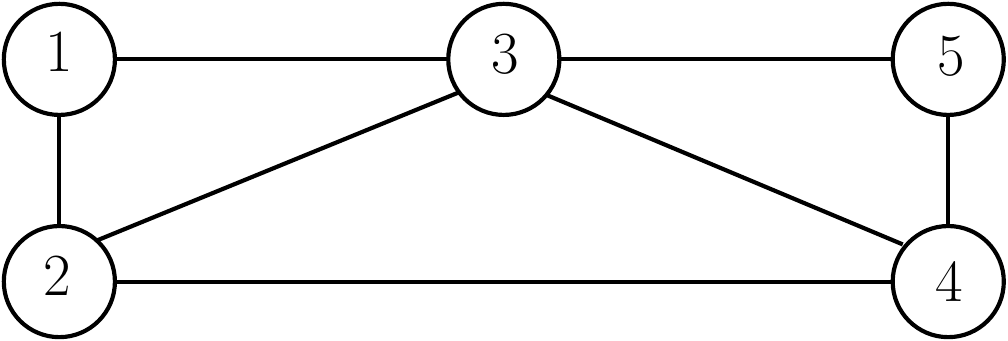}
    \caption{}
    \label{fig:chordal graph 2}
    \end{subfigure}
    \caption{Instances of chordal graphs.}
\end{figure}
\end{example}

\begin{proposition}\label{prop: convex decomposition iff theta_i}
Let $G$ be a chordal graph with $d$ vertices and maximal cliques $C_1,\ldots,C_t$ satisfying the \theproperty. Consider a convex function $f\colon\R^d\to\R$ of the form $f=\sum_{i=1}^t g_i$ for some functions $g_i\colon\R^{C_i}\to\R$ for all $i\in[t]$.
The following are equivalent:
\begin{itemize}
    \item[$(i)$] There exist convex functions $\hat{g}_i\colon\R^{C_i}\to\R$ for all $i\in[t]$ such that $f=\sum_{i=1}^t\hat{g}_i$.
    \item[$(ii)$] There exist functions $\theta_i \colon\R^{C_i \cap C_{[i+1,t]}}\to \R$ for all $i\in [t-1]$ such that the functions
    \begin{equation}
    g_i - \theta_i + \sum_{j\in A_i} \theta_j,\quad\forall\,i\in[t-1]
    \end{equation}
    and the function $g_t + \sum_{j\in A_t} \theta_j$ are convex, where
    \begin{equation}
    A_i\coloneqq \left\{j\in[i-1]\ \bigg|\ \min_{k \in \{j+1, \ldots, t\}} \{C_j \cap C_{[j+1,t]}\subseteq C_k\}=i\right\}\quad\forall\,i\in[t]\,.
    \end{equation}
\end{itemize}
\end{proposition}

\begin{lemma}\label{lem: sep matrix}
Let $M\in \R^{d\times d}$ be a symmetric positive semidefinite matrix and $\sP=(A,B,C)$ any partition of $[d]$.
Assume that $M$ has the block structure 
\begin{equation}\label{eq: block structure of a psd matrix}
M=
\begin{bmatrix}
M_{A,A} & M_{A,B} & 0 \\
M_{A,B}^T & M_{B,B} & M_{B,C} \\
0 & M_{B,C}^T & M_{C,C}
\end{bmatrix}\,.
\end{equation}
Then for any symmetric matrix $N\in \R^{B\times B}$, there exists a matrix $S_{M,N,\sP}\in \R^{B \times B}$ such that
\begin{align} \label{eq: Gaussian separation matrices}
\begin{bmatrix}
M_{A,A} & M_{A,B} \\
M_{A,B}^T & N-S_{M,N,\sP}
\end{bmatrix}\succcurlyeq 0
\quad\text{and}\quad
\begin{bmatrix}
M_{B,B}-N+S_{M,N,\sP} & M_{B,C} \\
M_{B,C}^T & M_{C,C}
\end{bmatrix}\succcurlyeq 0\,.
\end{align}
\end{lemma}

\begin{lemma}\label{lem: assumption twice differentiable}
Let $G=(V, E)$ be a graph with $d$ vertices and maximal cliques $C_1,\ldots,C_t$. Consider a function $f\colon\R^d\to\R$ defined as $f(x)=(x-\mu)^T K (x-\mu)$, where $\mu\in \R^d$ and $K\in \R^{d \times d}$ is a symmetric positive semidefinite matrix. If $f(x)=\sum_{i=1}^t g_i(x_{C_i})$ for some functions $g_i\colon\R^{C_i}\to\R$ for all $x\in \R^d$, then the functions $g_i$ can be chosen to be twice differentiable.
\end{lemma}

\begin{corollary} \label{corollary: Gaussian convex decomposition}
Let $G=(V, E)$ be a chordal graph with $d$ vertices and maximal cliques $C_1,\ldots,C_t$ satisfying the \theproperty. Also, assume that $f\colon\R^d\to\R$ is defined as $f(x)=(x-\mu)^T K (x-\mu)$, where $\mu\in \R^d$ and $K\in \R^{d \times d}$ is a symmetric positive semidefinite matrix. If there exist functions $g_i\colon\R^{C_i}\to\R$ for all $i\in[t]$ such that $f(x)=\sum_{i=1}^t g_i(x_{C_i})$ for all $x\in \R^d$, then there exist convex functions $\hat{g}_i\colon\R^{C_i}\to\R$ for all $i\in [t]$ such that $f(x)=\sum_{i=1}^t\hat{g}_i(x_{C_i})$ for all $x\in \R^d$.
\end{corollary}

\begin{corollary} \label{corollary: Gaussian log-concave factorization}
Let $G$ be a chordal graph. If a Gaussian distribution is in the graphical model corresponding to $G$, then there exists a factorization of this distribution according to $G$ such that the potential functions are log-concave.
\end{corollary}

\begin{example}\label{ex: convex decomposition}
The graph in Figure \ref{fig:chordal graph 2} is a chordal graph with maximal cliques $C_1=\{1,2,3\}$, $C_2=\{2,3,4\}$, and $C_3=\{3,4,5\}$. Note that $C_1 \cap C_{[2,3]}\subseteq C_2$. 
Consider the function $f\colon\R^5 \to \R$ defined by $f=g_1+g_2+g_3$, where
\begin{align}
\begin{split}
g_1(x_1,x_2,x_3) &\coloneqq 5x_1^2-18 x_1 x_2+6 x_1 x_3+19 x_2 ^2-12x_2 x_3+3 x_3^2\\
g_2(x_2,x_3,x_4) &\coloneqq 4x_4^2+ 2 x_2 x_4 -4 x_3 x_4\\
g_3(x_3,x_4,x_5) &\coloneqq 4x_5^2  + 2 x_3 x_5  - 6 x_4 x_5\,.
\end{split}
\end{align}
Then $f(x_1,x_2,x_3,x_4,x_5) = (x_1,x_2,x_3,x_4,x_5) K (x_1,x_2,x_3,x_4,x_5)^T$, where
\begin{align*}
K = 
\begin{bmatrix}
5 & -9 & 3 & 0 & 0 \\
-9 & 19 & -6 & 1 & 0 \\
3 & -6 & 3 & -2 & 1 \\
0 & 1 & -2 & 4 & -3 \\
0 & 0 & 1 & -3 & 4
\end{bmatrix}\,,
\end{align*}
which is a positive definite matrix. So, $f$ is a convex function. Consider the functions 
\begin{align*}
\hat{g}_1  (x_1,x_2,x_3) &\coloneqq 
5x_1^2-18 x_1 x_2+6 x_1 x_3+\frac{81}{5} x_2 ^2-\frac{54}{5}x_2 x_3+\frac{9}{5} x_3^2\\
\hat{g}_2  (x_2,x_3,x_4) &\coloneqq
 \frac{14}{5} x_2^2  - \frac{6}{5} x_2 x_3 + 2 x_2 x_4+ \frac{9}{70} x_3^2 -\frac{3}{7} x_3 x_4 + \frac{5}{14} x_4^2\\
\hat{g}_3(x_3,x_4,x_5) &\coloneqq
 \frac{15}{14} x_3^2  - \frac{25}{7} x_3 x_4 + 2 x_3 x_5  + \frac {51}{14} x_4^2 - 6 x_4 x_5  + 4x_5^2 \end{align*}
Then the functions $\hat{g}_1$, $\hat{g}_2$, and $\hat{g}_3$
are convex given that each of their Hessian matrices is positive semidefinite. Also, $f(x) =\hat{g}_1  (x_1,x_2,x_3) + \hat{g}_2  (x_2,x_3,x_4) + \hat{g}_3(x_3,x_4,x_5)$ for all $x\in\R^5$.
\end{example}

We proceed to state a sufficient condition for the existence of a convex decomposition for the convex functions which can be decomposed according to the maximal cliques of a chordal graph.
\begin{theorem} \label{thm: sufficient condition for convex decomposition on chordal graphs}
Let $G$ be a chordal graph with $d$ vertices and maximal cliques $C_1,\ldots,C_t$ satisfying the \theproperty.
Consider a convex function $f\colon\R^d\to\R$ of the form $f(x)=\sum_{i=1}^t g_i(x_{C_i})$ for some functions $g_i\colon\R^{C_i}\to\R$ for all $i\in[t]$. 
If for all $i\in[t-1]$, there exist vectors $a^{(i)}\in\R^{C_i\setminus C_{[i+1,t]}}$ and $b^{(i)}\in \R^{C_{[i+1,t]}\setminus C_i}$ such that the function
\[
f(a^{(1)},\ldots,a^{(i-1)},x_{C_i\setminus C_{[i+1,t]}}, x_{C_i\cap C_{[i+1,t]}}, b^{(i)}) - f(a^{(1)}, \ldots, a^{(i-1)}, a^{(i)}, x_{C_i\cap C_{[i+1,t]}}, b^{(i)})
\]
is convex on $\R^{C_i}$, then there exist convex functions $\hat{g}_i\colon\R^{C_i}\to \R$ for all $i\in[t]$ such that $f=\sum_{i=1}^t\hat{g}_i$.
\end{theorem}

\begin{proof}
We claim that for all $j\in [t-1]$ and $i\in[t]$, there exist functions $g^{(j)}_i\colon\R^{C_i}\to \R$ such that 
\begin{enumerate}
    \item the functions $g^{(j)}_i$ are convex,
    \item for all $x_{C_{[j+1,t]}}\in\R^{C_{[j+1,t]}}$, $\sum_{i=j+1}^t g^{(j)}_i (x_{C_i}) = f(a^{(1)},\ldots,a^{(j)},x_{C_{[j+1,t]}})$, and
    \item for all $x\in \R^d$, $f(x)=\sum_{i=1}^t g^{(j)}_i(x_{C_i})$.
\end{enumerate}
The statement follows by applying the claim for $j=t-1$ and setting $\hat{g}_i\coloneqq g^{(t-1)}_i$ for all $i\in[t]$.

Let $j=1$. By the running intersection property, there exists $j_0\in\{2,\ldots,t\}$ such that $C_1 \cap C_{[2,t]} \subseteq C_{j_0}$. Define
\begin{align*}
g^{(1)}_i(x_{C_i}) \coloneqq
\begin{cases}
f(x_{C_1}, b^{(1)}) - f(a^{(1)}, x_{C_1 \cap C_{[2,t]}}, b^{(1)}) & \text{if $i=1$}\\
g_{j_0}(x_{C_{j_0}}) + g_1(a^{(1)}, x_{C_1\cap C_{[2,t]}}) & \text{if $i=j_0$}\\
g_j(x_{C_i}) & \text{if $i\in [t]\setminus\{1,j_0\}$}\,.
\end{cases}
\end{align*}
Then 
$g^{(1)}_1$ is convex by assumption. Also, for all $x_{C_{[2,t]}}\in \R^{C_{[2,t]}}$,
\begin{align*}
\sum_{i=2}^t g^{(1)}_i(x_{C_i}) = \sum_{i=2}^t g_i(x_{C_i}) + g_1(a^{(1)}, x_{C_1\cap C_{[2,t]}}) = f(a^{(1)}, x_{C_{[2,t]}})\,.
\end{align*}
Additionally, for all $x\in \R^d$,
\begin{align*}
\sum_{i=1}^t g^{(1)}_i(x_{C_i}) &= f(x_{C_1}, b^{(1)}) - f(a^{(1)}, x_{C_1 \cap C_{[2,t]}}, b^{(1)})\\
&\quad +g_2(x_{C_2}) + \cdots + g_{j_0-1}(x_{C_{j_0-1}}) + g_{j_0+1}(x_{C_{j_0+1}}) + \cdots + g_t(x_{C_t})\\
&\quad +g_{j_0}(x_{C_{j_0}}) + g_1(a^{(1)}, x_{C_1\cap C_{[2,t]}})\\
&= g_1(x_{C_1}) + g_2(x_{C_1 \cap C_2},b^{(1)}_{C_2\setminus C_1}) + \cdots + g_t(x_{C_1 \cap C_t},b^{(1)}_{C_t\setminus C_1})\\
&\quad -g_1(a^{(1)}, x_{C_1 \cap C_{[2,t]}}) - g_2(x_{C_1 \cap C_2},b^{(1)}_{C_2\setminus C_1}) - \cdots - g_t(x_{C_1 \cap C_t},b^{(1)}_{C_t\setminus C_1})\\
&\quad +g_2(x_{C_2}) +  \cdots + g_t(x_{C_t}) + g_1(a^{(1)}, x_{C_1\cap C_{[2,t]}}) = f(x)\,.
\end{align*}

Now assume that the claim holds for $j=k$, where $k\in [t-2]$.
Let $j=k+1$. There exists $j_0\in \{k+2,\ldots,t\}$ such that $C_{k+1} \cap C_{[k+2,t]} \subseteq C_{j_0}$. Set
\begin{align*}
g^{(k+1)}_i(x_{C_i}) &\coloneqq
\begin{cases}
f(a^{(1)}, \ldots, a^{(k)}, x_{C_{k+1}}, b^{(k+1)})\\
- f(a^{(1)},\ldots,a^{(k+1)}, x_{C_{k+1} \cap C_{[k+2,t]}}, b^{(k+1)})&\text{if $i=k+1$}\\
g^{(k)}_{j_0}(x_{C_{j_0}}) + g^{(k)}_{k+1}(a^{(k+1)}, x_{C_{k+1}\cap C_{[k+2,t]}})&\text{if $i=j_0$}\\
g^{(k)}_i(x_{C_i})&\text{if $i\in [t]\setminus \{k+1,j_0\}$}\,.
\end{cases}
\end{align*}
Then each of the functions $g^{(k+1)}_1,\ldots,g^{(k+1)}_{k+1}$ is convex using the assumption as well as the induction hypothesis. Also, for all $x_{C_{[k+2,t]}}\in \R^{C_{[k+2,t]}}$,
\begin{align*}
g^{(k+1)}_{k+2}(x_{C_{k+2}}) + \cdots + g^{(k+1)}_n (x_{C_n}) &= g^{(k)}_{k+2}(x_{C_{k+2}}) + \cdots + g^{(k)}_n (x_{C_n})\\
&\quad +g^{(k)}_{k+1}(a^{(k+1)}, x_{C_{k+1}\cap C_{[k+2,t]}})\\
&= f(a^{(1)},\ldots,a^{(k)},a^{(k+1)}, x_{C_{[k+2,t]}})\,,
\end{align*}
where in the last equality we applied the induction hypothesis. Finally, for all $x\in \R^d$,
\begin{align*}
\sum_{i=1}^t g^{(k+1)}_i(x_{C_i})  &= f(a^{(1)},\ldots,a^{(k)},x_{C_{k+1}}, b^{(k+1)})\\
&\quad -f(a^{(1)},\ldots,a^{(k)},a^{(k+1)}, x_{C_{k+1} \cap C_{[k+2,t]}}, b^{(k+1)})\\
&\quad +\sum_{i=1}^k g^{(k)}_{i}(x_{C_i}) + g^{(k)}_{k+1}(a^{(k+1)}, x_{C_{k+1}\cap C_{[k+2,t]}}) + \sum_{i=k+2}^t g^{(k)}_{i}(x_{C_i})\\
&= g^{(k)}_{k+1}(x_{C_{k+1}}) + \sum_{i=k+2}^t g^{(k)}_{i}(x_{C_{k+1} \cap C_i},b^{(k+1)}_{C_i\setminus C_{k+1}})\\
&\quad -g^{(k)}_1(a^{(k+1)}, x_{C_{k+1} \cap C_{[k+2,t]}})-\sum_{i=k+2}^t g^{(k)}_{i}(x_{C_{k+1} \cap C_i},b^{(k+1)}_{C_i\setminus C_{k+1}})\\
&\quad +\sum_{i=1}^k g^{(k)}_i(x_{C_i}) + g^{(k)}_{k+1}(a^{(k+1)}, x_{C_{k+1}\cap C_{[k+2,t]}})+\sum_{i=k+2}^t g^{(k)}_{i} (x_{C_i})\\
&= f(x)\,.
\end{align*}
This proves the claim and completes the proof.
\end{proof}

\begin{remark}
The sufficient condition mentioned in Theorem \ref{thm: sufficient condition for convex decomposition on chordal graphs} is not satisfied by all functions which can be written as a sum of convex functions according to the maximal cliques of a chordal graph.
For instance, consider the chordal graph in Figure \ref{fig:simple chordal graph}, which has maximal cliques $C_1=\{1,2\}$ and $C_2=\{2,3\}$.
Let $f(x)=g_1(x_1,x_2)+g_2(x_2,x_3)$, where $g_1(x_1,x_2) = (x_1+x_2)^2$ and $g_2(x_2,x_3)=x_3^2$. Although both of the functions $g_1$ and $g_2$ are convex, there does not exist $a^{(1)}\in \R$ and $b^{(1)}\in \R$ such that the difference
\begin{equation}\label{eq: difference f example}
f(x_1,x_2,b^{(1)})- f(a^{(1)},x_2,b^{(1)}) = (x_1+x_2)^2- (a^{(1)} +x_2)^2
\end{equation}
is convex on $\R^2$. Indeed, for all $a^{(1)}\in \R$ and $b^{(1)}\in \R$ the Hessian matrix of the function at the right-hand side of \eqref{eq: difference f example} has eigenvalues $1\pm\sqrt{5}$, in particular it is not positive semidefinite.
\end{remark}

In the rest of the section, we consider graphs which satisfy a more restrictive property than the \theproperty, and study the convex functions expressed as a sum of twice continuously differentiable functions according to the maximal cliques of one of these graph.
\begin{proposition}\label{prop2}
Let $A$, $B$ and $C$ be three mutually disjoint subsets of $[d]$, and $X = Y \times Z$, where $Y\subseteq \R^{A}$ is a compact set and $Z\subseteq \R^{B\cup C}$ is any set. Let $f\colon\R^{A\cup B\cup C}\to \R$, $g\colon\R^{A\cup B}\to \R$, and $h\colon\R^{B\cup C}\to \R$ be three functions such that
\begin{enumerate}
    \item $f(x_A, x_B, x_C)=g(x_A, x_B)+h(x_B,x_C)$ for all $(x_A, x_B, x_C)\in X$,
    \item $f$ is convex on $X$,
    \item $g\in \sC^2(\pi_{A\cup B}(X))$ and $h\in \sC^2(\pi_{B\cup C}(X))$,
    \item ${H_f}|_{A}$ is invertible on $X$, and
    \item $|B|=1$.
\end{enumerate}
Then
\begin{itemize}
    \item[$(i)$] there exists a function $\gamma\in \sC^2(\pi_B(X))$ such that the functions $\hat{g}=g+\gamma\in C^2\left(\pi_{A\cup B}(X)\right)$ and $\hat{h}=h-\gamma\in C^2\left(\pi_{B\cup C}(X)\right)$ are convex on $X$.
    \item[$(ii)$] if $D\subseteq C$ and ${H_f}|_{A\cup B\cup D}$ is invertible on $X$, then ${H_{\hat{h}}}|_{B\cup D}$ is invertible on $X$.
\end{itemize}
\end{proposition}

\begin{theorem}\label{thm: convex decomposition}
Let $G$ be a graph with $d$ vertices and set of maximal cliques $\sC(G)=\{C_1,\ldots,C_t\}$. Consider a convex function $f\colon\R^d\to\R$ of the form $f(x)=\sum_{i=1}^t g_i(x_{C_i})$ for some functions $g_i\in \sC^2(\R^{C_i})$ for all $i\in[t]$.
Assume that there exist a reordering of $\sC(G)$, compact sets $X_i\subseteq \R^{C_i\setminus{C}_{[i+1,t]}}$ for all $i\in[t-1]$, and some set $X_t\subseteq\R^{C_t}$ such that
\begin{enumerate}
\item for all $i\le t-1$, each maximal clique $C_i$ shares at most one vertex with all the maximal cliques with larger indices, i.e. $|C_i\cap{C}_{[i+1,t]}|\le 1$, and
\item the Hessian ${H_{f}}|_{{C}_{[1,t-1]}}$ is invertible on $X = X_1\times\cdots\times X_t$.
\end{enumerate} 
Then there exist convex functions $\hat{g}_i\in \sC^2(\pi_{C_i}(X))$ for all $i\in[t]$ such that $f(x)=\sum_{i=1}^t\hat{g}_i(x_{C_i})$ for all $x\in X$.
\end{theorem}
\begin{proof}
The proof goes by induction. In particular, we claim that for all $j\in[t-1]$, there exist functions $\hat{g}^{(j)}_i\in \sC^2(\pi_{C_i}(X))$ for all $i\in[t]$ such that
\begin{enumerate}
    \item $\hat{g}^{(j)}_1,\ldots,\hat{g}^{(j)}_j$ and $\hat{f}_j = \sum_{i=j+1}^t \hat{g}^{(j)}_i$ are convex,
    \item $f(x)=\sum_{i=1}^t \hat{g}^{(j)}_i(x_{C_i})$ for all $x\in X$, and
    \item if $j\le t-2$, the Hessian ${H_{\hat{f}_j}}|_{{C}_{[j+1,t-1]}}$ is invertible on $X$.
\end{enumerate}
\noindent Let $j=1$. Given the assumptions, for all $x\in X$,
\begin{equation}
f(x) = g_1(x_{C_1\setminus{C}_{[2,t]}}, x_{C_1\cap{C}_{[2,t]}}) + h(x_{C_1\cap{C}_{[2,t]}}, x_{{C}_{[2,t]}\setminus C_1})\,,
\end{equation}
where $h(x_{C_1 \cap {C}_{[2,t]}}, x_{{C}_{[2,t]}\setminus C_1}) \coloneqq\sum_{i=2}^t g_i(x_{C_i})\in \sC^2(\pi_{{C}_{[2,t]}}(X))$.
Also, since ${H_{f}}|_{{C}_{[1,t-1]}}$ is invertible and positive semidefinite on $X$, it is positive definite on $X$. So, ${H_{f}}|_{C_1\setminus{C}_{[2,t]}}$ is also positive definite, and thus, invertible on $X$.
Hence, by Proposition \ref{prop2}$(i)$, there exists a function $\gamma_0\in \sC^2(\pi_{C_1 \cap{C}_{[2,t]}}(X))$ such that $g_1+\gamma_0$ and $h-\gamma_0$ are convex. Let $i_0\in\{2,\ldots,t\}$ be an index such that $C_1\cap{C}_{[2,t]}\subseteq C_{i_0}$. Then it suffices to define
\begin{equation}\label{eq: def hat g 1 i}
    \hat{g}^{(1)}_1 \coloneqq g_1+\gamma_0\,,\quad \hat{g}^{(1)}_{i_0}\coloneqq g_{i_0}-\gamma_0\,,\quad \hat{g}^{(1)}_{i} \coloneqq g_i \text{ for $i\in[t]\setminus\{1,i_0\}$}\,.
\end{equation}
Setting $\hat{f}_1 = \sum_{i=2}^t \hat{g}^{(1)}_i$, by Proposition \ref{prop2}$(ii)$ the Hessian ${H_{\hat{f}_1}}|_{{C}_{[2,t-1]}}$ is invertible on $X$.

Now assume that the claim holds for $j=k$, where $k\in[t-2]$.
In particular, there exist functions $\hat{g}^{(k)}_i\in \sC^2(\pi_{C_i}(X))$ for all $i\in[t]$ such that
\begin{enumerate}
    \item $\hat{g}^{(k)}_1,\ldots,\hat{g}^{(k)}_k$ and $\hat{f}_k = \sum_{i=k+1}^t \hat{g}^{(k)}_i$ are convex,
    \item $f(x)=\sum_{i=1}^t\hat{g}^{(k)}_i(x_{C_i})$ for all $x\in X$, and
    \item the Hessian ${H_{\hat{f}_k}}|_{{C}_{[k+1,t-1]}}$ is invertible on $X$.
\end{enumerate}
So, for all $x_{{C}_{[k+1,t]}}\in\pi_{{C}_{[k+1,t]}}(X)$,
\begin{equation}
\hat{f}_k(x_{{C}_{[k+1,t]}}) = \hat{g}^{(k)}_{k+1}(x_{C_{k+1}\setminus{C}_{[k+2,t]}}, x_{C_{k+1} \cap {C}_{[k+2,t]}}) + h(x_{C_{k+1} \cap {C}_{[k+2,t]}}, x_{{C}_{[k+2,t]}\setminus C_{k+1}})\,,
\end{equation}
where
\[
h(x_{C_{k+1} \cap {C}_{[k+2,t]}}, x_{{C}_{[k+2,t]}\setminus C_{k+1}}) \coloneqq\sum_{i=k+2}^t \hat{g}^{(k)}_i(x_{C_i})\in \sC^2(\pi_{{C}_{[k+2,t]}}(X))\,.
\]
Also, since ${H_{\hat{f}_k}}|_{{C}_{[k+1,t-1]}}$ is invertible and positive semidefinite on $X$, ${H_{\hat{f}_k}}|_{{C}_{[k+1,t-1]}}$ is positive definite on $X$.
So ${H_{\hat{f}_k}}|_{C_{k+1}\setminus{C}_{[k+2,t]}}$ is positive definite as well, in particular it is invertible on $X$. Hence, by Proposition \ref{prop2}$(i)$, there exists a function $\gamma_k\in \sC^2(\pi_{C_{k+1} \cap {C}_{[k+2,t]}}(X))$ such that $\hat{g}^{(k)}_{k+1}+\gamma_k$ and $h-\gamma_k$ are convex. Let $i_0\in\{k+2,\ldots,t\}$ such that $C_{k+1} \cap{C}_{[k+2,t]} \subseteq C_{i_0}$. Then  we define
\begin{equation}\label{eq: def hat g k+1 i}
\hat{g}^{(k+1)}_{k+1} \coloneqq \hat{g}^{(k)}_{k+1} +\gamma_k\,,\quad \hat{g}^{(k+1)}_{i_0} \coloneqq \hat{g}^{(k)}_{i_0}-\gamma_k\,,\quad\hat{g}^{(k+1)}_{i}\coloneqq\hat{g}^{(k)}_i\ \text{if $i\in[t]\setminus\{k+1,i_0\}$}\,.
\end{equation}
Moreover, by Proposition \ref{prop2}$(ii)$, ${H_{\hat{f}_{k+1}}}|_{{C}_{[k+2,t-1]}}$ is invertible on $X$ if $k\le t-3$. So, the claim is correct, and the proof is concluded by applying it to $j=t-1$.
\end{proof}

\begin{example}
The graphs in Figure \ref{fig:convex decomposition graphs} are instances of graphs whose maximal cliques can be ordered such that each maximal clique has at most one vertex in common with all the maximal cliques with larger indices. In such graphs, each cycle must lie in one of the cliques. 
\begin{figure}[ht]
\centering
    \begin{subfigure}[t]{.5\textwidth}
    \centering
     \includegraphics[scale=0.4]{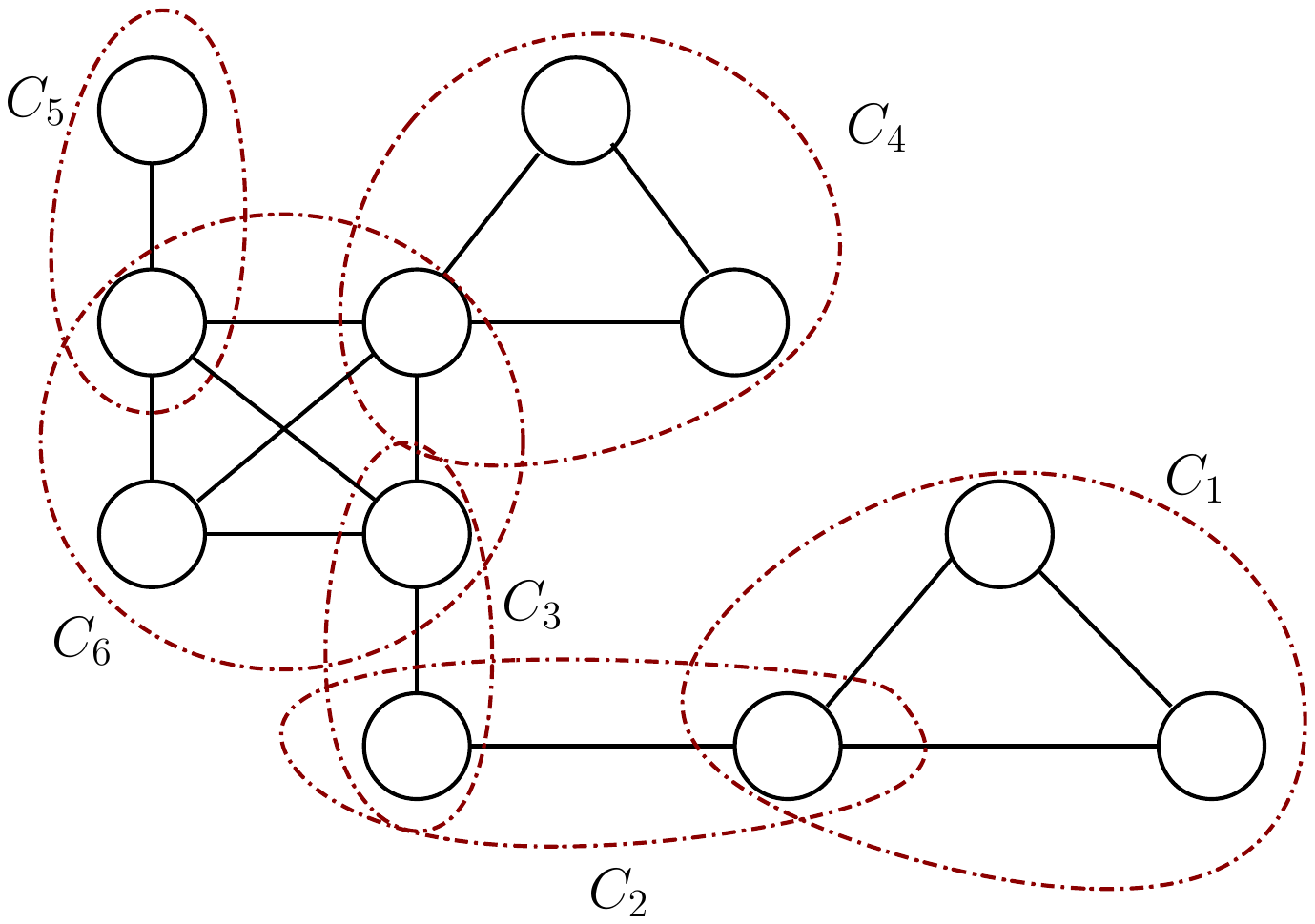}
     \caption{}
    \end{subfigure}%
    \begin{subfigure}[t]{.5\textwidth}
    \centering
    \includegraphics[scale=0.4]{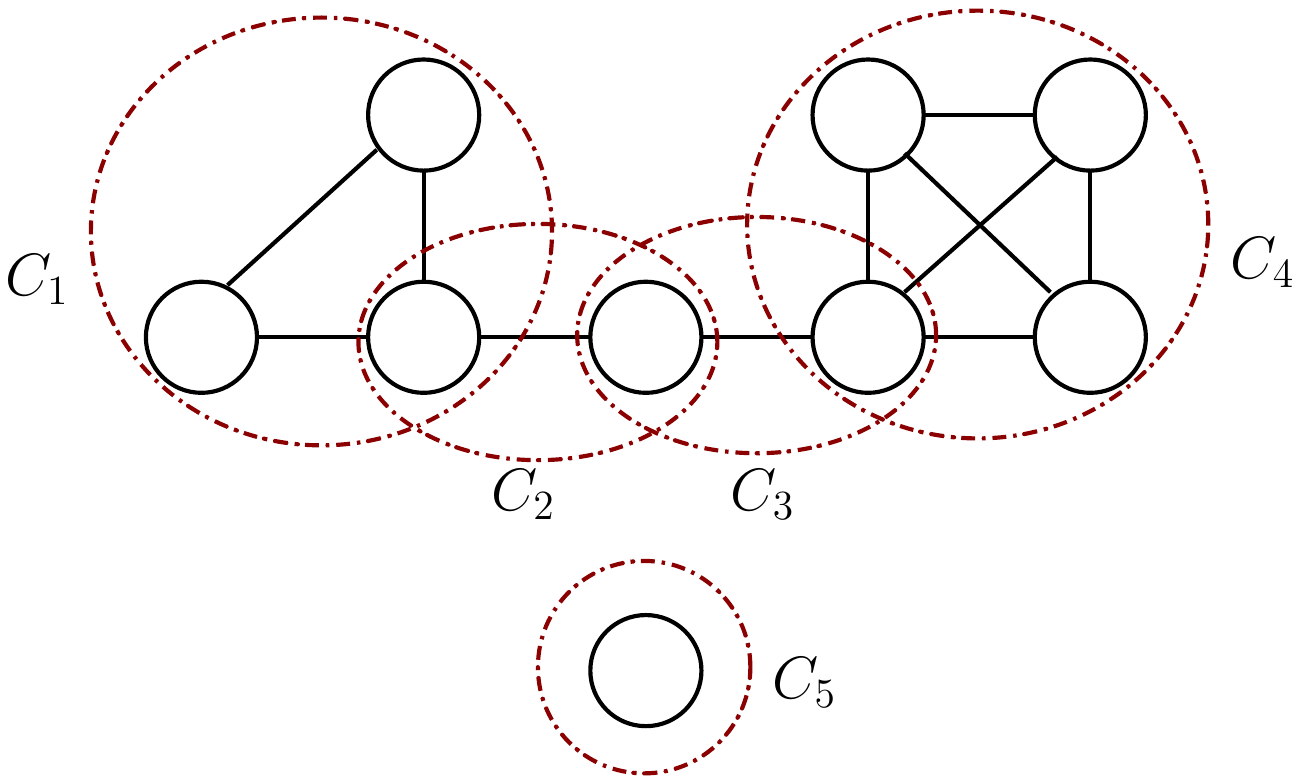}
    \caption{}
    \end{subfigure}
    \caption{Instances of graphs whose maximal cliques can be ordered such that each maximal clique has at most one vertex in common with maximal cliques with higher indices.}
    \label{fig:convex decomposition graphs}
\end{figure}
\end{example}

\section*{Acknowledgements}

We acknowledge the computational resources provided by the Aalto Science-IT project. Kaie Kubjas, Olga Kuznetsova and Luca Sodomaco were partially supported by the Academy of Finland Grant No. 323416.
Elina Robeva and Pardis Semnani were supported by NSERC Discovery Grant DGECR-2020-00338.

\bibliographystyle{alpha}
\bibliography{bibliography}

\appendix

\section{Notation and definitions from convex geometry}\label{app:background}

\begin{definition}\label{def: n-cycle, chordal}
Consider and undirected graph $G=(V,E)$, with vertex set $V$ and edge set $E$.
Given a subset $K\subseteq V$, the {\em induced subgraph} of $G$ with respect to $K$ is the graph with vertices in $K$ and edges in $\{\{i,j\}\mid\text{$i,j\in K$ and $\{i,j\}\in E$}\}$.
A sequence of vertices $v_1,\ldots,v_n$ of $G$ is an {\em $n$-cycle} if $v_1=v_n$, there are no other repeated vertices in the sequence, and for all $i\in[n-1]$ the edge $(v_i,v_{i+1})$ belongs to $E$.
A cycle in a graph $G$ is {\em induced} if the induced subgraph on the
vertex set of the cycle consists of exactly the edges in the cycle.
\end{definition}

\begin{definition}\label{def: chordal graph}
A graph $G$ is {\em chordal} (or {\em decomposable}) if  there are no induced $n$-cycles
with $n\ge 4$. Given a graph $G = (V, E)$, a graph $\tilde{G} = (V, \tilde{E})$ is a {\em chordal cover of $G$} if $\tilde{G}$ is chordal and $\tilde{E}\supset E$.
\end{definition}

\begin{definition}\label{def: subdivision}
A {\em polytope} of $\R^d$ is the convex hull $P=\conv(Z)$ of a set of finitely many points $Z\subseteq\R^d$. If $S$ is a convex set in $\R^d$, then a {\em supporting half-space} to $S$ is a closed half-space which contains $S$ and has a point of $S$ in its boundary. A {\em supporting hyperplane} $H$ to $S$ is a hyperplane which is the boundary of a supporting half-space to $S$.
If $Z$ is a finite set of points in $\R^d$ such that $P=\conv(Z)$ is a $d$-dimensional polytope in $\R^d$, then a {\em face} of $P$ is a set of the form $P\cap H$, where $H$ is a supporting hyperplane to $P$. In particular, the {\em vertex set} of $P$ is the set of zero-dimensional faces of $P$, namely the {\em vertices} of $P$. A {\em subdivision} of $P$ is a finite set $\sigma=\{Q_1,\ldots,Q_t\}$ of $d$-dimensional polytopes called {\em cells} such that
\begin{enumerate}
    \item $P=Q_1\cup\cdots\cup Q_t$
    \item the intersection of any two distinct cells is a face of both of them.
\end{enumerate}
A subdivision $\sigma$ of $P$ is {\em regular} if there exists a tent function $h_{X,y}$ supported on $P$ such that the cells of $\sigma$ correspond to the regions of linearity of $h_{X,y}$. Equivalently, we say that the regular subdivision $\sigma$ is {\em induced} by $h_{X,y}$.
\end{definition}

A crucial step towards the proof of Theorem \ref{thm: existence and uniqueness MLE} is Proposition \ref{prop: achieve support}, where we apply an equivalent definition of chordal graphs which uses {\em junction trees}. We refer to \cite{blair1993clique} for more details on clique trees. See also \cite{lauritzen1996graphical} for a more general approach with hypergraphs.
By Proposition~\ref{prop:mle-for-connected-components}, without loss of generality, we may assume that $G$ is a connected graph.
If $G$ is disconnected, we apply the next definitions and properties to each connected component of $G$.

\begin{definition}\label{def: junction tree}
Let $G$ be a connected graph and consider its set of maximal cliques $\sC(G)$.
A {\em junction tree} (or {\em clique tree}) of $G$ is any tree $\sT$ with node set corresponding to the elements of $\sC(G)$ that satisfies the following {\em clique-intersection} property: for every pair of distinct cliques $C_i$ and $C_j$ in $\sC(G)$, the set $C_i\cap C_j$ is contained in every clique on the unique path connecting $C_i$ and $C_j$ in $\sT$.
\end{definition}

An example of a graph and its clique tree is depicted in Figure~\ref{fig: example clique tree}.
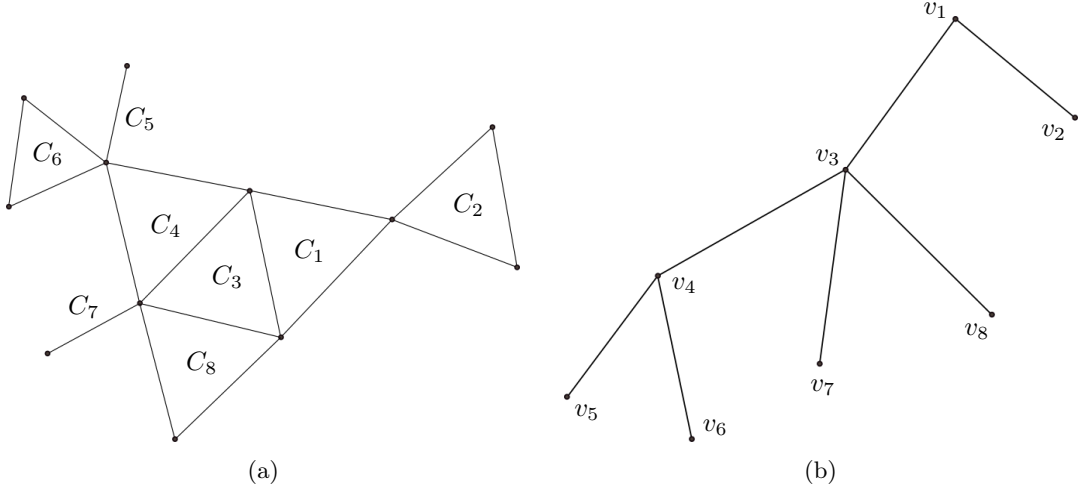
\begin{figure}[ht]
    \centering
    \begin{subfigure}[t]{.45\textwidth}
    \centering
    \begin{overpic}[width=\textwidth, tics=10]{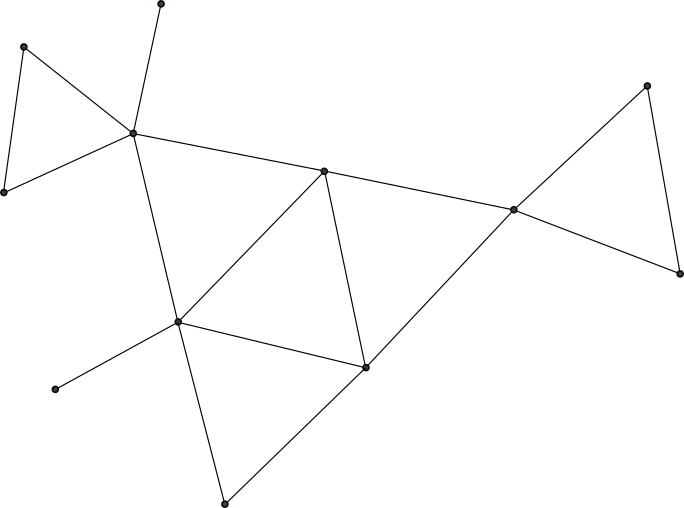}
    \put (56,36) {$C_1$}
    \put (87,45) {$C_2$}
    \put (40,31) {$C_3$}
    \put (28,41) {$C_4$}
    \put (23,62) {$C_5$}
    \put (5,55) {$C_6$}
    \put (12,25) {$C_7$}
    \put (35,14) {$C_8$}
    \end{overpic}
    \caption{}
    \label{fig: chordalgraph}
    \end{subfigure}
    \hspace{10pt}
    \begin{subfigure}[t]{.45\textwidth}
    \centering
    \begin{overpic}[width=\textwidth, tics=10]{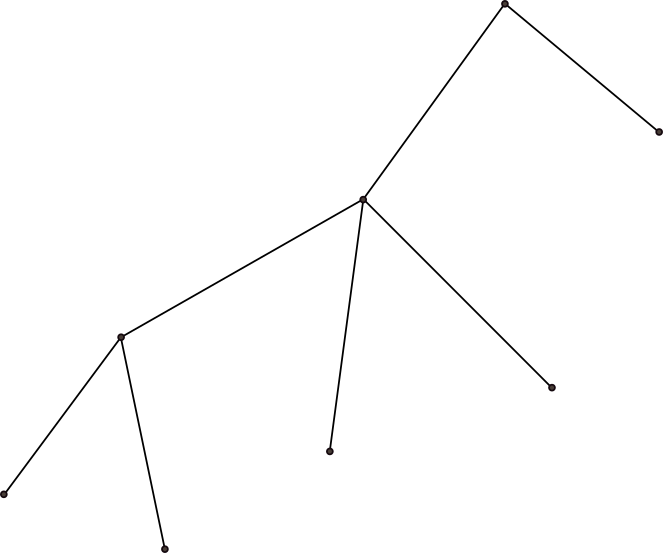}
    \put (70,84) {$v_1$}
    \put (93,60) {$v_2$}
    \put (49,55) {$v_3$}
    \put (21,30) {$v_4$}
    \put (2,5) {$v_5$}
    \put (27,2) {$v_6$}
    \put (48,10) {$v_7$}
    \put (78,21) {$v_8$}
    \end{overpic}
    \caption{}
    \label{fig: cliquetree}
    \end{subfigure}
    \caption{An example of clique tree of an undirected graph $G$ with $8$ maximal cliques.}
    \label{fig: example clique tree}
\end{figure}

\begin{theorem}{\cite[Theorem 3.1]{blair1993clique}}\label{thm: chordal iff has junction tree}
A connected graph $G$ is chordal if and only if it has a junction tree.
\end{theorem}

\begin{definition}\label{def: VL VR}
Let $G=(V,E)$ be a chordal graph with set of maximal cliques $\sC(G)=\{C_1,\ldots,C_t\}$.
Let $\sT$ be a tree with node set corresponding to the elements of $\sC(G)$. For every edge $(C_i,C_j)$ in $\sT$ define
\begin{align}
\begin{split}
V_L^{(i,j)} &\coloneqq \{v\in V \mid\text{$v$ appears in $\sT$ on the side of $C_i$}\}\setminus (C_i\cap C_j)\\
V_R^{(i,j)} &\coloneqq \{v\in V \mid\text{$v$ appears in $\sT$ on the side of $C_j$}\}\setminus(C_i\cap C_j)\,.
\end{split}
\end{align}
\end{definition}

\begin{lemma}\label{lem:CIJT}
Let $G$ as in Definition \ref{def: VL VR}. Suppose that $\sT$ is a junction tree of $G$. Then for all $i,j\in[t]$, $i<j$ we have
\begin{enumerate}
\item[(a)] $V_L^{(i,j)}\cap V_R^{(i,j)} = \emptyset$, and
\item[(b)] $V_L^{(i,j)}$ is separated from $V_R^{(i,j)}$ by $C_i\cap C_j$ in the graph $G$.
\end{enumerate}
\end{lemma}

In \cite[Definition 4.17]{koller2009probabilistic}, clique trees are essentially defined using the properties of Lemma~\ref{lem:CIJT}.

\begin{proof} [Proof of Lemma \ref{lem:CIJT}] ~\
$(a)$ Let $v\in V_L^{(i,j)} \cap V_R^{(i,j)}$. Then there exist a clique $C_k$ appearing on the side of $C_i$ in $\sT$ and a clique $C_\ell$ appearing on the side of $C_j$ in $\sT$ such that $v\in C_k \cap C_\ell$. By the choice of $C_k$ and $C_\ell$, the unique path from $C_k$ to $C_\ell$ in $\sT$ contains the edge $(C_i,C_j)$. So, considering that $\sT$ is a junction tree, $C_k \cap C_\ell \subseteq C_i \cap C_j$. This means that $v\in C_i \cap C_j$, which is impossible.

$(b)$ Let $a \in V_L^{(i,j)}$, $b\in V_R^{(i,j)}$, and $P\colon a=v_1,v_2,\ldots,v_m=b$ be a path from $a$ to $b$ in $G$. Set
\begin{align*}
    p \coloneqq \max \{q \in [m] \mid v_q \in V_L^{(i,j)}\}\,.
\end{align*}
By part $(a)$ we have $p \le m-1$. We prove that $v_{p+1}\in C_i\cap C_j$. By contradiction, assume that $v_{p+1}\not \in C_i\cap C_j$. There exists a maximal clique $C_k\in \sC(G)$ such that $(v_{p},v_{p+1})$ is an edge in $C_k$. Now one of the following happens:
\begin{enumerate}
    \item $C_k$ is on the side of $C_i$ in $\sT$. In this case, considering that $v_{p+1} \not \in C_i \cap C_j$, $v_{p+1} \in V_L^{(i,j)}$, which contradicts with the choice of $p$.
    \item $C_k$ is on the side of $C_j$ in $\sT$. In this case, given that $v_p \in V_L^{(i,j)}$, and thus, $v_p \not \in C_i \cap C_j$, we have $v_p\in V_R^{(i,j)}$, which contradicts with $V_L^{(i,j)} \cap V_R^{(i,j)} = \emptyset$.
\end{enumerate}
This proves that $v_{p+1} \in C_i \cap C_j$. Hence $V_L^{(i,j)}$ and $V_R^{(i,j)}$ are separated by $C_i \cap C_j$ in $G$.
\end{proof}

\begin{definition}\label{def: up-closure, height, depth}
Let $\sT$ be a tree with vertices $v_1,\ldots,v_t$. Fix $r\in[t]$ and consider the vertex $v_r$ as the root in $\sT$, so $\sT$ becomes a {\em rooted tree}.
We denote by $\le$ the {\em tree-order} associated with the rooted tree $\sT$, namely $v_i\le v_j$ if and only if $v_i$ belongs to the unique path in $\sT$ between $v_r$ and $v_j$. For every vertex $v_i$ of $\sT$, the {\em up-closure} of $v_i$ is the subtree $\sT_i\coloneqq\{v\mid v_i\le v\}$ (see Figure~\ref{fig:junction_tree}). In particular, $\sT_r=\sT$.
The {\em height} of a vertex $v_i$ in $\sT$ is the length $n_i$ of the longest possible path between any leaf in $\sT_i$ and $v_i$. In particular, the {\em height} of $\sT$ is the height of its root $v_r$. The {\em depth} of $v_i$ in $\sT$ is the length of the unique path between $v_i$ and the root of $\sT$. The maximum depth of a leaf of $\sT$ coincides with the height of $\sT$. A vertex $v_j$ is a {\em child} of $v_i$ if it has depth one in $\sT_i$.
\begin{figure}[ht]
    \centering
    \includegraphics[width=0.3\paperwidth]{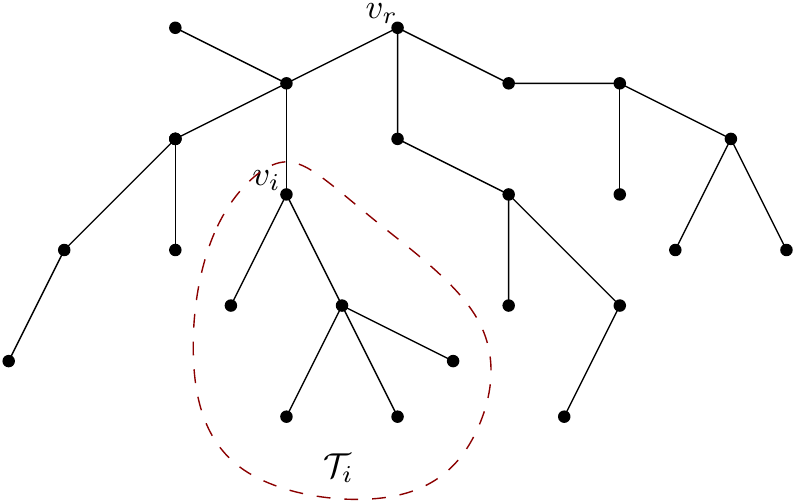}
    \caption{A tree $\sT$ with root $v_r$ and the up-closure $\sT_i$ of the vertex $v_i$.}
    \label{fig:junction_tree}
\end{figure}
\end{definition}

\section{Proofs of Sections~\ref{sec:support}, \ref{sec:existence}, and \ref{sec:location-of-tentpoles}}\label{appendix: proofs section 2,4,5}

\begin{proof}[Proof of Lemma~\ref{lemma: support of undirected is full dimensional}]
With probability 1, all sample points in $\sX$ are pairwise distinct.
Consider their convex hull $\conv(\sX)$.
Since $n\ge\max_{C\in \sC(G)}|C|+1$ and $f_0$ has full-dimensional support, it follows that the polytope $\pi_C(\conv(\sX))$ is a.s. full-dimensional in $\R^{C}$. 
Pick an arbitrary ordering $C_1,\ldots,C_t$ of the maximal cliques in $\sC(G)$. 
Pick a point $x$ in the relative interior of $\conv(\sX)$, that is the interior of $\conv(\sX)$ considered within the affine linear space spanned by $\conv(\sX)$. For each maximal clique $C_i$, the projection $x_{C_i}$ is in the relative interior of $\pi_{C_i}(\conv(\sX))$. Therefore, there exists a $|C_i|$-dimensional ball $B_{\varepsilon_i}(x_{C_i})$ of radius $\varepsilon_i$ centered at $x_{C_i}$ and contained in $\pi_{C_i}(\conv(\sX))$.
Take the ball $B_{\varepsilon}(x)$, where $\varepsilon\coloneqq\min_i\varepsilon_i$.
For all maximal cliques $C_i$ we have that $\pi_{C_i}(B_{\varepsilon}(x))\subseteq B_{\varepsilon_i}(x_{C_i})$, which means that $B_{\varepsilon}(x)$ is contained in the interior of $\sS_{G,\sX}$, proving that $\sS_{G,\sX}$ is a.s. full-dimensional.
\end{proof}

\begin{proof}[Proof of Proposition~\ref{prop: limit D_G i exists}]
It is easy to see that $(\sD_G^{(i)})_i$ is a nondecreasing sequence of convex polytopes.
Furthermore, every set $\sD_G^{(i)}$ is contained in the bounded set $\sS_{G, \sX}$, since the MLE $\hat{f}$ is positive on all the sets $\sD_{G}^{(i)}$ by construction.
Then, using the definitions of set-theoretic limit infimum and supremum, we get that
\begin{align*}
\mathrm{liminf}_{i\to\infty}\sD_G^{(i)} &\coloneqq \bigcup_{i\ge 0}\bigcap_{j\ge i}\sD_G^{(j)} = \bigcup_{i\ge 0}\sD_G^{(i)}\\
\mathrm{limsup}_{i\to\infty}\sD_G^{(i)} &\coloneqq \bigcap_{i\ge 0}\bigcup_{j\ge i}\sD_G^{(j)} = \bigcap_{i\ge 0}\bigcup_{j\ge 0}\sD_G^{(j)} = \bigcup_{j\ge 0}\sD_G^{(j)}\,.
\end{align*}
Hence the limit $\lim_{i\to\infty}\sD_G^{(i)}$ exists and is equal to the bounded convex set $\bigcup_{i\ge 0}\sD_G^{(i)}$.
\end{proof}

\begin{proof}[Proof of Proposition~\ref{prop: achieve support}]
By Theorem \ref{thm: chordal iff has junction tree} the graph $G$ admits a juction tree $\sT$ whose vertex set $\{v_1,\ldots,v_t\}$ corresponds to $\sC(G)=\{C_1,\ldots,C_t\}$. Each of the edges $(v_i,v_j)$ is labeled by the intersection $C_i\cap C_j$. Fix $r\in[t]$ and consider the vertex $v_r$ as the root in $\sT$.

Fix a point $x=(x_1,\ldots,x_d)\in\sS_{G,\sX}$. We want to show that $x\in\sD^{(t-1)}_{G}$.
The proof follows a two-level induction argument.
The first level of induction is with respect to the height parameter $\alpha\in\{0,\ldots,h(\sT)\}$, where $h(\sT)$ is the height of $\sT$. We claim the following:

\smallskip
\noindent{\bf Claim 1:}
For every $\alpha\in\{0,\ldots,h(\sT)\}$ and every vertex $v_i$ with height $h_i\le\alpha$ there exists a vector $a^{(\sT_i)}\in \sD^{\left(|\sT_i|-1\right)}_{G}$ such that
\begin{equation}\label{eq: a_j first induction}
a^{(\sT_i)}_j = x_j\quad\forall\,j\in\bigcup_{k\in\sT_i}C_k\,.
\end{equation}

Suppose that Claim 1 is true at the last step $\alpha=h(\sT)$.
Thus there exists $a^{(\sT_r)}\in \sD^{(t-1)}_{G}$ such that $a^{(\sT_r)}_j = x_j$ for all $j\in\bigcup_{k\in\sT}C_k=[d]$, namely $x=a^{(\sT_r)}\in\sD^{(t-1)}_{G}$. This concludes the proof.

\noindent$(1a)$ The base case of Claim 1 is $\alpha=0$. In this case, a vertex $v_i$ with height $h_i=0$ is a leaf in $\sT$. Equivalently $\sT_i$ consists only of $v_i$.
By construction of $\sS_{G,\sX}$, for each $i\in[t]$ there exists a point $a^{(i)}=(a^{(i)}_1,\ldots,a^{(i)}_d)\in\conv(\sX)=\sD^{(0)}_{G} \subseteq\R^d$ such that $a^{(i)}_j = x_j$ for all $j\in C_i$. Hence it suffices to set $a^{(\sT_i)}\coloneqq a^{(i)}$.

\noindent$(1b)$ Now suppose that Claim 1 is true for some $\alpha\in\{0,\ldots,h(\sT)-1\}$ and for all vertices with height at most $\alpha$. Consider the new height $\alpha+1$ and let $v_i$ be a vertex with height $h_i = \alpha+1$.
We want to show that there exists a vector $a^{(\sT_i)}\in \sD^{\left(|\sT_i|-1\right)}_{G}$ such that
\eqref{eq: a_j first induction} holds.
Let $v_{q_1},\ldots,v_{q_{m_i}}$ be the children of $v_i$ in $\sT_i$.
The second level of induction is with respect to the index $\ell\in[m_i]$. Figure \ref{fig: running example separation} explains on a concrete example the argument used in the following claim.

\smallskip
\noindent{\bf Claim 2:}
For all $\ell \in [m_i]$, there exists a vector $a^{(\sT_i,\ell)}\in \sD^{\left(\sum_{p=1}^\ell |\sT_{q_p}|\right)}_{G}$ such that 
\begin{equation}\label{eq: a_j second induction}
a^{(\sT_i,\ell)}_j = x_j\quad\forall\,j\in C_i\cup\left(\bigcup_{p=1}^\ell\bigcup_{k\in\sT_{q_p}}C_k\right)\,.
\end{equation}
\begin{figure}[ht]
    \centering
    \begin{subfigure}[t]{.31\textwidth}
    \centering
    \begin{overpic}[width=\textwidth, tics=10]{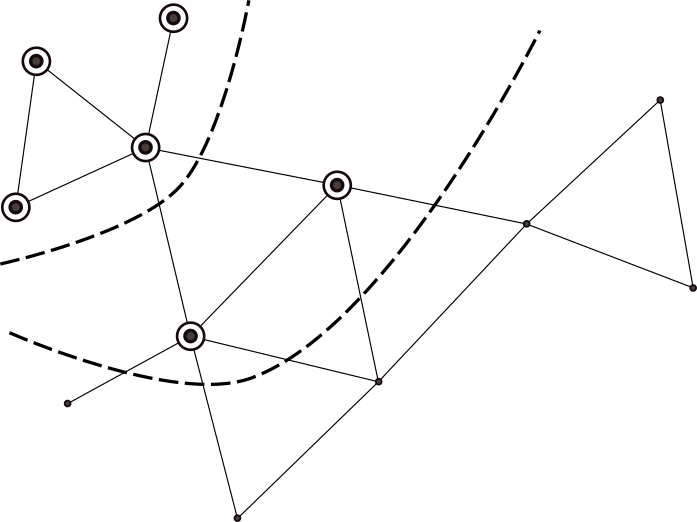}
    \put (5,70) {{\scriptsize $V_L^{(4,3)}$}}
    \put (77,60) {{\scriptsize $V_R^{(4,3)}$}}
    \put (39,31) {{\scriptsize $C_3$}}
    \put (30,40) {{\scriptsize $C_4$}}
    \end{overpic}
    \caption{}
    \label{fig: sep1}
    \end{subfigure}
    \hspace{9pt}
    \begin{subfigure}[t]{.31\textwidth}
    \centering
    \begin{overpic}[width=\textwidth, tics=10]{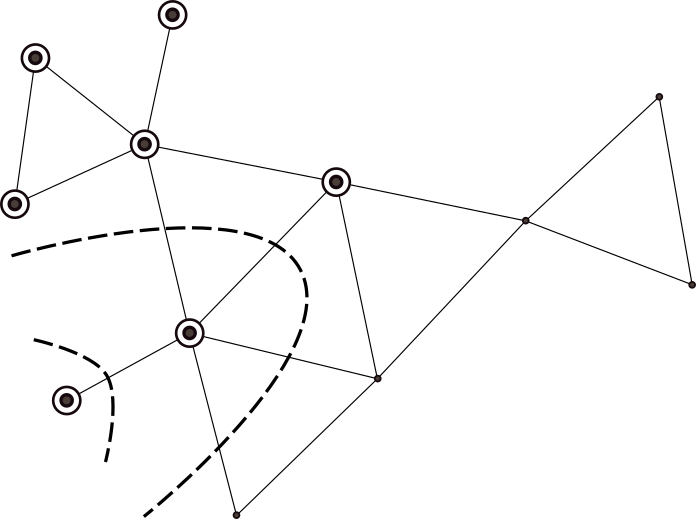}
    \put (-5,7) {{\scriptsize $V_L^{(7,3)}$}}
    \put (58,55) {{\scriptsize $V_R^{(7,3)}$}}
    \put (36,29) {{\scriptsize $C_3$}}
    \put (12,25) {{\scriptsize $C_7$}}
    \end{overpic}
    \caption{}
    \label{fig: sep2}
    \end{subfigure}
    \hspace{9pt}
    \begin{subfigure}[t]{.31\textwidth}
    \centering
    \begin{overpic}[width=\textwidth, tics=10]{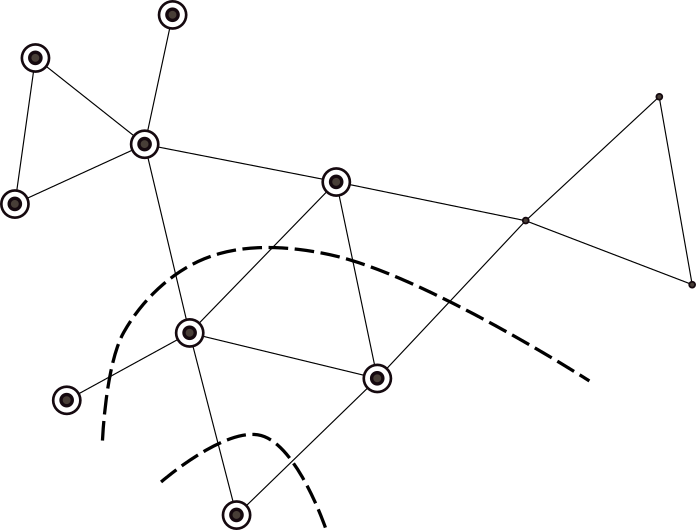}
    \put (17,0) {{\scriptsize $V_L^{(8,3)}$}}
    \put (58,55) {{\scriptsize $V_R^{(8,3)}$}}
    \put (39,31) {{\scriptsize $C_3$}}
    \put (34,15) {{\scriptsize $C_8$}}
    \end{overpic}
    \caption{}
    \label{fig: sep3}
    \end{subfigure}
    \caption{Induction proof of Claim 2 for the children $v_4,v_7,v_8$ of the node $v_3$ in the clique tree of Figure \ref{fig: cliquetree}, where the chosen root is $v_1$. At each step $j$, the circled vertices of $G$ correspond to coordinates where the vectors $a^{(\sT_3,j)}$ and $x$ coincide. The dashed lines indicate the current partition of $V$. After the induction, the last vector $a^{(\sT_3)}\coloneqq a^{(\sT_3,3)}$ coincides with $x$ on the union of cliques corresponding to the nodes of $\sT_3$.}
    \label{fig: running example separation}
\end{figure}
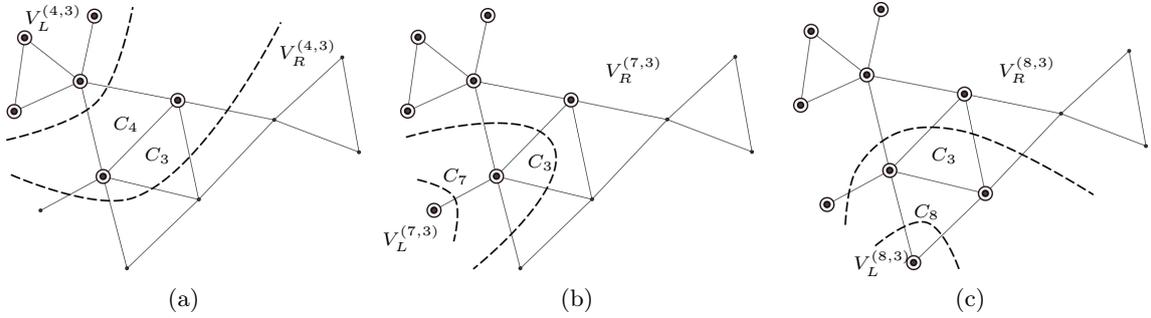
\noindent$(2a)$ The base case is $\ell=1$. Consider the vertex $v_{q_1}$.
We have
\begin{equation}\label{eq: separation1}
V_L^{(q_1,i)} = \left(\bigcup_{k\in\sT_{q_1}}C_k\right)\setminus(C_{q_1}\cap C_i)\,,\quad V_R^{(q_1,i)} = V \setminus\bigcup_{k\in\sT_{q_1}}C_k
\end{equation}
and by Lemma \ref{lem:CIJT} the set $V_L^{(q_1,i)}$ is separated from $V_R^{(q_1,i)}$ by $C_{q_1}\cap C_i$ in $G$.
Since $v_{q_1}$ has height $h_{q_1}\le\alpha$, by the induction hypothesis of Claim 1 there exists $a^{(\sT_{q_1})}\in \sD^{\left(|\sT_{q_1}|-1\right)}_{G}$ such that $a^{(\sT_{q_1})}_j = x_j$ for all $j\in\bigcup_{k\in\sT_{q_1}}C_k$.
Mimicking the argument of part $(1a)$, we know also that there exists a point $a^{(i)}\in\conv(\sX)=\sD^{(0)}_{G} \subseteq \sD^{\left(|\sT_{q_1}|-1\right)}_{G}$ , which satisfies $a^{(i)}_j=x_j$ for all $j\in C_i$. Summing up, we are considering the partition of $V$
\[
\{V_L^{(q_1,i)},V_R^{(q_1,i)},C_{q_1}\cap C_i\}\,,
\]
that corresponds to the global Markov statement $V_L^{(q_1,i)}\indep V_R^{(q_1,i)}|(C_{q_1}\cap C_i)\in\sI_{G}$, and we are picking two vectors $a^{(\sT_{q_1})}$ and $a^{(i)}$ of $\sD^{\left(|\sT_{q_1}|-1\right)}_{G}$ that coincide with $x$ on $C_{q_1}\cap C_i$.
Then, using the construction in Definition \ref{def: sets Di}, we define the new vector $a^{(\sT_i,1)}\in \sD^{\left(|\sT_{q_1}|\right)}_{G}$ as
\begin{equation}\label{eq: def aj1}
a^{(\sT_i,1)}_j \coloneqq
\begin{cases}
a^{(\sT_{q_1})}_j & \text{if $j\in V_L^{(q_1,i)}$}\\
a^{(i)}_j & \text{if $j\in V_R^{(q_1,i)}$}\\
x_j & \text{if $j\in C_{q_1}\cap C_i$}\,.
\end{cases}
\end{equation}
Note also that $a^{(\sT_i,1)}_j=a^{(i)}_j=x_j$ for all $j\in C_i\setminus C_{q_1}$, which implies that $a^{(\sT_i,1)}$ coincides with $x$ on the whole set $C_i\cup\bigcup_{k\in\sT_{q_1}}C_k$ as claimed in \eqref{eq: a_j second induction}.

\noindent$(2b)$ Suppose that Claim 2 holds for some $\ell\in [m_i-1]$, namely there exists a vector $a^{(\sT_i,\ell)}\in \sD^{\left(\sum_{p=1}^\ell |\sT_{q_p}|\right)}_{G} \subseteq \sD^{\left(\sum_{p=1}^{\ell+1}|\sT_{q_p}|-1\right)}_{G}$ such that \eqref{eq: a_j second induction} holds. Consider the vertex $v_{q_{\ell+1}}$. We have
\begin{equation}\label{eq: separation2}
V_L^{(q_{\ell+1},i)} = \left(\bigcup_{k\in\sT_{q_{\ell+1}}}C_k\right)\setminus(C_{q_{\ell+1}}\cap C_i)\,,\quad V_R^{(q_{\ell+1},i)} = V\setminus\bigcup_{k\in\sT_{q_{\ell+1}}}C_k
\end{equation}
and by Lemma \ref{lem:CIJT} the set $V_L^{(q_{\ell+1},i)}$ is separated from $V_R^{(q_{\ell+1},i)}$ by $C_{q_{\ell+1}}\cap C_i$ in $G$.
Since $v_{q_{\ell+1}}$ has height $h_{q_{\ell+1}}\le\alpha$, by the induction hypothesis of Claim 1 there exists $a^{(\sT_{q_{\ell+1}})}\in \sD^{\left(|\sT_{q_{\ell+1}}|-1\right)}_{G} \subseteq \sD^{\left(\sum_{p=1}^{\ell+1} |\sT_{q_p}|-1\right)}_{G}$ such that $a^{(\sT_{q_{\ell+1}})}_j = x_j$ for all $j\in \bigcup_{k\in\sT_{q_{\ell+1}}}C_k$.
Moreover by the induction hypothesis of Claim 2 we have $a^{(\sT_i,\ell)}_j = x_j$ for all $j\in C_i$. Hence we are considering the partition
\[
\{V_L^{(q_{\ell+1},i)},V_R^{(q_{\ell+1},i)},C_{q_{\ell+1}}\cap C_i\}\,,
\]
that corresponds to the global Markov statement $V_L^{(q_{\ell+1},i)}\indep V_R^{(q_{\ell+1},i)}|(C_{q_{\ell+1}}\cap C_i)\in\sI_{G}$, and we are picking two vectors $a^{(\sT_{q_{\ell+1}})}$ and $a^{(\sT_i,\ell)}$ of $\sD^{\left(\sum_{p=1}^{\ell+1}|\sT_{q_p}|-1\right)}_{G}$ that coincide with $x$ on $C_{q_{\ell+1}}\cap C_i$.
Again using the construction in Definition \ref{def: sets Di}, we define the new vector $a^{(\sT_i,\ell+1)}\in \sD^{\left(\sum_{p=1}^{\ell+1}|\sT_{q_p}|\right)}_{G}$ as
\begin{equation}\label{eq: def aj2}
a^{(\sT_i,\ell+1)}_j \coloneqq
\begin{cases}
a^{(\sT_{q_{\ell+1}})}_j & \text{if $j\in V_L^{(q_{\ell+1},i)}$}\\
a^{(\sT_i,\ell)}_j & \text{if $j\in V_R^{(q_{\ell+1},i)}$}\\
x_j & \text{if $j\in C_{q_{\ell+1}}\cap C_i$}\,.
\end{cases}.
\end{equation}
Note also that $a^{(\sT_i,\ell+1)}_j=a^{(\sT_i,\ell)}_j=x_j$ for all $j\in C_i\setminus C_{q_{\ell+1}}$, which implies that $a^{(\sT_i,\ell+1)}$ coincides with $x$ on the whole set $C_i\cup\bigcup_{p=1}^{\ell+1}\bigcup_{k\in\sT_{q_p}}C_k$ as claimed in \eqref{eq: a_j second induction}.

Having proved Claim 2, applying it for $l=m_i$ yields a vector $a^{(\sT_i,m_i)}\in \sD^{\left(\sum_{p=1}^{m_i}|\sT_{q_p}|\right)}_{G}=\sD^{\left(|\sT_i|-1\right)}_{G}$ such that $a^{(\sT_i,m_i)}_j = x_j$ for all $j\in C_i\cup\left(\bigcup_{p=1}^{m_i}\bigcup_{k\in\sT_{q_p}}C_k\right)=\bigcup_{k\in\sT_i}C_k$, similarly as in \eqref{eq: a_j first induction}. Hence the induction step of Claim 1 follows after setting $a^{(\sT_i)} \coloneqq a^{(\sT_i,m_i)}$.\qedhere
\end{proof}

\begin{proof}[Proof of Proposition~\ref{prop: equivalent def of X A ind B given S}] Let $D$ be any subset of $\R^d$, and consider a partition $\{A,B,S\}$ of $[d]$. Define
\begin{equation}\label{eq: def QABS}
\sQ_{A,B,S}(D) \coloneqq \pi_{A\cup S}^{-1}(\pi_{A\cup S}(D))\cap \pi_{B\cup S}^{-1}(\pi_{B\cup S}(D))\,.
\end{equation}
First, we show that $\sQ_{A,B,S}(D)\supseteq\sP_{A,B,S}(D)$. We have immediately that
\begin{equation}\label{eq: alternative X}
D=\bigcup_{\substack{x,y\in D\\x_S=y_S}}\{x,y\}\,.
\end{equation}
Now consider two points $x$ and $y$ in $D$. Then
\begin{equation}\label{eq: apply proj to x y}
\sQ_{A,B,S}(\{x,y\})=\{x_A,y_A\}\times\{x_B,y_B\}\times\{x_S\}\,.
\end{equation}
Consider the left-hand side of \eqref{eq: apply proj to x y} and take the union over all points $x$ and $y$ in $D$ such that $x_S=x_S$. Then
\begin{align*}
\bigcup_{\substack{x,y\in D\\x_S=y_S}}\sQ_{A,B,S}(\{x,y\}) &= \bigcup_{\substack{x,y\in D\\x_S=y_S}}\pi_{A\cup S}^{-1}(\pi_{A\cup S}(\{x,y\}))\cap \pi_{B\cup S}^{-1}(\pi_{B\cup S}(\{x,y\}))\\
&\subseteq \left(\bigcup_{\substack{x,y\in D\\x_S=y_S}}\pi_{A\cup S}^{-1}(\pi_{A\cup S}(\{x,y\}))\right)\cap\left(\bigcup_{\substack{x,y\in D\\x_S=y_S}}\pi_{B\cup S}^{-1}(\pi_{B\cup S}(\{x,y\}))\right)\\
&= \pi_{A\cup S}^{-1}\left(\pi_{A\cup S}\left(\bigcup_{\substack{x,y\in D\\x_S=y_S}}\{x,y\}\right)\right)\cap\pi_{B\cup S}^{-1}\left(\pi_{B\cup S}\left(\bigcup_{\substack{x,y\in D\\x_S=y_S}}\{x,y\}\right)\right)\\
&= \sQ_{A,B,S}(D)\,,
\end{align*}
where in the last equality we used \eqref{eq: alternative X}. On the other hand, considering the right-hand side of \eqref{eq: apply proj to x y} and taking the union over all points $x$ and $y$ in $D$ such that $x_S=x_S$, we get that
\[
\bigcup_{\substack{x,y\in D\\x_S=y_S}}\{x_A,y_A\}\times\{x_B,y_B\}\times\{x_S\}=\bigcup_{\substack{x,y\in D\\x_S=y_S}}\{(x_A,y_B,x_S),(y_A,x_B,x_S)\}=\sP_{A,B,S}(D)\,,
\]
therefore $\sQ_{A,B,S}(D)\supseteq\sP_{A,B,S}(D)$.
About the inclusion $\sQ_{A,B,S}(D)\subseteq\sP_{A,B,S}(D)$, consider a point $z\in\sQ_{A,B,S}(D)$. By definition of $\sQ_{A,B,S}(D)$ we have that $z_{A\cup S}\in\pi_{A\cup S}(D)$ and $z_{B\cup S}\in\pi_{B\cup S}(D)$, so there exist $x,y\in D$ such that
\[
z\in(\{x_A\}\times\R^B\times\{x_S\})\cap(\R^A\times\{y_B\}\times\{x_S\})=\{x_A\}\times\{y_B\}\times(\{x_S\}\cap\{y_S\})\,,
\]
hence necessarily there exist $x,y\in D$ such that $x_S=y_S$ and $z=(x_A,y_B,x_S)\in\sP_{A,B,S}(D)$.
\end{proof}

\begin{proof}[Proof of Corollary~\ref{corol: enough_samples}]
If $n\ge\max_{C\in \sC(\tilde{G})}|C|+1$, then by Lemma~\ref{lemma: support of undirected is full dimensional}, the polytope $\sS_{\tilde{G},\sX}$ of \eqref{eq: support of MLE for undirected} is full dimensional in $\R^d$ almost surely. In addition, since $\tilde{G}$ is a (chordal) cover of $G$, then, $\sS_{\tilde{G}, \sX}\subseteq \sS_{G, \sX}$.
The statement follows because $\sS_{\tilde{G},\sX}= \sD^{(t-1)}_{\tilde{G}}$ by Proposition~\ref{prop: achieve support}.
\end{proof}

\begingroup
\renewcommand{\arraystretch}{1.5}
\begin{table}
\centering
\resizebox{\textwidth}{!}{
\begin{tabular}{|c||c|c|c|c|c|c|c|c|c|c|c|c|c|c|c|c|c|c|}
\hline
$x_1$&7&7&3&7&8&$\frac{15}{2}$&$\frac{15}{2}$&4&$\frac{46}{7}$&4&3&$\frac{65}{9}$&$\frac{46}{7}$&$\frac{841}{127}$&$\frac{50}{11}$&$\frac{230}{31}$&$\frac{490}{67}$&$\frac{1297}{359}$\\
$x_2$&2&9&7&9&0&2&$\frac{9}{2}$&$\frac{15}{2}$&2&$\frac{28}{5}$&7&7&2&$\frac{245}{127}$&$\frac{171}{22}$&$\frac{162}{31}$&$\frac{414}{67}$&$\frac{2205}{359}$\\
$x_3$&8&8&9&8&1&8&8&8&8&8&$\frac{58}{9}$&9&8&$\frac{1969}{254}$&$\frac{155}{22}$&$\frac{268}{31}$&$\frac{592}{67}$&$\frac{2074}{359}$\\
$x_4$&0&0&3&4&8&4&4&4&$\frac{9}{28}$&4&3&3&4&$\frac{36}{127}$&$\frac{50}{11}$&$\frac{104}{31}$&$\frac{168}{67}$&$\frac{912}{359}$\\
\hline
\end{tabular}
}
\vspace*{1mm}
\caption{The 18 vertices of the polytope $\sS_{G,\sX}$ of Example \ref{ex: convergent sequence polytopes 4-cycle}.}\label{table: vertices support MLE 4-cycle}
\end{table}
\endgroup

\begin{proof}[Proof of Lemma \ref{lem: bound log-likelihood}]
Consider a density $f \in \sF_G$. Define
\begin{equation}\label{eq: def m M x^*}
m\coloneqq \min_i \log f(X^{(i)})\,,\quad M\coloneqq\max_{x \in\R^d} \log f(x)\,,\quad x^* \in \mathrm{argmax}_{x \in\R^d}\log f(x)\,.
\end{equation}
Let $X_A\indep X_B|X_C$ be a conditional independence statement.
We show by induction on $i$ that for all $i\ge 0$ and $x \in \sD^{(i)}_G$
\[
\log f(x) \ge 2^im-(2^i-1)M\,.
\]
Since $\sD^{(0)}_G \coloneqq \conv(\sX)$, then $\log f(x) \ge m$ by concavity of $\log f$. So assume the claim holds for $\sD^{(i)}_G$ and let $x \in \sD^{(i+1)}_G$. Then there exist $(x_A,x_B,x_C),(y_A,y_B,x_C) \in \sD^{(i)}_G$ s.t. $x=(x_A,y_B,x_C)$ and
\begin{align*}
    \log f(x_A,y_B,x_C)&= \log f(x_A,x_B,x_C)+\log f(y_A,y_B,x_C)- \log f(y_A,x_B,x_C)\\
    &\ge2(2^im-(2^i-1)M)-M=2^{i+1}m-(2^{i+1}-1)M\,.
\end{align*}
Define
\begin{equation}\label{eq: def m'}
m'\coloneqq 2^{t-1}m-(2^{t-1}-1)M\,.
\end{equation}
We know that $\log f(x)\ge m'$ for all $x\in \sD^{(t-1)}_G$.
Whenever $M$ is sufficiently large, $M-m'>1$ or, otherwise, $f$ is not a density. Take $x$ in $\sD^{(t-1)}_G$ and consider the convex linear combination $[1/(M-m')]x+[1-1/(M-m')]x^*$ between $x$ and the point $x^*$ defined in \eqref{eq: def m M x^*}.
By concavity of $\log f$ and the previous part we get
\begin{align*}
\log f\left[\frac{1}{M-m'}x+\left(1-\frac{1}{M-m'}\right)x^*\right] &\ge \frac{1}{M-m'}\log f(x)+\left(1-\frac{1}{M-m'}\right)\log f(x^*)\\
&\ge \frac{m'}{M-m'}+M-\frac{M}{M-m'}=M-1\,.
\end{align*}
We recall that $\mu$ denotes the Lebesgue measure on $\R^d$. Then
\[
\mu(\{x\mid \log f(x)\ge M-1\})\ge\mu\left[\frac{1}{M-m'}\sD^{(t-1)}_G+\left(1-\frac{1}{M-m'}\right)x^*\right]=\frac{\mu(\sD^{(t-1)}_G)}{(M-m')^d}\,.
\]
Therefore,
\[
1 = \int_{\R^d}f\,d\mu \ge \int_{\sD^{(t-1)}_G}f\,d\mu \ge e^{(M-1)}\frac{\mu(\sD^{(t-1)}_G)}{(M-m')^d}\,.
\]
In particular, we obtain the inequality $(M-m')^d\ge e^{(M-1)}\mu(\sD^{(t-1)}_G)$, which implies that \begin{equation}\label{eq: inequality m'}
m'\le M-e^{\frac{M-1}{d}}\mu(\sD^{(t-1)}_G)^{\frac{1}{d}}\,.
\end{equation}
For sufficiently large $M$, we have the inequality
\begin{equation}\label{eq: inequality M}
M \le \frac{1}{2}e^{\frac{M-1}{d}}\mu(\sD^{(t-1)}_G)^{\frac{1}{d}}\,.
\end{equation}
Plugging in \eqref{eq: inequality M} into \eqref{eq: inequality m'} we get $m'\le -\frac{1}{2}e^{(M-1)/d}\mu(\sD^{(t-1)}_G)^{1/d}$.
By \eqref{eq: def m'}, this is equivalent to
\[
m\le \frac{1}{2^{t-1}}\left[(2^{t-1}-1)M -\frac{1}{2}e^{\frac{M-1}{d}}\mu(\sD^{(t-1)}_G)^{\frac{1}{d}}\right]\,.
\]
Suppose that the minimum $m$ introduced in \eqref{eq: def m M x^*} is attained at the sample point $X^{(j)}$. Hence the log-likelihood function satisfies
\begin{align}\label{eq: bound on M}
\begin{split}
\ell(f,\sX) &= \sum_{i\neq j} w_i\log f(X^{(i)}) + w_j\log f(X^{(j)})\\ &\le (1-w_j)M+\frac{w_j}{2^{t-1}}\left[(2^{t-1}-1)M-\frac{1}{2}e^{\frac{M-1}{d}}\mu(\sD^{(t-1)}_G)^{\frac{1}{d}}\right]\,.
\end{split}
\end{align}
Since $n\ge\max_{C\in \sC(\tilde{G})}\{|C|\}+1$, by Corollary \ref{corol: enough_samples} we have $\mu(\sD^{(t-1)}_{\tilde{G}})>0$. Since the set of conditional independence statements associated with the Markov property for $\tilde{G}$ is contained in the set of conditional independence statements associated with the Markov property for $G$, we have $\sD^{(t-1)}_{\tilde{G}} \subseteq \sD^{(t-1)}_{G}$. Hence $\mu(\sD^{(t-1)}_{G})>0$ as well. Therefore, from \eqref{eq: bound on M} it follows that $\ell(f,\sX) \to -\infty$ as $M \to +\infty$. Hence, for every $R\in\R$, there exists a constant $M_R$ such that if $\log f(x)>M_R$ for some $x \in \sS_{G,\sX}$, then $\ell(f,\sX)\le R$.
\end{proof}

\begin{proof}[Proof of Lemma \ref{lem: lower_bound_on_function}]
Let $M$ and $m$ be as in \eqref{eq: def m M x^*}.
By assumption we have $M\le M_L$. Let $j = \mathrm{argmin}_i \log f(X^{(i)})$ be the index of a sample point $X^{(j)}$ at which $\log f$ is minimal, in particular $\log f(X^{(j)})=m$. Then
\[
L<\ell(f,\sX)=\sum_{i\neq j} w_i\log f(X^{(i)})+w_j\log f(X^{(j)})\le (1-w_j)M_L+w_j m\,,
\]
which implies that
\[
m \ge \frac{L}{w_j}-\left(\frac{1-w_j}{w_j}\right)M_L\,.
\]
As shown in the proof of Lemma~\ref{lem: bound log-likelihood} we have $\log f(x) \ge 2^{t-1}m - (2^{t-1}-1)M$ for all $x\in \sD^{(t-1)}_G$. Thus 
\begin{align*}
\min_{\sD^{(t-1)}_G}\log f &\ge 2^{t-1}m - (2^{t-1}-1)M \\
&\ge 2^{t-1}\left[\frac{L}{w_j}-\left(\frac{1-w_j}{w_j}\right)M_L\right] - (2^{t-1}-1)M_L\\
&=\frac{2^{t-1}}{w_j}(L-M_L)+M_L\\
&\ge\frac{2^{t-1}}{\max_i w_i}(L-M_L)+M_L\,,
\end{align*}
that is the inequality \eqref{eq: minimum log f S_n}.\qedhere
\end{proof}

\begin{proof}[Proof of Lemma \ref{lem:bounded_components}]
Let $\sC(G)= \{C_1,\ldots,C_t\}$ and let $D_1=C_2\cup\cdots\cup C_t$.
We assume $t\ge 2$, as in the case $t=1$ there is nothing to prove.
Then $\phi(x)=\phi_{C_1}(x_{C_1})+\Phi_{D_1}(x_{D_1})$, where $\Phi_{D_1}(x_{D_1}) = \sum_{i=2}^t\phi_{C_i}(x_{C_i})$.
Let $B = C_1\cap D_1$ be the set of indices of the variables shared by $\phi_{C_1}$ and $\Phi_{D_1}$. Note that $B\neq\emptyset$ as we are restricting to connected graphs.
What is more, since $t\ge 2$ we have that at least one of the containments $B\subseteq C_1$ and $B\subseteq D_1$ is strict.
With probability $1$, the projection $\pi_B(\sS_{G,\sX})$ of $\sS_{G,\sX}$ onto the coordinates of $B$ is full-dimensional, hence there exists a $|B|$-dimensional simplex $\sS_B = \conv\{b_0, b_1,\ldots, b_{|B|}\}$ contained in $\pi_B(\sS_{G,\sX})$.

\smallskip
\noindent{\bf Claim 1:} We can add and subtract a linear function in $x_B$ from $\phi_{C_1}$ and $\Phi_{D_1}$, respectively, so that
\begin{equation}\label{eq: def m_i}
m_i\coloneqq
\begin{cases}
\min_{\R^{C_1\setminus B}} \phi_{C_1}(\cdot, b_i) & \text{if $B\subsetneq C_1$}\\
\phi_{C_1}(b_i) & \text{if $B=C_1$}
\end{cases}
\end{equation}
does not depend on the index $i\in\{0,\ldots,|B|\}$.

To see this, consider $\alpha=(\alpha_i)_{i\in B}\in\R^{|B|}$, $\beta\in\R$ and let $\ell_{\alpha,\beta}(x_B)=\alpha\cdot x_B + \beta=\sum_{i\in B}\alpha_i x_i + \beta$ be a linear function in $x_B$. The linear system
\begin{equation}\label{eq: system l(x_B)}
m_i+\ell_{\alpha,\beta}(b_i) = \frac{1}{|B|+1}\sum_{j=0}^{|B|}m_j\quad\forall\,i\in\{0,\ldots,|B|\}
\end{equation}
has $|B|+1$ equations in the unknowns $(\alpha,\beta)\in\R^{|B|+1}$.
Since the vectors $b_0, \ldots, b_{|B|}$ are the vertices of a simplex, then, the vectors $(b_0, 1), \ldots, (b_{|B|}, 1)$ are linearly independent in $\R^{|B|+1}$. Hence the system \eqref{eq: system l(x_B)} has a unique solution $(\tilde{\alpha},\tilde{\beta})\in\R^{|B|+1}$. Furthermore, for all $i\in\{0,\ldots,|B|\}$ we have
\[
\frac{1}{|B|+1}\sum_{j=0}^{|B|}m_j = m_i + \ell_{\tilde{\alpha},\tilde{\beta}}(b_i) = 
\begin{cases}
\min_{\R^{C_1\setminus B}} \left[\phi_{C_1}(\cdot, b_i) + \ell_{\tilde{\alpha},\tilde{\beta}}(b_i)\right]  & \text{if $B\subsetneq C_1$}\\
\phi_{C_1}(b_i) + \ell_{\tilde{\alpha},\tilde{\beta}}(b_i) & \text{if $B=C_1$}\,.
\end{cases}
\]
Thus, the new functions $\tilde{\phi}_{C_1}(x_{C_1})\coloneqq\phi_{C_1}(x_{C_1}) + \ell_{\tilde{\alpha},\tilde{\beta}}(x_B)$ and $\tilde{\Phi}_{D_1}(x_{D_1})\coloneqq\Phi_{D_1}(x_{D_1})-\ell_{\tilde{\alpha},\tilde{\beta}}(x_B)$ sum up to $h(x)$, are still concave, and $\tilde{\phi}_{C_1}$ satisfies the condition of Claim 1.

\smallskip
\noindent{\bf Claim 2:} Let
\begin{equation}\label{eq: def V1}
V_1\coloneqq
\begin{cases}
(\R^{C_1\setminus B}\times \sS_B)\cap \pi_{C_1}(\sS_{G,\sX}) & \text{if $B\subsetneq C_1$}\\
\sS_B & \text{if $B=C_1$}\,.
\end{cases}
\end{equation}
Then
\begin{equation}\label{eq: max min 4K}
\max_{V_1}\phi_{C_1}-\min_{V_1}\phi_{C_1}\le 4K\,.
\end{equation}

Assume first that $B\subsetneq C_1$. Pick points $y_0, \ldots, y_{|B|}\in\R^{C_1\setminus B}$ such that $(y_i, b_i)\in (\R^{C_1\setminus B}\times \{ b_i\})\cap \pi_{C_1}(\sS_{G,\sX})$, and
\[
\phi_{C_1}(y_i, b_i) = \min_{z\mid (z, b_i)\in V_1}\phi_{C_1}(z, b_i)\,.
\]
In other words, for every $i\in\{0,\ldots,|B|\}$ the point $(y_i,b_i)$ achieves the minimum $m_i$ introduced in \eqref{eq: def m_i}.
Now choose any point $(x, b)\in V_1$.
In particular $b = \alpha_0 b_0 + \alpha_1 b_1 + \cdots + \alpha_{|B|} b_{|B|}$ for some vector $(\alpha_i)_{i\in B}\in\R^{|B|}$ of nonnegative entries with $\alpha_0+\cdots+\alpha_{|B|}=1$.
See the construction in Figure \ref{fig: example V_1 1} where $G$ and $\sS_{G,\sX}$ are the ones of Example \ref{ex: shape of S_n}.

\definecolor{dcrutc}{rgb}{0.8627450980392157,0.0784313725490196,0.23529411764705882}
\definecolor{ffqqqq}{rgb}{1.,0.,0.}
\definecolor{ffqqff}{rgb}{1.,0.,1.}
\definecolor{qqffff}{rgb}{0.,1.,1.}

\begin{figure}[ht]
\centering
\begin{subfigure}[t]{.5\textwidth}
\centering
\begin{tikzpicture}[line cap=round,line join=round,>=triangle 45,x=1.0cm,y=1.0cm]
\begin{axis}[
x=1.0cm,y=1.0cm,
axis lines=middle,
xmin=0.0,
xmax=2.5,
ymin=0.0,
ymax=5.5,
xtick={0,1,2},
ytick={0,1,2,3,4,5},]
\clip(0.,0.) rectangle (2.5,5.5);
\fill[line width=1.pt,color=qqffff,fill=qqffff,fill opacity=0.1] (0.,1.) -- (1.,5.) -- (2.,3.) -- cycle;
\fill[line width=1.pt,color=ffqqff,fill=ffqqff,fill opacity=0.3] (0.75,4.) -- (1.5,4.) -- (2.,3.) -- (1.,2.) -- (0.25,2.) -- cycle;
\draw [line width=2.pt,color=dcrutc] (0.,2.)-- (0.,4.);
\draw [line width=1.pt,color=qqffff] (0.,1.)-- (1.,5.);
\draw [line width=1.pt,color=qqffff] (1.,5.)-- (2.,3.);
\draw [line width=1.pt,color=qqffff] (2.,3.)-- (0.,1.);
\draw [line width=1.pt,color=ffqqff] (0.75,4.)-- (1.5,4.);
\draw [line width=1.pt,color=ffqqff] (1.5,4.)-- (2.,3.);
\draw [line width=1.pt,color=ffqqff] (2.,3.)-- (1.,2.);
\draw [line width=1.pt,color=ffqqff] (1.,2.)-- (0.25,2.);
\draw [line width=1.pt,color=ffqqff] (0.25,2.)-- (0.75,4.);
\draw [line width=1.pt] (1.313024846996962,4.)-- (0.5,2.);
\draw [line width=0.3pt,dash pattern=on 2pt off 2pt] (0.5,2.)-- (0.5,0.);
\draw [line width=0.3pt,dash pattern=on 2pt off 2pt] (1.313024846996962,4.)-- (1.313024846996962,0.);
\draw [line width=0.3pt,dash pattern=on 2pt off 2pt] (0.,4.)-- (0.75,4.);
\draw [line width=0.3pt,dash pattern=on 2pt off 2pt] (0.,2.)-- (0.25,2.);
\begin{scriptsize}
\draw [fill=qqffff] (0.91,3.) circle (0.75pt);
\draw [fill=qqffff] (0.,1.) circle (0.5pt);
\draw [fill=ffqqff] (2.,3.) circle (0.5pt);
\draw [fill=qqffff] (1.,5.) circle (0.5pt);
\draw [fill=ffqqqq] (0.,2.) circle (0.5pt);
\draw [fill=ffqqqq] (0.,4.) circle (0.5pt);
\draw [fill=ffqqff] (1.5,4.) circle (0.5pt);
\draw [fill=ffqqff] (0.75,4.) circle (0.5pt);
\draw [fill=ffqqff] (1.,2.) circle (0.5pt);
\draw [fill=ffqqff] (0.25,2.) circle (0.5pt);
\draw [fill=black] (0.5,0.) circle (0.5pt);
\draw [color=black] (0.5,2.)-- ++(-2.0pt,-2.0pt) -- ++(4.0pt,4.0pt) ++(-4.0pt,0) -- ++(4.0pt,-4.0pt);
\draw [fill=black] (1.313024846996962,0.) circle (0.5pt);
\draw [color=black] (1.313024846996962,4.)-- ++(-2.0pt,-2.0pt) -- ++(4.0pt,4.0pt) ++(-4.0pt,0) -- ++(4.0pt,-4.0pt);
\draw[color=black] (2.4,0.2) node {\tiny{$x_1$}};
\draw[color=black] (0.25,5.4) node {\tiny{$x_2$}};
\draw[color=ffqqff] (1.6,3) node {\scriptsize{$V_1$}};
\draw[color=dcrutc] (0.275,3.5) node {\scriptsize{$\sS_B$}};
\draw[color=black] (0.5,1.75) node {\tiny{$(y_0,\!b_0\!)$}};
\draw[color=black] (1.3,4.25) node {\tiny{$(y_1,\!b_1\!)$}};
\draw[color=black] (1.15,2.8) node {\tiny{$(y,\!b)$}};
\draw[color=black] (0.3,0.85) node {\tiny{$X_{12}^{(1)}$}};
\draw[color=black] (1,5.25) node {\tiny{$X_{12}^{(2)}$}};
\draw[color=black] (2.2,2.8) node {\tiny{$X_{12}^{(3)}$}};
\end{scriptsize}
\end{axis}
\end{tikzpicture}
\caption{}
\label{fig: example V_1 1}
\end{subfigure}%
\hspace{-3cm}
\begin{subfigure}[t]{.5\textwidth}
\centering
\begin{tikzpicture}[line cap=round,line join=round,>=triangle 45,x=1.0cm,y=1.0cm]
\begin{axis}[
x=1.0cm,y=1.0cm,
axis lines=middle,
xmin=0.0,
xmax=2.5,
ymin=0.0,
ymax=5.5,
xtick={0,1,2},
ytick={0,1,2,3,4,5},]
\clip(0.,0.) rectangle (2.5,5.5);
\fill[line width=2.pt,color=qqffff,fill=qqffff,fill opacity=0.1] (0.,1.) -- (1.,5.) -- (2.,3.) -- cycle;
\fill[line width=2.pt,color=ffqqff,fill=ffqqff,fill opacity=0.3] (0.75,4.) -- (1.5,4.) -- (2.,3.) -- (1.,2.) -- (0.25,2.) -- cycle;
\draw [line width=2.pt,color=dcrutc] (0.,2.)-- (0.,4.);
\draw [line width=1.pt,color=qqffff] (0.,1.)-- (1.,5.);
\draw [line width=1.pt,color=qqffff] (1.,5.)-- (2.,3.);
\draw [line width=1.pt,color=qqffff] (2.,3.)-- (0.,1.);
\draw [line width=1.pt,color=ffqqff] (0.75,4.)-- (1.5,4.);
\draw [line width=1.pt,color=ffqqff] (1.5,4.)-- (2.,3.);
\draw [line width=1.pt,color=ffqqff] (2.,3.)-- (1.,2.);
\draw [line width=1.pt,color=ffqqff] (1.,2.)-- (0.25,2.);
\draw [line width=1.pt,color=ffqqff] (0.25,2.)-- (0.75,4.);
\draw [line width=0.3pt,dash pattern=on 2pt off 2pt] (0.,4.)-- (0.75,4.);
\draw [line width=0.3pt,dash pattern=on 2pt off 2pt] (0.,2.)-- (0.25,2.);
\draw [line width=1.pt] (1.0247503486013,4.413743919851489)-- (1.3304960287179113,2.8500731558265353);
\draw [line width=1.pt,color=orange] (1.3304960287179113,2.8500731558265353)-- (0.,1.);
\draw [line width=1.pt,color=teal] (1.3304960287179113,2.8500731558265353)-- (1.5902845922722233,2.5902845922722233);
\begin{scriptsize}
\draw [fill=qqffff] (0.,1.) circle (0.5pt);
\draw [fill=ffqqff] (2.,3.) circle (0.5pt);
\draw [fill=qqffff] (1.,5.) circle (0.5pt);
\draw [fill=ffqqqq] (0.,2.) circle (0.5pt);
\draw [fill=ffqqqq] (0.,4.) circle (0.5pt);
\draw [fill=ffqqff] (1.5,4.) circle (0.5pt);
\draw [fill=ffqqff] (0.75,4.) circle (0.5pt);
\draw [fill=ffqqff] (1.,2.) circle (0.5pt);
\draw [fill=ffqqff] (0.25,2.) circle (0.5pt);
\draw [fill=black] (1.3304960287179113,2.8500731558265353) circle (0.5pt);
\draw [fill=black] (1.0247503486013,4.413743919851489) circle (0.5pt);
\draw [fill=black] (1.105649997734273,4.) circle (0.5pt);
\draw [color=blue] (1.3304960287179115,2.8500731558265353)-- ++(-1.5pt,-1.5pt) -- ++(3.0pt,3.0pt) ++(-3.0pt,0) -- ++(3.0pt,-3.0pt);
\draw [fill=black] (1.5902845922722233,2.5902845922722233) circle (0.5pt);
\draw[color=black] (2.4,0.2) node {\tiny{$x_1$}};
\draw[color=black] (0.25,5.4) node {\tiny{$x_2$}};
\draw[color=ffqqff] (1.5,3.4) node {\scriptsize{$V_1$}};
\draw[color=dcrutc] (0.275,3.5) node {\scriptsize{$\sS_B$}};
\draw[color=black] (0.3,0.85) node {\tiny{$X_{12}^{(1)}$}};
\draw[color=black] (1,5.25) node {\tiny{$X_{12}^{(2)}$}};
\draw[color=black] (2.2,2.8) node {\tiny{$X_{12}^{(3)}$}};
\draw[color=blue] (0.875,2.9) node {\tiny{$(x^*\!\!,\!b^*\!)$}};
\draw[color=blue] (1.45,4.175) node {\tiny{$(x'\!\!,\!b'\!)$}};
\draw[color=blue] (1.35,4.5) node {\tiny{$(x\!,\!b)$}};
\draw[color=teal] (1.55,2.8) node {\tiny{$\gamma$}};
\draw[color=orange] (0.7,2.2) node {\tiny{$\Gamma$}};
\end{scriptsize}
\end{axis}
\end{tikzpicture}
\caption{}
\label{fig: example V_1 2}
\end{subfigure}
\caption{The polytope $V_1$ for the graph $G$ of Example \ref{ex: shape of S_n} and choosing $\sS_B=[2,4]$.}
\end{figure}
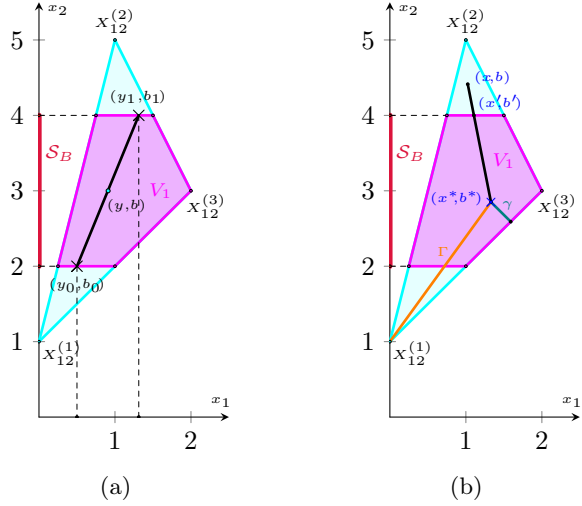

Define $y\coloneqq \alpha_0y_0+\cdots+\alpha_{|B|}y_{|B|}$, so that $(y, b) = \alpha_0(y_0, b_0) + \cdots + \alpha_{|B|}(y_{|B|}, b_{|B|})$.
We show that
\begin{itemize}
\item[$(i)$] $\phi_{C_1}(y, b) - \phi_{C_1}(y_i, b_i) \le 2K$ for all $i$, and
\item[$(ii)$] $|\phi_{C_1}(x, b) - \phi_{C_1}(y, b)|\le 2K$.
\end{itemize}
From $(i)$ and $(ii)$, it follows immediately that $|\phi_{C_1}(x,b)-\phi_{C_1}(y_i, b_i)|\le 4K$, proving Claim 2.

In order to prove $(i)$, define $T\coloneqq\phi_{C_1}(y, b)-\phi_{C_1}(y_i, b_i)$.
Note that $T$ does not depend on the index $i$ chosen because of the choice of $y_i$ and because of Claim 1.
If $B\subsetneq D_1$, then for every $i\in D_1\setminus B$, let $z_i\in\R^{D_1\setminus B}$ be such that $(y_i, b_i, z_i)\in \sS_{G,\sX}$ and define $z\coloneqq\alpha_0z_0+\cdots+\alpha_{|B|}z_{|B|}$. Otherwise $B=D_1$ and we proceed with $z=b$.
Then $(y, b, z)\in \sS_{G,\sX}$ as well. Note that, by concavity of $\Phi_{D_1}$, we have
\[
\Phi_{D_1}(b,z) - \sum_{i=0}^{|B|}\alpha_i\Phi_{D_1}(b_i, z_i) \ge 0\,.
\]
On the other hand, knowing that $\Phi_{D_1} = \phi - \phi_{C_1}$, we get
\begin{align*}
0 &\le \phi(y,b,z)-\phi_{C_1}(y,b) - \sum_{i=0}^{|B|}\alpha_i[\phi(y_i, b_i, z_i) - \phi_{C_1}(y_i, b_i)]\\
&= \phi(y,b,z) - \sum_{i=0}^{|B|}\alpha_i \phi(y_i, b_i, z_i) - \left(\phi_{C_1}(y,b) - \sum_{i=0}^{|B|}\alpha_i\phi_{C_1}(y_i, b_i)\right)\\
&\stackrel{\star}{=}\sum_{i=0}^{|B|}\alpha_i[\phi(y,b,z)-\phi(y_i, b_i, z_i)] - \sum_{i=0}^{|B|} \alpha_i[\phi_{C_1}(y,b) - \phi_{C_1}(y_i, b_i)] \le 2K - T\,,
\end{align*}
which shows $(i)$. Note that in the equality $\star$ we are using the property $\alpha_0+\cdots+\alpha_{|B|}=1$.
Now, to prove $(ii)$, if $B\subsetneq D_1$, then choose a point $z$ such that $(b, z)\in\pi_{D_1}(\sS_{G,\sX})$, otherwise $B=D_1$ and we proceed with $z=b$. First, we show that $(x, b, z) \in \sS_{G,\sX} = \bigcap_{i=1}^t \pi_{C_i}^{-1}\left(\pi_{C_i}(\conv(\sX))\right)$ or equivalently that $\pi_{C_i}(x,b,z) \in \pi_{C_i}(\conv(\sX))$ for every $i\in[t]$. We have $\pi_{C_1}(x,b,z) = (x,b) \in \pi_{C_1}(\sS_{G,\sX}) = \pi_{C_1}(\conv(\sX))$ because by choice $(x,b) \in \pi_{C_1}(\sS_{G,\sX})$. For $2 \le i \le t$, the containment $\pi_{C_i}(x,b,z) = \pi_{C_i}(b,z) \in \pi_{C_i}(\pi_{D_1}(\sS_{G,\sX})) = \pi_{C_i}(\sS_{G,\sX}) = \pi_{C_i}(\conv(\sX))$ holds because $(b, z)\in\pi_{D_1}(\sS_{G,\sX})$. By a similar argument $(y, b, z) \in \sS_{G,\sX}$.
Then,
\begin{align*}
|\phi_{C_1}(x,b) - \phi_{C_1}(y,b)| &= |\phi(x,b, z) - \Phi_{D_1}(b, z) - \phi(y,b,z) + \Phi_{D_1}(b, z)|\\
&= |\phi(x,b,z) - \phi(y,b,z)| \le 2K\,,
\end{align*}
which proves $(ii)$ and concludes the proof of Claim 2 for $B\subsetneq C_1$.

If instead $B=C_1$, then we repeat the proof with $y_i=b_i$ and choosing only $b\in V_1=\sS_B$. We recall that, under our assumption, if $B=C_1$ then necessarily $B\subsetneq D_1$.

\smallskip
\noindent{\bf Claim 3:} Let
\begin{equation}\label{eq: def W1}
W_1\coloneqq
\begin{cases}
(\sS_B\times \R^{D_1\setminus B})\cap\pi_{D_1}(\sS_{G,\sX}) & \text{if $B\subsetneq D_1$}\\
\sS_B & \text{if $B=D_1$}\,.
\end{cases}
\end{equation}
Then
\begin{equation}\label{eq: max min 6K}
\max_{W_1}\Phi_{D_1} - \min_{W_1}\Phi_{D_1} \le 6K\,.
\end{equation}
Assume $B\subsetneq D_1$. Let $(b', z'), (b'', z'')\in W_1$. If $B\subsetneq C_1$, then pick $x',x''\in\R^{C_1\setminus B}$ such that both $(x',b',z')$ and $(x'',b'',z'')$ are points of $\sS_{G,\sX}$. Otherwise $B=C_1$ and we proceed with $(b',z')$ and $(b'',z'')$.
Then $(x',b'), (x'', b'')\in V_1$, and combining the two previous claims we get
\[
\begin{gathered}
|\Phi_{D_1}(b', z') - \Phi_{D_1}(b'', z'')| = |\phi(x',b',z') - \phi(x'',b'',z'') - \phi_{C_1}(x',b') + \phi_{C_1}(x'',b'')| \le\\
\le |\phi(x',b',z') - \phi(x'',b'',z'')| + |\phi_{C_1}(x',b') - \phi_{C_1}(x'',b'')|\le 2K+4K=6K\,.
\end{gathered}
\]
If $B=D_1$, then repeat the argument with $b'=z'$ and $b''=z''$.

For the rest of the proof, we assume that $B\subsetneq C_1$ and $B\subsetneq D_1$, otherwise the proof simplifies as in the previous steps and we omit it for brevity.

\smallskip
\noindent{\bf Claim 4:} Over the set $\pi_{C_1}(\sS_{G,\sX})$,
\begin{equation}\label{eq: max min 2thetaK}
\max_{\pi_{C_1}(\sS_{G,\sX})}\phi_{C_1}-\min_{\pi_{C_1}(\sS_{G,\sX})}\phi_{C_1}\le \theta K\,,
\end{equation}
where $\theta$ only depends on the graph $G$ and the sample $\sX$.

The basic idea behind the proof of Claim 4 is to use concavity of $\phi_{C_1}$ and $\Phi_{D_1}$ and the fact that the two functions behave nicely (in the sense of Claims 2 and 3) over the simplex $\sS_B$. Recall the definition of $V_1$ in \eqref{eq: def V1}. Fix any point $(x^*,b^*)\in\mathrm{int}(V_1)$. We can do this because $\mathrm{int}(V_1)\neq 0$ with probability $1$.
Define the quantities
\begin{align}\label{eq: def gammas}
\begin{split}
\gamma &\coloneqq \dist((x^*,b^*), \partial(V_1))=\min_{(x,b)\in\partial(V_1)}\|(x,b)-(x^*,b^*)\|\,,\\
\Gamma &\coloneqq \max_{(x,b)\in\pi_{C_1}(\sS_{G,\sX})}\|(x,b)-(x^*,b^*)\|\,.
\end{split}
\end{align}
Pick a point $(x,b)\in\pi_{C_1}(\sS_{G,\sX})$. If $(x,b)\in V_1$, then we already know by Claim 2 that
\begin{equation}\label{eq:case1}
\phi_{C_1}(x,b)\le\phi_{C_1}(x^*,b^*) + 4K\,.
\end{equation}
Otherwise, $(x,b)\in\pi_{C_1}(\sS_{G,\sX})\setminus V_1$. Therefore, the line segment between $(x,b)$ and $(x^*,b^*)$ intersects the boundary $\partial(V_1)$ at a point, call it $(x',b')$, and $(x',b')$ is in between $(x,b)$ and $(x^*,b^*)$ on this line segment. See Figure \ref{fig: example V_1 2} for an example.
In particular, we have $(x',b')=\eta_1(x,b)+\eta_2(x^*,b^*)$, where the nonnegative numbers $\eta_1=\|(x^*,b^*)-(x',b')\|/\|(x,b)-(x^*,b^*)\|$ and $\eta_2=\|(x,b)-(x',b')\|/\|(x,b)-(x^*,b^*)\|$ sum to $1$.
What is more, observe that
\begin{align*}
\|(x^*, b^*) - (x',b')\| &\ge \gamma\\
\|(x,b) - (x', b')\| &= \|(x,b) - (x^*,b^*)\|-\|(x^*, b^*) - (x',b')\|\\
&\le \Gamma-\|(x^*, b^*) - (y',b')\|\\
&\le\Gamma-\gamma\,.
\end{align*}
Putting together all this information, and using concavity of $\phi_{C_1}$, we have that
\[
\phi_{C_1}(x',b') \ge \eta_1\,\phi_{C_1}(x,b)+\eta_2\,\phi_{C_1}(x^*,b^*)\,,
\]
or equivalently,
\begin{equation}\label{eq: inequality bound phi_C1}
\phi_{C_1}(x,b) - \phi_{C_1}(x',b')\le \frac{\|(x,b)-(x',b')\|}{\|(x^*,b^*)-(x',b')\|}[\phi_{C_1}(x',b') - \phi_{C_1}(x^*,b^*)]\le 4K\frac{\Gamma-\gamma}{\gamma}\,.
\end{equation}
Combining \eqref{eq: inequality bound phi_C1} with the inequality \eqref{eq:case1} for $(x',b')\in V_1$, we have that
\begin{equation}\label{eq:case3}
\phi_{C_1}(x,b) \le \phi_{C_1}(x',b') + 4K\frac{\Gamma-\gamma}{\gamma} \le \phi_{C_1}(x^*,b^*) + 4K\frac{\Gamma}{\gamma}\,.
\end{equation}
Summing up, in \eqref{eq:case1} and \eqref{eq:case3} we have computed upper bounds of $\phi_{C_1}$ over $V_1$ and $\pi_{C_1}(\sS_{G,\sX})\setminus V_1$, respectively. Since $\Gamma \ge \gamma$, we conclude that
\begin{equation}\label{eq:bound_above}
\max_{\pi_{C_1}(\sS_{G,\sX})}\phi_{C_1} \le \phi_{C_1}(x^*,b^*)+4K\frac{\Gamma}{\gamma}\,.
\end{equation}
Next, we show that $\phi_{C_1}(x,b)$ is also bounded below.
To do that, we perform the same argument as above but for the function $\Phi_{D_1}$, showing that $\Phi_{D_1}(b,z)$ is bounded above. This yields a lower bound on $\phi_{C_1}(x,b)$.
Consider the set $W_1$ introduced in \eqref{eq: def W1}.
Let $(b_*,z_*)\in\mathrm{int}(W_1)$.
We define the constants $\delta$ and $\Delta$ similarly as in \eqref{eq: def gammas}:
\begin{align}\label{eq: def deltas}
\begin{split}
\delta &\coloneqq \dist((b_*,z_*),\partial(W_1))\\
\Delta &\coloneqq \max_{(b,z)\in\pi_{D_1}(\sS_{G,\sX})}\|(b,z), (b_*,z_*)\|\,.
\end{split}
\end{align}
With a similar proof of inequality \eqref{eq:bound_above} and applying the upper bound \eqref{eq: max min 6K}, one verifies that
\begin{equation}\label{eq: upper_bound_Phi}
\max_{\pi_{D_1}(\sS_{G,\sX})}\Phi_{D_1} \le \Phi_{D_1}(b_*,z_*) +6K\frac\Delta\delta\,.
\end{equation}
Now, let $(x,b,z)\in \sS_{G,\sX}$ and consider the point $(b_*,z_*)$ used to define the constants $\Delta$ and $\delta$ in \eqref{eq: def deltas}. Let $x_*\in\R^{C_1\setminus B}$ be such that $(x_*,b_*,z_*)\in \sS_{G,\sX}$. Then
\begin{align*}
\phi_{C_1}(x,b) &= \phi(x,b,z) - \Phi_{D_1}(b,z)\\
&\ge \phi(x,b,z) - \Phi_{D_1}(b_*,z_*) -6K\frac\Delta\delta\quad\text{(by \eqref{eq: upper_bound_Phi})}\\
&= \phi(x,b,z) - \phi(x_*, b_*,z_*) + \phi_{C_1}(x_*,b_*)-6K\frac\Delta\delta\,.
\end{align*}
Knowing that $\mathrm{Im}(\phi)\subseteq[-K,K]$,
\[
\phi_{C_1}(x_*,b_*)- \phi_{C_1}(x,b)\le \phi(x_*, b_*,z_*) - \phi(x,b,z) +6K\frac\Delta\delta \le 2K + 6K\frac\Delta\delta\,.
\]
This yields the lower bound
\begin{equation}\label{eq:lower_bound}
\min_{\pi_{C_1}(\sS_{G,\sX})}\phi_{C_1}\ge \phi_{C_1}(x_*,b_*) - 2K - 6K\frac\Delta\delta\,.
\end{equation}
Therefore, combining \eqref{eq:bound_above} and \eqref{eq:lower_bound}, we get that
\begin{align*}
\max_{\pi_{C_1}(\sS_{G,\sX})}\phi_{C_1}-\min_{\pi_{C_1}(\sS_{G,\sX})}\phi_{C_1} &\le \phi_{C_1}(x^*,b^*) - \phi_{C_1}(x_*,b_*) + 2K\left(1 + 2\frac{\Gamma}{\gamma} + 3\frac\Delta\delta\right)\\
&\le \max_{V_1}\phi_{C_1}-\min_{V_1}\phi_{C_1} + 2K\left(1 + 2\frac{\Gamma}{\gamma} + 3\frac\Delta\delta\right)\\
&\le 2K\left(3 + 2\frac{\Gamma}{\gamma} + 3\frac\Delta\delta\right)\,,
\end{align*}
where in the last inequality we applied Claim 2 since both $(x^*,b^*)$ and $(x_*,b_*)$ belong to $V_1$.
This completes the proof of Claim 4 setting $\theta\coloneqq 2(3 + 2\Gamma/\gamma + 3\Delta/\delta)$. Note that $\theta$ depends only on the geometry of $\sS_{G,\sX}$, that is, only on the graph $G$ and the sample $\sX$.

\smallskip
\noindent{\bf Claim 5:} We have
\begin{equation}\label{eq: max min Phi}
\max_{\pi_{D_1}(\sS_{G,\sX})}\Phi_{D_1}-\min_{\pi_{D_1}(\sS_{G,\sX})}\Phi_{D_1}\le (2+\theta)K\,.
\end{equation}

Consider a point $(x,b,z)\in \sS_{G,\sX}$, where $(x,b)\in\pi_{C_1}(\sS_{G,\sX})$ and $(b,z)\in\pi_{D_1}(\sS_{G,\sX})$. Knowing that $\Phi_{D_1}(b,z)=\phi(x,b,z)-\phi_{C_1}(x,b)$, we have that
\[
-K-\max_{\pi_{C_1}(\sS_{G,\sX})}\phi_{C_1} \le \Phi_{D_1}(b,z) \le K-\min_{\pi_{C_1}(\sS_{G,\sX})}\phi_{C_1}\,.
\]
Therefore
\[
\max_{\pi_{D_1}(\sS_{G,\sX})}\Phi_{D_1}-\min_{\pi_{D_1}(\sS_{G,\sX})}\Phi_{D_1} \le 2K + \max_{\pi_{C_1}(\sS_{G,\sX})}\phi_{C_1}-\min_{\pi_{C_1}(\sS_{G,\sX})}\phi_{C_1} \le (2 + \theta)K\,,
\]
where in the last inequality we plugged in the estimate \eqref{eq: max min 2thetaK}.

\smallskip
\noindent{\bf Claim 6:} We can adjust $\phi_{C_1}$ and $\Phi_{D_1}$ by adding and subtracting constants so that they take values in $[-(4+\theta)K, (4+\theta)K]$ on their respective domains.

First, we add and subtract a constant $M$ so that, in addition to the requirements in Claims 4 and 5, we also have that $\max_{\pi_{C_1}(\sS_{G,\sX})}\phi_{C_1} = \max_{\pi_{D_1}(\sS_{G,\sX})}\Phi_{D_1}$.
More precisely, define $M\coloneqq(\max_{\pi_{C_1}(\sS_{G,\sX})}\phi_{C_1}-\max_{\pi_{D_1}(\sS_{G,\sX})}\Phi_{D_1})/2$. Then the functions $\tilde{\phi}_{C_1}\coloneqq\phi_{C_1}-M$ and $\tilde{\Phi}_{D_1}\coloneqq\Phi_{D_1}+M$ add up to $\phi$ and their common maximum value is $(\max_{\pi_{C_1}(\sS_{G,\sX})}\phi_{C_1}+\max_{\pi_{D_1}(\sS_{G,\sX})}\Phi_{D_1})/2$.
With a little abuse of notation, we redefine $\phi_{C_1}$ and $\Phi_{D_1}$ as the functions $\tilde{\phi}_{C_1}$ and $\tilde{\Phi}_{D_1}$ just computed. Note that this modification does not affect the parameter $\theta$ computed in Claim 4.

\noindent Now consider a point $(x,b,z)\in \sS_{G,\sX}$, where $(x,b)\in\pi_{C_1}(\sS_{G,\sX})$ and $(b,z)\in\pi_{D_1}(\sS_{G,\sX})$. We have
\[
\phi_{C_1}(x,b) = \phi(x,b,z)-\Phi_{D_1}(b,z) \le \phi(x,b,z)+\max_{\pi_{D_1}(\sS_{G,\sX})}\Phi_{D_1}-\min_{\pi_{D_1}(\sS_{G,\sX})}\Phi_{D_1} \le (3+\theta)K\,,
\]
where in the last inequality we applied \eqref{eq: max min Phi}. Hence
\begin{equation}\label{eq: common max phi Phi}
\max_{\pi_{C_1}(\sS_{G,\sX})}\phi_{C_1} = \max_{\pi_{D_1}(\sS_{G,\sX})}\Phi_{D_1} \le (3+\theta)K\,.
\end{equation}
Moreover,
\begin{align}\label{eq: common min phi Phi}
\begin{split}
\phi_{C_1}(x,b) &= \phi(x,b,z) - \Phi_{D_1}(b,z) \ge -K-\max_{\pi_{D_1}(\sS_{G,\sX})}\Phi_{D_1} \ge -(4+\theta)K\\
\Phi_{D_1}(b,z) &= \phi(x,b,z) - \phi_{C_1}(b,z) \ge -K-\max_{\pi_{C_1}(\sS_{G,\sX})}\phi_{C_1} \ge -(4+\theta)K\,.
\end{split}
\end{align}
Therefore, combining \eqref{eq: common max phi Phi} and \eqref{eq: common min phi Phi}, we get that both $\phi_{C_1}$ and $\Phi_{D_1}$ take values in the interval $[-(4+\theta)K, (4+\theta)K]$. This concludes the proof of Claim 6.

Next we proceed by replacing $\phi$ with $\Phi_{D_1}$ and $K$ with $K'\coloneqq(4+\theta)K$. We repeat Claim 1 up to Claim 6 after setting $D_2=C_3\cup\cdots\cup C_t$ and writing $\Phi_{D_1}(x_{D_1})=\phi_{C_2}(x_{C_2})+\Phi_{D_2}(x_{D_2})$, where $\Phi_{D_2}(x_{D_2})\coloneqq\sum_{i=3}^t\phi_{C_i}(x_{C_i})$. The new constant $\theta_2$ computed in Claim 4 depends only on the geometry of $\pi_{D_1}(\sS_{G,\sX})$. At the end of the second iteration, the possible modifications of $\phi_{C_2}(x_{C_2})$ and $\Phi_{D_2}(x_{D_2})$ take values in $[-(4+\theta_2)K', (4+\theta_2)K']=[-(4+\theta_2)(4+\theta_1)K, (4+\theta_2)(4+\theta_1)K]$, where $\theta_1\coloneqq\theta$. After the $(t-1)$-th step each modified summand $\phi_{C_i}$ takes values in $[-\rho K,\rho K]$, where $\rho\coloneqq\prod_{i=1}^{t-1}(4+\theta_i)$ and each constant $\theta_i$ was computed in Claim 4 at the $i$-th step of iteration.
This concludes the proof.
\end{proof}

\begin{proof}[Proof of Proposition~\ref{prop: subdivision_intersection}]
We start by showing that the right-hand side of \eqref{equation:intersection-of-tent-functions-on-cliques} is a subdivision of $\sS_{G,\sX}$. Assume that $\sC(G)=\{C_1,\ldots,C_t\}$ and that each $\sigma_i\coloneqq\sigma_{C_i}=\{Q_{i,1},\ldots,Q_{i,r_i}\}$ for all $i\in[t]$. For every $i\in[t]$ and $j_i\in[r_i]$, we call
\[
P_{j_1\cdots j_t}\coloneqq \bigcap_{i=1}^t\pi_{C_i}^{-1}(Q_{i,j_i})\,.
\]
First, we have that
\[
\begin{gathered}
\bigcup_{j_1,\ldots,j_t}P_{j_1\cdots j_t} = \bigcup_{j_1,\ldots,j_t}\bigcap_{i=1}^t\pi_{C_i}^{-1}(Q_{i,j_i}) = \bigcap_{i=1}^t\bigcup_{j_i=1}^{r_i}\pi_{C_i}^{-1}(Q_{i,j_i}) = \bigcap_{i=1}^t\pi_{C_i}^{-1}\left(\bigcup_{j_i=1}^{r_i}Q_{i,j_i}\right) = \\
=\bigcap_{i=1}^t\pi_{C_i}^{-1}(\pi_{C_i}(\sS_{G,\sX})) =
\bigcap_{i=1}^t\pi_{C_i}^{-1}(\pi_{C_i}(\conv(\sX))) =
\sS_{G,\sX}\,.
\end{gathered}
\]
Secondly, we need to show that, given two cells $P_{j_1\cdots j_t}$ and $P_{k_1\cdots k_t}$, then their intersection is a face of both of them. Note that each face of $P_{j_1\cdots j_t}$ can be written as
\begin{equation}\label{eq: face subdivision}
\text{face of $P_{j_1\cdots j_t}$}=\bigcap_{i=1}^t\pi_{C_i}^{-1}(\text{face of $Q_{i,j_i}$})\,.
\end{equation}
Then we have
\[
\begin{gathered}
P_{j_1\cdots j_t}\cap P_{k_1\cdots k_t} = \bigcap_{i=1}^t\pi_{C_i}^{-1}(Q_{i,j_i}) \cap \bigcap_{i=1}^t\pi_{C_i}^{-1}(Q_{i,k_i}) =\\ = \bigcap_{i=1}^t\left(\pi_{C_i}^{-1}(Q_{i,j_i})\cap \pi_{C_i}^{-1}(Q_{i,k_i})\right) = \bigcap_{i=1}^t\pi_{C_i}^{-1}(Q_{i,j_i}\cap Q_{i,k_i})\,.
\end{gathered}
\]
Since $\sigma_i$ is a subdivision, then $Q_{i,j_i}\cap Q_{i,k_i}$ is a face of both $Q_{i,j_i}$ and $Q_{i,k_i}$. The conclusion follows from applying \eqref{eq: face subdivision}.

Now fix a clique $C\in \sC(G)$. The tent function $h_{\sX_C,y_C}$ is a piecewise linear function supported on $\conv(\sX_C)$. In particular, for each $x_C\in\conv(\sX_C)$ the value $h_{\sX_C,y_C}(x_C)$ is defined by taking the minimum over a finite number of linear functions. The regions where we choose a specific linear function are exactly the maximal cells of the regular subdivision $\sigma_C$. This set of maximal cells is indexed by a finite set $I_C$. In short, 
\begin{equation} \label{eqn:tent-funtion-as-max}
h_{\sX_C,y_C}(x_C) = \min_{i_C \in I_C} \left\{a^{(C)}_{i_C} x_C + b^{(C)}_{i_C} \right\}\quad\forall\,x_C \in\conv(\sX_C)\,. 
\end{equation}
Hence the sum $h=\sum_{C \in \sC(G)} h_{\sX_C,y_C}$ is equal to 
\begin{equation} \label{eqn:sums-of-tent-funtions-as-max}
\sum_{C \in \sC(G)} \min_{i_C \in I_C} \left\{ a^{(C)}_{i_C} x_C + b^{(C)}_{i_C}  \right\} = \min_{i_C \in I_C \,\, \forall  C \in \sC(G)} \left\{ \sum_{C \in \sC(G)} \left(a^{(C)}_{i_C} x_C + b^{(C)}_{i_C} \right) \right\}.
\end{equation}
The minimum at the right-hand side of \eqref{eqn:sums-of-tent-funtions-as-max} evaluated at a point $z \in \sS_{G,\sX}$ is achieved by a particular $\sum_{C \in \sC(G)} \left(a^{(C)}_{i_C} z_C + b^{(C)}_{i_C} \right)$ if and only if each $a^{(C)}_{i_C} z + b^{(C)}_{i_C}$ achieves the minimum in~\eqref{eqn:tent-funtion-as-max} evaluated at $z$. 
A point $z \in \sS_{G,\sX}$ is in the interior of a $d$-dimensional polytope of the form~\eqref{equation:intersection-of-tent-functions-on-cliques} if and only if for all $C \in \sC(G)$, the projection $z_C$ is in the interior of $\sigma_C$. This is equivalent to the condition that for each $C \in \sC(G)$, the minimum in~\eqref{eqn:tent-funtion-as-max} evaluated at $z_C$ is achieved exactly once, which is again equivalent to the condition that the minimum in~\eqref{eqn:sums-of-tent-funtions-as-max} is achieved exactly once.  This holds if and only if $z$ is in the interior of the subdivision induced by the tent function $h$ in~\eqref{eqn:sums-of-tent-funtions-as-max}. Since the interior points of the two subdivisions are equal, the subdivisions must be equal as well.
\end{proof}

\begin{proof}[Proof of Corollary \ref{cor:faces_of_the_subdivision}]
Let $z \in \sS_{G,\sX}$ and let $S \in \sigma$ be the minimal face of $\sigma$ that contains $z$.
By Proposition~\ref{prop: subdivision_intersection}, we can write $S = \bigcap_{C\in\sC(G)} \pi^{-1}_C(S_C)$ where $S_C \in \sigma_C$. Then 
\begin{equation}\label{eq: chain ineq k_z}
k_z = \codim(S) \le \sum_{C \in \sC(G)} \codim (\pi^{-1}_C(S_C)) = \sum_{C \in \sC(G)} \codim (S_C) \le \sum_{C \in \sC(G)}k_{z,C}\,.
\end{equation}
The second equality in \eqref{eq: chain ineq k_z} holds since $\pi_C$ is a projection.
The second inequality in \eqref{eq: chain ineq k_z} follows since $\codim(S_C)\le k_{z,C}$ by definition of $k_{z,C}$.\qedhere
\end{proof}

\begin{proof}[Proof of Corollary~\ref{corollary:tent_poles}]
Being a tent pole is the same as being a vertex in the subdivision $\sigma$, or equivalently belonging to a face of codimension $d$.
\end{proof}

\section{Data samples used in Section~\ref{sec:optimization}}\label{app: sec optimization}

\begin{table}[H]
\centering
\resizebox{\textwidth}{!}{
\begin{tabular}{|rr|}
\hline
$x_1$ & $x_2$ \\
\hline
-3.458 & -0.490 \\ 
1.322 & -0.376 \\ 
-1.573 & 1.868 \\ 
2.192 & 0.152 \\ 
0.147 & -1.039 \\ 
-0.301 & -0.959 \\ 
-1.185 & 1.159 \\ 
1.124 & -0.239 \\ 
0.383 & -0.152 \\ 
-0.231 & 0.450 \\ 
\hline
\end{tabular}
\begin{tabular}{|rr|}
\hline
$x_1$ & $x_2$ \\
\hline
0.948 & 0.701 \\ 
0.095 & -1.323 \\ 
-0.654 & 0.171 \\ 
0.886 & 0.215 \\ 
1.187 & 1.055 \\ 
1.502 & -0.088 \\ 
0.139 & 1.945 \\ 
0.876 & 0.792 \\ 
-2.246 & 1.446 \\ 
-0.231 & 0.145 \\ 
\hline
\end{tabular}
\begin{tabular}{|rr|}
\hline
$x_1$ & $x_2$ \\
\hline
1.175 & 0.380 \\ 
1.702 & 0.736 \\ 
0.546 & 0.250 \\ 
-0.857 & -1.406 \\ 
0.093 & 0.048 \\ 
1.659 & -0.268 \\ 
-1.806 & 1.660 \\ 
-0.625 & -1.604 \\ 
0.281 & 1.461 \\ 
1.418 & 0.039 \\ 
\hline
\end{tabular}
\begin{tabular}{|rr|}
\hline
$x_1$ & $x_2$ \\
\hline
0.414 & -1.158 \\ 
-0.222 & -1.161 \\ 
-1.222 & -0.164 \\ 
-1.209 & 0.292 \\ 
0.130 & 0.465 \\ 
-0.076 & 0.210 \\ 
-0.592 & 0.864 \\ 
0.605 & -0.426 \\ 
-0.502 & 0.116 \\ 
0.363 & 0.629 \\ 
\hline
\end{tabular}
\begin{tabular}{|rr|}
\hline
$x_1$ & $x_2$ \\
\hline
0.528 & -1.890 \\ 
-0.368 & -0.396 \\ 
0.793 & 0.465 \\ 
0.062 & -0.517 \\ 
2.381 & 2.129 \\ 
0.036 & -0.106 \\ 
-0.912 & 1.159 \\ 
2.530 & 0.598 \\ 
-0.756 & 0.036 \\ 
0.036 & -0.585 \\ 
\hline
\end{tabular}
}
\vspace*{1mm}
\caption{The sample $\sX$ used for the ML estimation of Example \ref{ex: MLE indModel}.}\label{table: sample 50 points indModel}
\end{table}

\begingroup
\renewcommand{\arraystretch}{1.3}
\begin{table}[htbp]
\centering
\resizebox{\textwidth}{!}{
\begin{tabular}{|r|r|r|r|r|r|r|r|r|r|r|r|r|r|r|}
\hline
$n$ & 10 & 20 & 30 & 40 & 50 & 60 & 70 & 80 & 90 & 100 & 200 & 300 & 400 & 500\\ 
\hline
& 6.6$\times 10^{-2}$ & 1.6$\times 10^{-2}$ & 1.6$\times 10^{-2}$ & 9.8$\times 10^{-3}$ & 8.4$\times 10^{-3}$ & 6.0$\times 10^{-3}$ & 6.3$\times 10^{-3}$ & 4.4$\times 10^{-3}$ & 3.6$\times 10^{-3}$ & 4.6$\times 10^{-3}$ & 2.6$\times 10^{-3}$ & 2.0$\times 10^{-3}$ & 2.0$\times 10^{-3}$ & 1.4$\times 10^{-3}$ \\ 
& 4.2$\times 10^{-2}$ & 1.1$\times 10^{-2}$ & 1.1$\times 10^{-2}$ & 3.7$\times 10^{-3}$ & 4.3$\times 10^{-3}$ & 2.8$\times 10^{-3}$ & 4.2$\times 10^{-3}$ & 2.1$\times 10^{-3}$ & 8.5$\times 10^{-4}$ & 3.4$\times 10^{-3}$ & 9.3$\times 10^{-4}$ & 5.4$\times 10^{-4}$ & 1.1$\times 10^{-3}$ & 8.1$\times 10^{-4}$ \\
\hline
& 1.3$\times 10^{-1}$ & 3.4$\times 10^{-2}$ & 2.5$\times 10^{-2}$ & 1.5$\times 10^{-2}$ & 1.3$\times 10^{-2}$ & 1.0$\times 10^{-2}$ & 8.9$\times 10^{-3}$ & 6.7$\times 10^{-3}$ & 6.1$\times 10^{-3}$ & 6.5$\times 10^{-3}$ & 3.1$\times 10^{-3}$ & 2.7$\times 10^{-3}$ & 1.9$\times 10^{-3}$ & 1.5$\times 10^{-3}$ \\ 
& 6.9$\times 10^{-2}$ & 2.0$\times 10^{-2}$ & 1.3$\times 10^{-2}$ & 3.8$\times 10^{-3}$ & 5.0$\times 10^{-3}$ & 4.6$\times 10^{-3}$ & 4.6$\times 10^{-3}$ & 1.5$\times 10^{-3}$ & 2.1$\times 10^{-3}$ & 4.4$\times 10^{-3}$ & 1.2$\times 10^{-3}$ & 7.9$\times 10^{-4}$ & 8.6$\times 10^{-4}$ & 5.7$\times 10^{-4}$ \\ 
\hline
\end{tabular}
}
\vspace*{1mm}
\caption{The table refers to Example \ref{ex: MLE indModel}. The first row shows the mean $\ell_2$-distances between the log-concave graphical MLE and the true 2-dimensional Gaussian distribution for some values of $n$ using our method for $E=\emptyset$ over 10 trials. The second row displays the corresponding standard deviations. Similarly the last two rows show mean $\ell_2$-distances and standard deviations using \texttt{LogConcDEAD} for $E=\{\{1,2\}\}$.}
\label{table: mean variance distances}
\end{table}
\endgroup

\begin{table}[H]
\centering
\resizebox{\textwidth}{!}{
\begin{tabular}{|rrr|}
\hline
$x_1$ & $x_2$ & $x_3$\\
\hline
5.906 & 1.743 & -4.934 \\ 
-4.672 & -1.749 & -7.213 \\ 
8.622 & 1.490 & 8.641 \\ 
-1.097 & 4.886 & 6.906 \\ 
-0.322 & 0.615 & 4.244 \\ 
-0.641 & 0.411 & 6.558 \\ 
-2.226 & -2.570 & -2.310 \\ 
-4.278 & 9.503 & 7.090 \\ 
-1.950 & -2.976 & 6.474 \\ 
-0.405 & 5.036 & 3.457 \\ 
-7.097 & -7.087 & -4.342 \\ 
-3.970 & -3.370 & -4.219 \\ 
-1.833 & 8.053 & 1.566 \\ 
1.178 & 2.087 & -2.304 \\ 
2.366 & -7.479 & -1.998 \\ 
\hline
\end{tabular}
\begin{tabular}{|rrr|}
\hline
$x_1$ & $x_2$ & $x_3$\\
\hline
-2.638 & -3.376 & -2.721 \\ 
13.219 & 1.696 & -5.774 \\ 
-1.332 & 1.882 & 0.523 \\ 
-0.629 & 5.117 & 3.900 \\ 
4.604 & 13.679 & 2.010 \\ 
-3.660 & 6.569 & 3.188 \\ 
0.254 & 1.463 & 9.136 \\ 
-10.922 & -7.743 & -2.662 \\ 
7.746 & -4.914 & 4.842 \\ 
3.137 & -0.369 & 2.390 \\ 
-2.455 & -1.277 & 10.255 \\ 
0.209 & 2.817 & -0.091 \\ 
5.635 & 11.880 & -1.510 \\ 
-1.752 & -8.478 & -3.901 \\ 
-0.012 & 2.006 & -2.627 \\ 
\hline
\end{tabular}
\begin{tabular}{|rrr|}
\hline
$x_1$ & $x_2$ & $x_3$\\
\hline
3.914 & -1.609 & 2.219 \\ 
0.902 & 5.642 & -9.180 \\ 
-8.480 & -1.247 & -5.709 \\ 
-3.694 & -0.043 & 3.064 \\ 
0.359 & -1.992 & 0.673 \\ 
-2.979 & -1.343 & -0.596 \\ 
3.537 & -4.447 & -6.351 \\ 
-9.498 & 2.469 & 7.971 \\ 
-0.361 & 7.126 & 14.001 \\ 
-0.773 & -5.056 & 1.343 \\ 
0.156 & -1.064 & -0.113 \\ 
-4.042 & 5.800 & 8.677 \\ 
1.481 & 1.117 & -6.952 \\ 
-7.186 & -3.628 & 3.905 \\ 
0.701 & 8.741 & -4.089 \\ 
\hline
\end{tabular}
\begin{tabular}{|rrr|}
\hline
$x_1$ & $x_2$ & $x_3$\\
\hline
2.679 & -13.096 & 2.009 \\ 
-1.988 & -3.468 & 3.075 \\ 
3.725 & -4.299 & -5.100 \\ 
-9.492 & 3.577 & 1.564 \\ 
-0.188 & 2.285 & -6.891 \\ 
0.529 & -6.187 & 2.643 \\ 
3.704 & 0.905 & -1.815 \\ 
-2.811 & 7.392 & 5.978 \\ 
-2.865 & 6.445 & 1.921 \\ 
0.800 & -1.961 & -11.380 \\ 
3.947 & -0.985 & -0.963 \\ 
-1.644 & -1.156 & 2.944 \\ 
-8.564 & -1.749 & -11.601 \\ 
-5.433 & -2.632 & -7.001 \\ 
-6.590 & -2.610 & 5.949 \\ 
\hline
\end{tabular}
}
\vspace*{1mm}
\caption{The sample $\sX$ used for the ML estimation of Example \ref{ex: optimal-tent-function2}.}\label{table: sample 60 points graph 12-23}
\end{table}

\begin{table}[H]
\centering
\resizebox{\textwidth}{!}{
\begin{tabular}{|r||r|r|r|r|r|r|r|r|r|r|r|r|}
\hline
$x_1$ & -9.492 & -9.492 & 13.219 & 13.219 & 11.101 & -10.105 & -10.922 & -10.922 & -4.278 & -6.370 & 9.315 & 10.382 \\ 
$x_2$ & 3.577 & 3.577 & 1.696 & 1.696 & -1.277 & -1.277 & -7.743 & -7.743 & 9.503 & 7.126 & 7.126 & 5.642 \\ 
$x_3$ & 12.419 & -9.857 & 11.580 & -10.473 & 10.255 & 10.255 & 5.743 & -4.744 & 9.651 & 14.001 & 14.001 & -9.180 \\ 
\hline
$x_1$ & -4.278 & -7.675 & 10.764 & 10.613 & -10.165 & -10.192 & 5.970 & -9.055 & 2.679 & 0.778 & 5.897 & 4.604 \\
$x_2$ & 9.503 & 5.642 & -1.749 & -1.961 & -1.749 & -1.961 & -8.478 & -8.478 & -13.096 & 11.880 & 11.880 & 13.679 \\
$x_3$ & -4.433 & -9.180 & -11.601 & -11.380 & -11.601 & -11.380 & -3.901 & -3.901 & 2.009 & -1.510 & -1.510 & 2.010 \\
\hline
\end{tabular}
}
\vspace*{1mm}
\caption{The 24 vertices of the polytope $\sS_{G,\sX}$ of Example \ref{ex: optimal-tent-function2}.}\label{table: vertices support MLE graph 12-23}
\end{table}

\section{Proofs of Section~\ref{sec:decomposition-of-convex-functions}}\label{app: section decomposition convex}

\tikzstyle{thm} = [rectangle, rounded corners, minimum width=3cm, minimum height=1cm,text centered, draw=black, fill=red!30]
\tikzstyle{nothm} = [rectangle, rounded corners, minimum width=3cm, minimum height=1cm,text centered, draw=black]
\tikzstyle{arrow} = [thick,->,>=stealth]

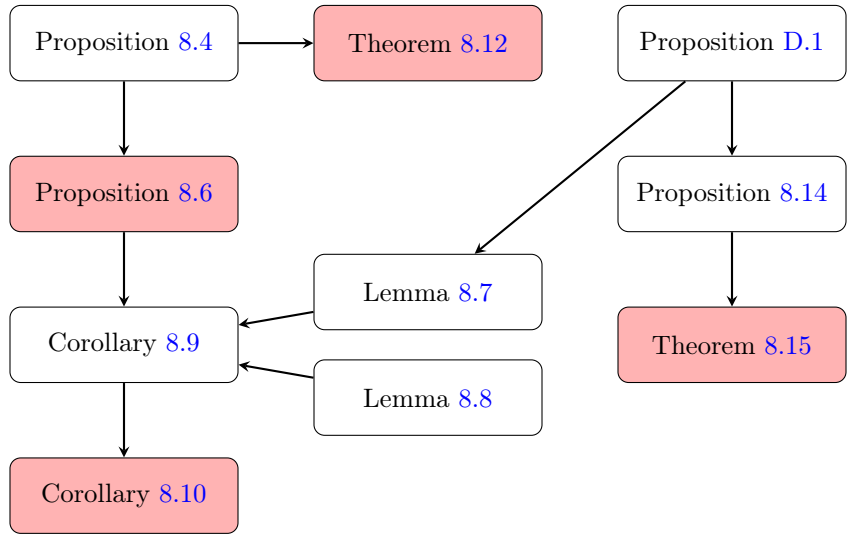
\begin{figure}[H]
\centering
\begin{tikzpicture}[node distance=2cm]
\node (prop chordal iff) [nothm] {Proposition~\ref{prop: chordal if and only if the running intersection property}};
\node (thm suff dec chordal) [thm, right of=prop chordal iff, xshift=2cm] {Theorem~\ref{thm: sufficient condition for convex decomposition on chordal graphs}};
\node (prop theta) [thm, below of=prop chordal iff] {Proposition~\ref{prop: convex decomposition iff theta_i}};
\node (cor Gaussian) [nothm, below of=prop theta] {Corollary~\ref{corollary: Gaussian convex decomposition}};
\node (lem sepmat) [nothm, right of=cor Gaussian, xshift=2cm, yshift=0.7cm] {Lemma~\ref{lem: sep matrix}};
\node (lem twicediff) [nothm, right of=cor Gaussian, xshift=2cm, yshift=-0.7cm] {Lemma~\ref{lem: assumption twice differentiable}};
\node (cor Gaussian fact) [thm, below of=cor Gaussian] {Corollary~\ref{corollary: Gaussian log-concave factorization}};
\node (prop Schur) [nothm, right of=thm suff dec chordal, xshift=2cm] {Proposition~\ref{prop: generalized Schur}};
\node (prop convdec) [nothm, below of=prop Schur] {Proposition~\ref{prop2}};
\node (thm convdec) [thm, below of=prop convdec] {Theorem~\ref{thm: convex decomposition}};
\draw [arrow] (prop chordal iff) -- (prop theta);
\draw [arrow] (prop chordal iff) -- (thm suff dec chordal);
\draw [arrow] (prop theta) -- (cor Gaussian);
\draw [arrow] (lem sepmat) -- (cor Gaussian);
\draw [arrow] (lem twicediff) -- (cor Gaussian);
\draw [arrow] (cor Gaussian) -- (cor Gaussian fact);
\draw [arrow] (prop Schur) -- (lem sepmat);
\draw [arrow] (prop Schur) -- (prop convdec);
\draw [arrow] (prop convdec) -- (thm convdec);
\end{tikzpicture}
\caption{A map of Section \ref{sec:decomposition-of-convex-functions} showing how the results in this section relate to each other.}
\label{fig: section 8 map}
\end{figure}

Let us recall some linear algebra background. Given a matrix $A\in\R^{m\times n}$ of rank $r$, let $A=U\Sigma V^T$ be the (compact) singular value decomposition of $A$, where the columns of $U\in\R^{m\times r}$ and $V\in\R^{n\times r}$ are orthonormal and $\Sigma\in\R^{r\times r}$ is the diagonal matrix $\Sigma=\mathrm{diag}(\sigma_1,\ldots,\sigma_r)$, where $\sigma_1\ge\cdots\ge\sigma_r>0$ are the positive singular values of $A$. The {\em (Moore-Penrose) pseudoinverse} of $A$ is $A^\dagger\coloneqq V\Sigma^{-1}U^T\in\R^{n\times m}$. If $\rank(A)=n$, then $A^\dagger=(A^TA)^{-1}A^T$. If $\rank(A)=m$, then $A^\dagger = A^T(AA^T)^{-1}$. If $m=n$ and $A$ is nonsingular, then $A^\dagger = A^{-1}$.

\noindent Now consider the symmetric matrix $M$ with the following block structure:
\begin{equation}\label{eq: block matrix M}
M = \begin{bmatrix}A & B \\ B^T & C\end{bmatrix}\,,
\end{equation}
If $A$ is nonsigular, then the {\em Schur complement} of $A$ in $M$ is defined to be $C-B^TA^{-1}B$. Similarly, if $C$ is nonsingular, the Schur complement of $C$ in $M$ is the matrix $A-B C^{-1} B^T$. Moreover, the {\em generalized Schur complements} of $A$ and $C$ in $M$ are the matrices $C-B^TA^\dagger B$ and $A-B C^{\dagger} B^T$ respectively. The following property involving generalized Schur complements, which is proved in \cite[Appendix A.5]{boyd2004convex}, is a key tool for deriving the subsequent results.

\begin{proposition}\label{prop: generalized Schur}
Consider the symmetric block matrix $M$ in \eqref{eq: block matrix M}. The following are equivalent:
\begin{itemize}
\item[$(i)$] $M\succcurlyeq 0$.
\item[$(ii)$] $A\succcurlyeq 0$, $(I-AA^\dagger)B=0$, and $C-B^T A^\dagger B\succcurlyeq 0$.
\item[$(iii)$] $C\succcurlyeq 0$, $(I-CC^\dagger)B^T=0$, and $A-BC^\dagger B^T\succcurlyeq 0$.
\end{itemize}
\end{proposition}

\begin{proof}[Proof of Proposition~\ref{prop: convex decomposition iff theta_i}]
$(i)\Longrightarrow(ii)$ Suppose that there exist convex functions $\hat{g}_1,\ldots,\hat{g}_t$ as described in $(i)$.
We claim that for all $i\in[t-1]$, there exists a function $\theta_i\colon\R^{C_i \cap C_{[i+1,t]}}\to\R$ such that the function
\[
g_i(x_{C_i}) + \sum_{j\in A_i} \theta_j (x_{C_j \cap C_{[j+1,t]}}) - \theta_i(x_{C_i \cap C_{[i+1,t]}})
\]
is convex, and for all $x_{C_{[i+1,t]}}\in\R^{C_{[i+1,t]}}$,
\[
\sum_{j=i+1}^t g_j(x_{C_j}) + \sum_{j=i+1}^t \sum_{k\in A_j\cap[i]} \theta_k(x_{C_k \cap C_{[k+1,t]}}) = \sum_{j=i+1}^t \hat{g}_j(x_{C_j})\,.
\]
For $i=1$, assume $a \in \R^{C_1 \setminus C_{[2,t]}}$ and define the function $\theta_1\colon\R^{C_1 \cap C_{[2,t]}}\to\R$ as
\begin{align*}
\theta_1(x_{C_1 \cap C_{[2,t]}}) \coloneqq g_1(a,x_{C_1 \cap C_{[2,t]}})-\hat{g}_1(a,x_{C_1 \cap C_{[2,t]}})\,.
\end{align*}
Since $\sum_{j=1}^t g_j(x_{C_j}) = f(x) =  \sum_{j=1}^t \hat{g}_j(x_{C_j})$ for all $x\in \R^d$, then
\[
g_1(a,x_{C_1 \cap C_{[2,t]}}) + \sum_{j=2}^t g_j(x_{C_j}) = f(a, x_{C_{[2,t]}}) =  \hat{g}_1(a,x_{C_1 \cap C_{[2,t]}}) + \sum_{j=2}^t \hat{g}_j(x_{C_j})\,,
\]
and thus,
\begin{equation}\label{q3:2}
\theta_1(x_{C_1 \cap C_{[2,t]}}) =\sum_{j=2}^t \hat{g}_j(x_{C_j}) - \sum_{j=2}^t g_j(x_{C_j}) = g_1(x_{C_1})-\hat{g}_1(x_{C_1}).
\end{equation}
By the second equality in \eqref{q3:2}, $g_1(x_{C_1}) - \theta_1(x_{C_1 \cap C_{[2,t]}}) =\hat{g}_1(x_{C_1})$, and thus, is convex. Also, by the first equality in \eqref{q3:2}, 
$\sum_{j=2}^t g_j(x_{C_j}) + \theta_1 (x_{C_1 \cap C_{[2,t]}}) = \sum_{j=2}^t \hat{g}_j(x_{C_j})$ as desired.

\noindent
Now assume that the claim holds for all $i = k$, where $k\in[t-2]$. Then for $i=k+1$, let $a \in \R^{C_{k+1} \setminus C_{[k+2,t]}}$ and define the function $\theta_{k+1}\colon\R^{C_{k+1} \cap C_{[k+2,t]}} \to \R$ as
\begin{align*}
\theta_{k+1}(x_{C_{k+1} \cap C_{[k+2,t]}}) &\coloneqq g_{k+1}(a,x_{C_{k+1} \cap C_{[k+2,t]}}) -\hat{g}_{k+1}(a,x_{C_{k+1} \cap C_{[k+2,t]}})\\
&\quad+\sum_{j\in A_{k+1}} \theta_j (a_{(C_j \cap C_{k+1}) \setminus C_{[k+2,t]}},x_{(C_j \cap C_{k+1}) \cap C_{[k+2,t]}})\,.
\end{align*}
Since by the induction hypothesis, for all $x_{C_{[k+1,t]}}\in \R^{C_{[k+1,t]}}$,
\begin{equation}
\sum_{j=k+1}^t g_j(x_{C_j}) + \sum_{j=k+1}^t \sum_{l\in A_j \cap [k]} \theta_l (x_{C_l \cap C_{[l+1,t]}}) = \sum_{j=k+1}^t \hat{g}_j(x_{C_j})\,,
\end{equation}
then
\[
\begin{gathered}
g_{k+1}(a, x_{C_{k+1}\cap C_{[k+2,t]}}) +
\sum_{j \in A_{k+1}} \theta_j (a_{(C_j \cap C_{k+1}) \setminus C_{[k+2,t]}}, x_{(C_j \cap C_{k+1}) \cap C_{[k+2,t]}})+\\
+ \sum_{j=k+2}^t g_j(x_{C_j}) + \sum_{j=k+2}^t \sum_{l\in A_j \cap [k]} \theta_l (x_{C_l \cap C_{[l+1,t]}}) = \hat{g}_{k+1}(a, x_{C_{k+1} \cap C_{[k+2,t]}}) + \sum_{j=k+2}^t \hat{g}_j(x_{C_j})\,.
\end{gathered}
\]
Thus
\begin{align}
\theta_{k+1}(x_{C_{k+1} \cap C_{[k+2,t]}}) &= \sum_{j=k+2}^t \hat{g}_j(x_{C_j}) - \sum_{j=k+2}^t g_j(x_{C_j}) - \sum_{j=k+2}^t \sum_{l\in A_j \cap [k]} \theta_l (x_{C_l \cap C_{[l+1,t]}}) \label{q3:3}\\
&= g_{k+1}(x_{C_{k+1}}) - \hat{g}_{k+1}(x_{C_{k+1}}) + \sum_{j\in A_{k+1}} \theta_j(x_{C_j \cap C_{[j+1,t]}})\label{q3:4}\,.
\end{align}
By \eqref{q3:4}
\[
g_{k+1}(x_{C_{k+1}}) +  \sum_{j\in A_{k+1}} \theta_j(x_{C_j \cap C_{[j+1,t]}}) - \theta_{k+1}(x_{C_{k+1} \cap C_{[k+2,t]}})  = \hat{g}_{k+1}(x_{C_{k+1}})\,,
\]
and so, the function on the left hand side is convex. Also, by \eqref{q3:3},
\begin{align*}
\sum_{j=k+2}^t \hat{g}_j(x_{C_j}) &= \sum_{j=k+2}^t g_j(x_{C_j}) + \sum_{j=k+2}^t \sum_{l\in A_j \cap [k]} \theta_l (x_{C_l \cap C_{[l+1,t]}}) + \theta_{k+1}(x_{C_{k+1} \cap C_{[k+2,t]}})\\
&= \sum_{j=k+2}^t g_j(x_{C_j}) + \sum_{j=k+2}^t \sum_{l\in A_j \cap [k+1]} \theta_l (x_{C_l \cap C_{[l+1,t]}})
\end{align*}
as desired. This proves that the claim is correct, and therefore, one can conclude the proof of this side of the proposition by applying $i=1,\ldots,i=t-1$ to the first part of the claim as well as applying $i=t-1$ to the second part.

$(ii)\Longrightarrow(i)$ Suppose that there exist functions $\theta_1,\ldots,\theta_{t-1}$ as in $(ii)$. Define
\begin{equation}
\hat{g}_i(x_{C_i}) \coloneqq
\begin{cases}
g_i(x_{C_i}) + \sum_{j\in A_i} \theta_j (x_{C_j \cap C_{[j+1,t]}}) - \theta_i(x_{C_i \cap C_{[i+1,t]}}) & \text{if $i\in[t-1]$}\,,\\
g_t(x_{C_t}) + \sum_{j\in A_t} \theta_j (x_{C_j \cap C_{[j+1,t]}}) & \text{if $i=t$}\,.
\end{cases}
\end{equation}
It immediately follows that all $\hat{g}_i$'s are convex and $\sum_{i=1}^t \hat{g}_i(x_{C_i}) = \sum_{i=1}^t g_i (x_{C_i}) = f(x)$ for all $x\in\R^d$.
\end{proof}

\begin{proof}[Proof of Lemma~\ref{lem: sep matrix}]
Recall the block structure of $M$ in \eqref{eq: block structure of a psd matrix}. Since $M\succcurlyeq 0$, by Proposition \ref{prop: generalized Schur}, the generalized Schur complement of $M_{A,A}$ in $M$ is also positive semidefinite, i.e.
\begin{equation}\label{eq: block matrix gaussian}
\begin{gathered}
\begin{bmatrix}
M_{B,B} & M_{B,C} \\
M_{B,C}^T & M_{C,C}
\end{bmatrix}
-\begin{bmatrix}
M_{A,B}^T \\ 0
\end{bmatrix}
\cdot M_{AA}^\dagger\cdot
\begin{bmatrix}
M_{A,B} & 0
\end{bmatrix}
=\\
=\begin{bmatrix}
M_{B,B} - M_{A,B}^T \cdot M_{A,A}^\dagger \cdot M_{A,B} & M_{B,C} \\
M_{B,C}^T & M_{C,C} 
\end{bmatrix}\succcurlyeq 0\,.
\end{gathered}
\end{equation}
Again by Proposition \ref{prop: generalized Schur}, the generalized Schur complement of $M_{C,C}$ in \eqref{eq: block matrix gaussian} is positive semidefinite:
\begin{align} \label{eq: Schur positive semidefinite}
    M_{B,B} - M_{A,B}^T \cdot M_{A,A}^\dagger \cdot M_{A,B} - M_{B,C} \cdot M_{C,C} ^\dagger \cdot M_{B,C}^T\succcurlyeq 0.
\end{align}

Define $S_{M,N,\sP} \coloneqq -M_{B,B} + N + M_{B,C} \cdot M_{C,C} ^\dagger \cdot M_{B,C}^T$. Then \eqref{eq: Schur positive semidefinite} implies that the generalized Schur complement of $M_{A,A}$ in the first matrix in \eqref{eq: Gaussian separation matrices} is positive semidefinite. Also, applying Proposition \ref{prop: generalized Schur} twice to the first block structure of $M$ in \eqref{eq: block structure of a psd matrix}, we conclude that $M_{A,A}\succcurlyeq 0$ and $(I- M_{A,A} \cdot M_{A,A}^\dagger) M_{A,B}=0$. Thus, the first matrix in \eqref{eq: Gaussian separation matrices} is positive semidefinite by Proposition \ref{prop: generalized Schur}. 

Similarly, applying Proposition \ref{prop: generalized Schur} with respect to $M_{C,C}$, we have $M_{C,C}\succcurlyeq 0$ and $(I- M_{C,C} \cdot M_{C,C}^\dagger) M_{B,C}^T =0$. Moreover, the generalized Schur complement of $M_{C,C}$ in the second matrix in \eqref{eq: Gaussian separation matrices} is 0. Hence, the second matrix in \eqref{eq: Gaussian separation matrices} is also positive semidefinite by Proposition \ref{prop: generalized Schur}.
\end{proof}

\begin{proof}[Proof of Lemma \ref{lem: assumption twice differentiable}]
We first prove that, for all $\ell,m \in V$, if $(\ell,m)\notin E$, then $K_{\ell m}=0$. Assume $(\ell,m)\notin E$. Then there is no maximal clique in $G$ containing both $\ell$ and $m$. Without loss of generality, assume that the first $s$ cliques $C_1,\ldots,C_s$ are the maximal cliques in $\sC(G)$ containing the vertex $\ell$. Then
\begin{align*}
   2\sum_{j=1}^d K_{\ell j} (x_j-\mu_j)
   &= \frac{\partial f(x)}{\partial x_\ell}\\
   &= \lim_{h\to 0}\frac{f(x_\ell+h,x_{V \setminus \{\ell\}}) - f(x)}{h}\\
   &= \lim_{h\to 0}\frac{\sum_{i=1}^s g_i(x_\ell+h, x_{C_i \setminus \{\ell\}}) + \sum_{i=s+1}^t g_i(x_{C_i}) - \sum_{i=1}^t g_i(x_{C_{i}})}{h} \\
   &= \lim_{h\to 0}\frac{\sum_{i=1}^s(g_i(x_\ell+h, x_{C_i \setminus \{\ell\}})-g_i(x_{C_{i}}))}{h}\\
   &= \kappa (x_{C_{[1,s]}})
\end{align*}
for some function $\kappa\colon\R^{C_{[1,s]}}\to\R$. Then $2K_{\ell m}= \partial\kappa(x_{C_{[1,s]}})/\partial x_m = 0$ since $m \not\in C_{[1,s]}$. Thus, $K_{\ell m}=0$ as desired.
Now for all $i\in [t]$, define the matrix $K^{(i)}\in\R^{d\times d}$ as follows:
\begin{align*}
    K_{\ell m}^{(i)}=\begin{cases}
    K_{\ell m} & \text{if $\{\ell,m\}\subseteq C_i$ and $\{\ell,m\}\not\subset C_{[i+1,t]}$}\\
    0 & \text{otherwise.}
    \end{cases}
\end{align*}
For $\ell,m\in V$, if $(\ell,m)\in E$, let $i_0 \coloneqq \max \{i \in [t] \mid \{\ell,m\}\subseteq C_i\}$. Then $K_{\ell m}^{(i_0)} = K_{\ell m}$ and $K_{\ell m}^{(i)} = 0$ for all $i \in [t] \setminus \{i_0\}$. Furthermore, if $(\ell,m)\notin E$, on the one hand, by the claim previously proved, $K_{\ell m}=0$, and on the other hand, there is no maximal clique containing $\ell$ and $m$, and thus, $K_{\ell m}^{(i)}=0$ for all $i\in[t]$. So, in both cases, $K_{\ell m} = \sum_{i=1}^t K_{\ell m}^{(i)}$, which means that $f(x) = \sum_{i=1}^t (x-\mu)^T K^{(i)} (x-\mu)$. So, without loss of generality, we can assume that $g_i = (x-\mu)^T K^{(i)} (x-\mu)$, and thus, $g_i$ is twice differentiable for all $i\in[t]$.
\end{proof}

\begin{proof} [Proof of Corollary~\ref{corollary: Gaussian convex decomposition}]
By Lemma \ref{lem: assumption twice differentiable} we can assume that $g_i$ is twice differentiable for all $i \in [t]$.
Applying Lemma \ref{lem: sep matrix} to $M=K$, $N={H_{g_1}}|_{C_1 \cap C_{[2,t]}}$ and $\mathcal{P}=(C_1 \setminus C_{[2,t]},C_1 \cap C_{[2,t]}, C_{[2,t]}\setminus C_1)$, we define
\[
\theta_1 (x_{C_1 \cap C_{[2,t]}}) \coloneqq  x_{C_1 \cap C_{[2,t]}}^T\cdot S_{M,N,\sP}\cdot x_{C_1 \cap C_{[2,t]}}\,.
\]
Also, for all $i\in\{2,\ldots,t-1\}$, we apply Lemma \ref{lem: sep matrix} to $M=H_{\sum_{j=i}^t (g_j +\sum_{k \in A_{j} \cap [i-1]} \theta_k)}$, $N={H_{g_i + \sum_{j \in A_i} \theta_j}}|_{C_i \cap C_{[i+1,t]}}$ and $\mathcal{P}=(V \setminus C_{[i+1,t]},C_i \cap C_{[i+1,t]}, C_{[i+1,t]}\setminus C_i)$, and define
\[
\theta_i (x_{C_i \cap C_{[i+1,t]}}) \coloneqq x_{C_i \cap C_{[i+1,t]}}^T\cdot S_{M,N,\sP}\cdot x_{C_i \cap C_{[i+1,t]}}\,.
\]
The statement follows from Proposition \ref{prop: convex decomposition iff theta_i}.
\end{proof}

\begin{proof}[Proof of Corollary~\ref{corollary: Gaussian log-concave factorization}]
By Lemma \ref{lem: assumption twice differentiable} we can assume that $g_i$ is twice differentiable for all $i \in [t]$.
Applying Lemma \ref{lem: sep matrix} to $M=K$, $N={H_{g_1}}|_{C_1 \cap C_{[2,t]}}$ and $\mathcal{P}=(C_1 \setminus C_{[2,t]},C_1 \cap C_{[2,t]}, C_{[2,t]}\setminus C_1)$, we define
\[
\theta_1 (x_{C_1 \cap C_{[2,t]}}) \coloneqq  x_{C_1 \cap C_{[2,t]}}^T\cdot S_{M,N,\sP}\cdot x_{C_1 \cap C_{[2,t]}}\,.
\]
Also, for all $i\in\{2,\ldots,t-1\}$, we apply Lemma \ref{lem: sep matrix} to $M=H_{\sum_{j=i}^t (g_j +\sum_{k \in A_{j} \cap [i-1]} \theta_k)}$, $N={H_{g_i + \sum_{j \in A_i} \theta_j}}|_{C_i \cap C_{[i+1,t]}}$ and $\mathcal{P}=(V \setminus C_{[i+1,t]},C_i \cap C_{[i+1,t]}, C_{[i+1,t]}\setminus C_i)$, and define
\[
\theta_i (x_{C_i \cap C_{[i+1,t]}}) \coloneqq x_{C_i \cap C_{[i+1,t]}}^T\cdot S_{M,N,\sP}\cdot x_{C_i \cap C_{[i+1,t]}}\,.
\]
The statement follows from Proposition \ref{prop: convex decomposition iff theta_i}.
\end{proof}

\begin{proof}[Proof of Proposition~\ref{prop2}]
Given that $f(x_A, x_B, x_C)=g(x_A, x_B)+h(x_B,x_C)$ for all $(x_A, x_B, x_C)\in X$, the Hessian matrix of $f$ has the following block structure.
\begin{equation}\label{eq: block matrix 1}
H_f =
\begingroup
\left[\begin{array}{c|cc}
{H_g}|_{A} & {H_g}|_{A,B} & 0\\\hline
{{H_g}|_{A,B}}^T &  {H_g}|_{B}+ {H_h}|_{B} &  {H_h}|_{B,C}\\
0 & {{H_h}|_{B,C}}^T &  {H_h}|_{C}
\end{array}\right]=
\left[\begin{array}{cc|c}
{H_g}|_{A} & {H_g}|_{A,B} & 0\\
{{H_g}|_{A,B}}^T &  {H_g}|_{B}+ {H_h}|_{B} &  {H_h}|_{B,C}\\\hline
0 & {{H_h}|_{B,C}}^T &  {H_h}|_{C}
\end{array}\right]\,.
\endgroup
\end{equation}
Since $f$ is convex on $X$, we have $H_f(x)\succcurlyeq 0$ for all $x\in X$. Moreover, since ${H_f}|_{A}={H_g}|_{A}$ is invertible on $X$, we have ${{H_g}|_{A}}^\dagger={{H_g}|_{A}}^{-1}$ on $X$.
We apply Proposition \ref{prop: generalized Schur} to the first block arrangement of $H_f$ in \eqref{eq: block matrix 1}. In particular, on the set $X$ the Schur complement of ${H_g}|_{A}$ in $H_f$ is positive semidefinite:
\begin{equation}
\begin{bmatrix}
{H_g}|_{B}+ {H_h}|_{B} &  {H_h}|_{B,C}\\
{{H_h}|_{B,C}}^T &  {H_h}|_{C}
\end{bmatrix}-
\begin{bmatrix}{{H_g}|_{A,B}}^T \\ 0\end{bmatrix}
\cdot{{H_g}|_{A}}^{-1}\cdot
\begin{bmatrix}{H_g}|_{A,B}  & 0\end{bmatrix}\succcurlyeq 0\,.
\end{equation}
Equivalently, we have that
\begin{equation}\label{eq: block matrix 2}
\begin{bmatrix}
{H_g}|_{B}+ {H_h}|_{B}-{{H_g}|_{A,B}}^T\cdot{{H_g}|_{A}}^{-1}\cdot {H_g}|_{A,B} &  {H_h}|_{B,C}\\
{{H_h}|_{B,C}}^T &  {H_h}|_{C}
\end{bmatrix}
\succcurlyeq 0\,.
\end{equation}
Again by Proposition \ref{prop: generalized Schur}, the Schur complement of ${H_h}|_{C}$ in the block matrix in \eqref{eq: block matrix 2} is positive semidefinite:
\begin{equation}\label{eq: Schur complement HhCC}
{H_g}|_{B}+ {H_h}|_{B}-{{H_g}|_{A,B}}^T\cdot{{H_g}|_{A}}^{-1}\cdot{H_g}|_{A,B} -{H_h}|_{B,C}\cdot{{H_h}|_{C}}^\dagger\cdot{{H_h}|_{B,C}}^T\succcurlyeq 0\,.
\end{equation}
Since $|B|=1$, each of the terms in \eqref{eq: Schur complement HhCC} is a $1\times 1$ matrix.
Define
\[
\varphi_g\colon\pi_{A\cup B}(X)\to\R\,,\quad\varphi_g(x_A,x_B)\coloneqq\left(-{H_g}|_{B}+{{H_g}|_{A,B}}^T {{H_g}|_{A}}^{-1} {H_g}|_{A,B}\right)(x_A,x_B)\,.
\]
Note that $\varphi_g$ is a continuous function on $\pi_{A\cup B}(X)$.
Now assume that $x_C^0\in \pi_C(X)$. Then by \eqref{eq: Schur complement HhCC}, for all $(x_A,x_B)\in \pi_{A\cup B}(\pi^{-1}_C(\{x_C^0\}))$,
\begin{equation}
{H_h}|_{B}(x_B, x_C^0)-{H_h}|_{B,C}\cdot{{H_h}|_{C}}^\dagger\cdot{{H_h}|_{B,C}}^T(x_B, x_C^0) \succcurlyeq \varphi_g(x_A, x_B)\,.
\end{equation}
Consider the function $\tilde{\gamma}\colon\pi_{B}(X)\to\R$ defined by
\begin{equation}\label{eq: def function s}
\tilde{\gamma}(x_B)\coloneqq\sup_{x_A\in Y} \varphi_g(x_A,x_B)\,.
\end{equation}
We claim that $\tilde{\gamma}$ is continuous on its domain $\pi_{B}(X)$. Let $a\in \R$. Then $\tilde{\gamma}^{-1}(a,+\infty) = \pi_{B}\left(\varphi_g^{-1}(a,+\infty)\right)$, which means that $\tilde{\gamma}^{-1}(a,+\infty)$ is open in $\pi_B(X)$. Moreover, if $x_B^0\in \tilde{\gamma}^{-1}(-\infty, a)$, then for all $x_A\in Y$, $\varphi_g(x_A,x_B^0)<a$. So, for all $x_A\in Y$, there exist open sets $V_{x_A}\subseteq Y$ and $W_{x_A}\subseteq \pi_B(X)$ such that 
\[
(x_A,x_B^0)\in V_{x_A}\times W_{x_A} \subseteq \varphi_g^{-1}(-\infty,a).
\]
Given that $Y=\bigcup_{x_A\in Y} V_{x_A}$ and $Y$ is compact, there exist $x_A^{1},\ldots,x_A^{m}\in Y$ such that $Y=\bigcup_{i=1}^m V_{x_{A}^i}$. So, if $x_B\in \bigcap_{i=1}^m W_{x_A^i}$, then $\varphi_g(x_A, x_B) <a$ for all $x_A\in Y$, which implies that $x_B^0 \in \bigcap_{i=1}^m W_{x_A^i} \subseteq \tilde{\gamma}^{-1}(-\infty, a)$. So, $\tilde{\gamma}^{-1}(-\infty, a)$ is also open in $\pi_B(X)$. Hence, $\tilde{\gamma}$ is a continuous function. So, there exists a function $\gamma\in \sC^2(\pi_{B}(X))$ such that $\gamma^{\prime\prime}=\tilde{\gamma}$.

Now we show that $\hat{g}=g+\gamma$ and $\hat{h}=h-\gamma$ are two convex functions. Since $H_f\succcurlyeq 0$, applying Proposition \ref{prop: generalized Schur} to the second block arrangement of $H_f$ in \eqref{eq: block matrix 1}, we have that
\begin{align*}
\begin{bmatrix}
{H_g}|_{A} & {H_g}|_{A,B}\\
{{H_g}|_{A,B}}^T &  {H_g}|_{B}+ {H_h}|_{B}
\end{bmatrix}
\succcurlyeq 0\,.
\end{align*}
In turn, this implies that ${H_g}|_{A}\succcurlyeq 0$. Moreover, by the choice of $\gamma$, for all $(x_A,x_B)\in \pi_{A\cup B}(X),$
\[
\begin{gathered}
\left({H_g}|_{B}+H_\gamma-{{H_g}|_{A,B}}^T\cdot{{H_g}|_{A}}^\dagger\cdot{H_g}|_{A,B}\right)(x_A,x_B) = \\
= \left({H_g}|_{B}+\tilde{\gamma}-{{H_g}|_{A,B}}^T\cdot{{H_g}|_{A}}^\dagger\cdot{H_g}|_{A,B}\right)(x_A,x_B)\succcurlyeq 0\,. 
\end{gathered}
\]
So, by Proposition \ref{prop: generalized Schur}, on $\pi_{A\cup B}(X)$,
\[
H_{\hat{g}}=\begin{bmatrix}{H_g}|_{A} & {H_g}|_{A,B} \\ {{H_g}|_{A,B}}^T & {H_g}|_{B}+H_{\gamma}\end{bmatrix}\succcurlyeq 0\,.
\]
This proves that $\hat{g} = g+\gamma$ is convex on $\pi_{A\cup B}(X)$.
A similar argument shows that $\hat{h} = h-\gamma$ is convex on $\pi_{B\cup C}(X)$. 

Now assume that $D\subseteq C$ and ${H_f}|_{A\cup B\cup D}$ is invertible on $X$. For all $(x_B^0, x_D^0)\in \pi_{B\cup D}(X)$, given that $Y$ is compact, there exists $x_A^0 \in Y$ such that 
\begin{equation}\label{eq: gamma'' equals phi_g}
\gamma^{\prime\prime}(x_B^0) = \tilde{\gamma}(x_B^0)= \varphi_g(x_A^0,x_B^0)\,.
\end{equation}
Consider the Hessian matrix
\[
{H_f}|_{A\cup B\cup D} =
\begin{bmatrix}
{H_g}|_{A} & {H_g}|_{A,B}  & 0\\
{{H_g}|_{A,B}}^T &  {H_g}|_{B}+ {H_h}|_{B} &  {H_h}|_{B,D}\\
0 & {{H_h}|_{B,D}}^T &  {H_h}|_{D}
\end{bmatrix}\,.
\]
Since ${H_f}|_{A\cup B\cup D}$ is invertible on $X$, using the block matrix determinant formula,
\begin{equation}\label{eq: matrix determinant formula}
0 \neq \det\left({H_f}|_{A\cup B\cup D}(x_A^0, x_B^0, x_D^0)\right) = \det\left({H_g}|_{A}(x_A^0, x_B^0)\right)\cdot\det(M(x_A^0, x_B^0, x_D^0))\,,
\end{equation}
where
\begin{align*}
M &=
\begin{bmatrix}
{H_g}|_{B}+ {H_h}|_{B} & {H_h}|_{B,D}\\
{{H_h}|_{B,D}}^T &  {H_h}|_{D}
\end{bmatrix}-
\begin{bmatrix}{{H_g}|_{A,B}}^T \\ 0\end{bmatrix}
\cdot{{H_g}|_{A}}^{-1}\cdot
\begin{bmatrix}{H_g}|_{A,B} & 0\end{bmatrix}\\
&=
\begin{bmatrix}
{H_g}|_{B}+ {H_h}|_{B} & {H_h}|_{B,D}\\
{{H_h}|_{B,D}}^T & {H_h}|_{D}
\end{bmatrix}-
\begin{bmatrix}{{H_g}|_{A,B}}^T\cdot{{H_g}|_{A}}^{-1}\cdot{H_g}|_{A,B} & 0 \\ 0 & 0\end{bmatrix}\\
&=
\begin{bmatrix}
{H_h}|_{B}-\varphi_g &  {H_h}|_{B,D}\\
{{H_h}|_{B,D}}^T &  {H_h}|_{D}
\end{bmatrix}\,.
\end{align*}
Evaluating $M$ at $(x_A^0, x_B^0, x_D^0)$ and using \eqref{eq: gamma'' equals phi_g}, from the previous chain of identities we get
\begin{align*}
M(x_A^0, x_B^0, x_D^0) & = 
\begin{bmatrix}
{H_h}|_{B}(x_B^0)-\varphi_g(x_A^0, x_B^0) &  {H_h}|_{B,D}(x_B^0, x_D^0)\\
{{H_h}|_{B,D}}^T(x_B^0, x_D^0) &  {H_h}|_{D}(x_D^0)
\end{bmatrix}\\
& =
\begin{bmatrix}
{H_h}|_{B}(x_B^0)-\tilde{\gamma}(x_B^0) &  {H_h}|_{B,D}(x_B^0, x_D^0)\\
{{H_h}|_{B,D}}^T(x_B^0, x_D^0) &  {H_h}|_{D}(x_D^0)
\end{bmatrix}\\
& = {H_{\hat{h}}}|_{B\cup D}(x_B^0, x_D^0)\,,
\end{align*}
in turn ${H_{\hat{h}}}|_{B\cup D}(x_B^0, x_D^0)$ is invertible by \eqref{eq: matrix determinant formula}. This shows that ${H_{\hat{h}}}|_{B\cup D}$ is invertible on $X$.
\end{proof}

\end{document}